\documentclass[11pt,a4paper,reqno]{amsart}
\usepackage[applemac]{inputenc}
\usepackage[T1]{fontenc}
\usepackage{amsmath}
\usepackage{amsthm}
\usepackage{amsfonts}
\usepackage{amssymb}
\usepackage{graphicx}
\usepackage{amsbsy}
\usepackage{mathrsfs}
\usepackage{bbm}
\newtheorem{theorem}{Theorem}[section]
\newtheorem{lemma}[theorem]{Lemma}

\newtheorem{proposition}[theorem]{Proposition}
\newtheorem{corollary}[theorem]{Corollary}
\newtheorem{definition}[theorem]{Definition}
\theoremstyle{remark}
\newtheorem{remark}[theorem]{\it \bf{Remark}\/}

\numberwithin{equation}{section}
\catcode`@=11
\def\section{\@startsection{section}{1}%
  \z@{1.5\linespacing\@plus\linespacing}{.5\linespacing}%
  {\normalfont\bfseries\large\centering}}
\catcode`@=12
\newcommand{\be}{\begin{equation}}
\newcommand{\ee}{\end{equation}}
\newcommand{\bea}{\begin{eqnarray}}
\newcommand{\eea}{\end{eqnarray}}
\newcommand{\bee}{\begin{eqnarray*}}
\newcommand{\eee}{\end{eqnarray*}}

\def\pa{\partial}

\def\RR{\mathbb{R}}

\def\fref#1{{\rm (\ref{#1})}}

\catcode`@=11
\def\supess{\mathop{\operator@font Sup\,ess}}
\catcode`@=12

\def\RR{\mathbb{R}}

\def\e{\varepsilon}

\def\bar#1{{\overline #1}}
\def\fref#1{{\rm (\ref{#1})}}

\def\R2+{\RR ^2_+}

\def\pa{\partial}

\def\lim{\mathop{\rm lim}}

\def\l{\lambda}

\def\pa{\partial}

\def\pa{\partial}

\def\matchal{\mathcal}

\def\para{\parallel}
\def\ba{\begin{array}}
\def\ea{\end{array}}
\def\la{\label}
\def\bal{\begin{align*}}
\def\eal{\end{align*}}
\def\non{\nonumber}

\title[]{Dynamics near the ground state for the energy critical nonlinear heat equation in large dimensions}
\author[C.Collot]{Charles Collot}
\address{Laboratoire J.A. Dieudonn\'e, Universit\'e de Nice-Sophia Antipolis, France}
\email{ccollot@unice.fr}
\author[F.Merle]{Frank Merle}
\address{Laboratoire Laga, Universit\'e de Cergy Pontoise, France\\
\& IHES, Bures-sur-Yvette, France}
\email{merle@math.ucergy.fr}
\author[P.Rapha\"el]{Pierre Rapha\"el}
\address{Laboratoire J.A. Dieudonn\'e, Universit\'e de Nice-Sophia Antipolis, France}
\email{praphael@unice.fr}


\begin{document}

\begin{abstract} 
We consider the energy critical semilinear heat equation $$\pa_tu=\Delta u+|u|^{\frac{4}{d-2}}u, \ \ x\in \mathbb R^d$$ and give a complete classification of the flow near the ground state solitary wave $$Q(x)=\frac{1}{\left( 1+\frac{|x|^2}{d(d-2)}\right)^{\frac{d-2}{2}}}$$ in dimension $d\ge 7$, in the energy critical topology and without radial symmetry assumption. Given an initial data $Q+\e_0$ with $\|\nabla \e_0\|_{L^2}\ll 1$, the solution either blows up in the ODE type I regime, or dissipates, and these two open sets are separated by a codimension one set of solutions asymptotically attracted by the solitary wave. In particular, non self similar type II blow up is ruled out in dimension $d\ge 7$ near the solitary wave even though it is known to occur in smaller dimensions \cite{Schw}. Our proof is based on sole energy estimates deeply connected to \cite{MMR1} and draws a route map for the classification of the flow near the solitary wave in the energy critical setting. A by-product of our method is the classification of minimal elements around $Q$ belonging to the unstable manifold.
 
\end{abstract}

\maketitle


\section{Introduction}



\subsection{Setting of the problem}


We deal in this paper with the the energy critical semilinear heat equation 
\be \label{eq:NLH} (NLH)
\left\{ \begin{array}{l l} \partial_t u=\Delta u+|u|^{p-1}u, \ \ p=\frac{d+2}{d-2}\\
u(0,x)=u_0(x)\in \mathbb R
\end{array}
\right., \ \ (t,x)\in \mathbb R\times \mathbb R^d
\ee
in dimensions $d\ge 3$ and for an initial data in the energy space $$u_0\in \dot{H}^1(\mathbb R^d).$$ From standard parabolic regularity arguments, the Cauchy problem is well posed in the energy space and for all $u_0\in \dot{H}^1(\mathbb R^d)$, there exists a maximal solution $u\in \mathcal C((0,T),\dot{H}^1\cap \dot H^2\cap L^{\infty})$ with the blow up criterion 
\be
\label{beibeibvei}
T<+\infty\ \ \mbox{implies}\ \ \lim_{t\to T}\|u(t)\|_{\dot{H}^2}=+\infty \ \text{and} \ \lim_{t\to T}\|u(t)\|_{L^{\infty}}=+\infty.
\ee The dissipation of the full energy 
\be \la{eq:E}
\frac{d}{dt}E(u)=-\int_{\mathbb R^d} u_t^2 \leq 0, \ \ E(u)=\frac 12\int_{\mathbb R^d} |\nabla u|^2 -\frac{d-2}{2d}\int_{\mathbb R^d} |u|^{\frac{2d}{d-2}} 
\ee
and the scaling symmetry $$u_\l(t,x)=\l^{\frac{d-2}{2}}u(\l^2 t,\l x), \ \ E(u_\l)=E(u)$$ reflect the energy critical nature of the problem. The Aubin-Talenti ground state solitary wave 
\be
\label{eq:def Q}
Q(x)=\frac{1}{\left( 1+\frac{|x|^2}{d(d-2)}\right)^{\frac{d-2}{2}}}
\ee 
attains the best constant in the critical Sobolev estimate $\|u\|_{L^{\frac{2d}{d-2}}}\lesssim \|u\|_{\dot{H}^1}$, \cite{Au,Ta}, and is the unique \cite{Gi} up to symmetry positive decaying at infinity solution to 
\be
\label{eqellipticq}
\Delta Q+Q^{\frac{d+2}{d-2}}=0.
\ee The aim of this paper is to obtain a complete description of the flow near the ground state $Q$ in the energy critical topology in large dimensions $d\geq 7$.


\subsection{Type I and type II blow up}


Let us recall some expected dynamical behaviours near the ground state. A first well known scenario is the dissipative behavior $$\lim_{t\to +\infty}\|\nabla u(t)\|_{L^2}= 0$$ which follows from a simple upper solution argument coupled with the dissipation of the Dirichlet energy.\\
 A second well know scenario \cite{Giga,Gi1,Gi2,Gi3,Gi4} is the possibility of type I blow up \be \la{intro:def type I}\|u(t)\|_{L^{\infty}}\sim \frac{1}{(T-t)^{\frac{1}{p-1}}}, \ \ p=\frac{d+2}{d-2}.\ee More precisely, the  solution matches to leading order the ODE constant blow up solution \be \la{intro:def kappa} u(t,x)=\frac{\kappa_p}{(T-t)^{\frac{1}{p-1}}}, \ \ \kappa_p=\left(\frac{1}{p-1}\right)^{\frac 1{p-1}}\ee and the blow up profile is given by the constant self similar solution, \cite{Gi1}. These solutions do not see the energy critical structure of the problem and exist in fact for all $p>1$ in \eqref{eq:NLH}, and in particular lead to the focusing of the Dirichlet energy 
\be
\label{otherlimit}
\lim_{t\uparrow T}\|\nabla u(t)\|_{L^2}=+\infty.
\ee
The ODE type I blow up dynamic is stable in the energy critical topology, see \cite{Fe} and Section \ref{sec:type I}. In the subcritical range $1<p<\frac{d+2}{d-2}$, all blow-ups are ODE type I blow-ups \cite{Gi2,Gi4}. \\
A third expected scenario is the possibility of type II blow up, that is solutions which concentrate in finite time $T<+\infty$ but in the energy critical topology $$\limsup_{t\uparrow T}\|\nabla u(t)\|_{L^2}<+\infty.$$ The existence of such type of solutions, both in critical and super critical settings, has attracted a considerable attention for the past twenty years, both in the parabolic and dispersive settings, in connection with classical conjectures in mathematical physics, see in particular \cite{NT2, KST, KST2, MR1,MR4,MR5,MMR1,RaphRod,MRRod,MRRsur,Co,RSc1,RSc2}. After the pioneering works in the super critical setting \cite{HVsur, MaM1,MaM2, Mizo} and which relies on maximum principle based Lyapounov functionals, the first rigorous result of existence of type II blow up solutions in the parabolic energy critical setting \eqref{eq:NLH} is given in \cite{Schw} in dimension $d=4$ and follows the route map for the construction of such blow up bubbles based on sole energy estimates designed in \cite{RaphRod,MRRod,RSc1,RSc2}. In the pioneering work \cite{Fi}, the authors have predicted formally the existence of such type II bubbles near $Q$ in dimension $d\leq 6$, and the case $d\geq 7$ was left open even at the formel level. Note that in another setting like the energy critical harmonic heat flow of surfaces, type II blow up has been proved in the radially symmetric setting for small homotopy number \cite{RSc1,RSc2} and ruled out for large homotopy number \cite{NT2} which shows the dramatic\footnote{but so far poorly understood} influence of numerology in these problems.\\

\noindent Long time dynamics for global solutions are also closely related to the flow near stationary states. For the energy critical semilinear heat equation \fref{eq:NLH}, any radial global solution in the energy space decomposes as $t\rightarrow +\infty$ as a sum of decoupled ground states in the energy space $\dot H^1$, \cite{MaM3}.

\subsection{Statement of the result}

The main result of this paper is the following classification of the flow near the ground state in the energy critical topology in large dimensions $d\ge 7$.

\begin{theorem}[Classification of the flow near $Q$ for $d\ge 7$]
 \label{th:main}
Let $d\geq 7$. There exists $0<\eta\ll1$ such that the following holds. Let $u_0\in \dot H^1(\mathbb R^d)$ with
$$\|u_0-Q\|_{{\dot H}^1}<\eta,$$ then the corresponding solution $u\in \matchal C((0,T),\dot{H}^1\cap\dot{H}^2)$ to \eqref{eq:NLH} follows one of the three regimes:\\
\noindent{\em 1.( Soliton)}: the solution is global and asymptotically attracted by a solitary wave
$$
\exists(\lambda_{\infty},z_\infty)\in \mathbb R^*_+\times \mathbb R^d \ \mbox{such that}\ \ \lim_{t\to +\infty}\left\|u(t,\cdot)-\frac{1}{\l_\infty^{\frac{d-2}{2}}}Q\left(\frac{\cdot-z_\infty}{\lambda_\infty}\right)\right\|_{\dot{H}^1} =0.
$$
Moroever, $$|\l_\infty-1|+ |z_{\infty}|\to 0\ \ \mbox{as}\ \ \eta\to 0.$$
\noindent{\em 2. (Dissipation)}: the solution is  global and dissipates
$$
\lim_{t\to +\infty}\|u(t,\cdot)\|_{\dot{H}^1}=0, \ \ \lim_{t\to +\infty}\|u(t,\cdot)\|_{L^{\infty}}=0.$$
\noindent{\em 3. (Type I blow up)}: the solution blows up in finite time $0<T<+\infty$ in the ODE type I self similar blow up regime near the singularity $$
\|u(t,\cdot)\|_{L^{\infty}}\sim (T-t)^{-\frac{d-2}{4}}.
$$
There exist solutions associated to each scenario. Moreover, the scenario (Dissipation) and (Type I blow up) are stable in the energy topology.
\end{theorem}

In particular, type II blow up is {\it ruled out} in dimension $d\ge 7$ near the solitary wave. A by-product of our method is the following Liouville theorem which classifies global in time backwards solutions near the solitary wave.

\begin{theorem}[Liouville theorem for minimal solutions around $Q$] \label{th:liouville} Let $d\geq 7$.\\
\noindent{\em 1. Existence of backwards minimal elements}: There exist two strictly positive, $\mathcal C^{\infty}$ radial solutions of \fref{eq:NLH}, $Q^+$ and $Q^-$, defined on $(-\infty,0]\times \mathbb R^d$, which are minimal backwards in time $$\lim_{t\to -\infty}\|Q^{\pm}-Q\|_{\dot{H}^1}=0$$ and have the following forward behaviour: $Q^+$  explodes according to type I blow up at some finite later time, and $Q^-$ is global and dissipates $Q^-\rightarrow 0$ as $t\rightarrow +\infty$ in $\dot H^1(\mathbb R^d)$.\\
\noindent{\em 2. Rigidity}:  Moreover, there exists $0<\delta\ll 1$ such that if $u$ is a solution of \fref{eq:NLH} on $(-\infty,0]$ such that:
$$
\underset{t\leq 0}{\text{sup}} \underset{\lambda>0, \ z\in \mathbb R^d}{\text{inf}} \parallel u(t)-Q_{z,\lambda}\parallel_{\dot H^1}\leq \delta 
$$
then $u=Q^\pm$ or $u=Q$ up to the symmetries of the flow.
\end{theorem}

\vspace{0.5 cm}
\noindent{\it Comments on the result.}\\

\noindent{\em 1. Connection to other classification theorems.} Theorem \ref{th:main} is a classification theorem of the flow near the solitary in a situation where blow up is possible. This kind of questions has attracted a considerable attention in particular in the dispersive setting for the Non Linear Schr\"odinger equation both in the mass critical \cite{MR4,MR5,MRSz} and mass super critical \cite{NakSch} cases, the mass critical KdV equation \cite{MMR1,MMR2} and the energy critical wave equation \cite{DKM}, and also in the radial parabolic setting in connection with the large homotopy harmonic heat flow problem \cite{NT2}. Depending on the situation and the underlying algebra of the problem, the classification is more or less complete, and Theorem \ref{th:main} gives the first complete result in the non radial energy critical setting. Note that the fact that the regime (Soliton) near $Q$ is attained exactly one a codimension smooth manifold of initial data could be proved as in \cite{MMN}.\\

\noindent{\em 2. The energy method.} The proof of Theorem \ref{th:main} turns out to be deeply connected to the proof of the analogous result in \cite{MMR1,MMR2} for the critical gKdV equation and based on energy estimates and a key no return type lemma for the flow near Q inherited from the modulation equations, see the strategy of the proof of below, see also \cite{NakSch} for a related approach. The energy estimates are based on higher Sobolev norms constructed on the linearized operator close to $Q$. These norms were introduced in small dimensions in \cite{RaphRod,MRRod,RSc1,RSc2} to produce the type II blow up bubbles, and a key point of our analysis is that in dimension $d\ge 7$, the type II instabilities disappear and these norms become coercive thanks to suitable weighted Hardy type inequalities\footnote{see in particular Appendix \ref{sec:coercivite} and \eqref{coercivite}.}. This validates the functional framework of \cite{RaphRod,MRRod} and draws a route map for the derivation of related results in the dispersive setting. The non radial nature of the analysis is also an essential feature, see \cite{Co2} for further extension in the energy supercritical setting.\\

\noindent{\em 3. Minimal elements.} An essential step of the proof is the derivation of the existence and uniqueness of suitable minimal elements which connect the dynamics at $- \infty$ near $Q$ to the forward in time (Exit) regime of either dissipation or finite time type I blow up away from $Q$, Theorem \ref{th:liouville}. Such problems were addressed again in the parabolic and dispersive settings in \cite{Me9,RSmin,MMR2} and more specifically in the energy critical setting in \cite{Me,DM,DM2}. The stability of these regimes is directly connected to the long time behavior of the solution once exiting a neighborhood of the soliton. This question also relates to the more general problem of classification of minimal elements of the flow which through the general concentration compactness decomposition of the solution has become a powerful tool for the description of the flow both near the solitary wave and for large data \cite{DKM}. Note that in the parabolic setting and for radial data, there is an even more powerful tool based on the maximum principle and the intersection number Lyapounov functionals \cite{MaM1,MaM2,Mizo} which should lead to further large data radial data results, see also \cite{MeZa2,Me9} for the subcritical heat equation.\\

The main open problem after this work is the complete classification of the flow near the solitary wave in the presence of type II blow up for $3\leq d\leq 6$.

\subsection*{Acknowledgements}  F.M. would like to thanks H. Matano for stimulating discussion about this work. F.M. is partly supported by the ERC advanced grant 291214 BLOWDISOL. P.R. and C.C are supported by the ERC-2014-CoG 646650 SingWave. P.R. is a junior member of the Institut Universitaire de France.


\subsection*{Notations}

We collect the main notations used throughout the paper.\\

\noindent\emph{General notations.} We will use a generic notation for the constants, $C$, whose value can change from line to line but just depends on $d$ and not on the other variables used in the paper. The notations $a \lesssim b$ (respectively $a \gtrsim b$) means $a\leq Cb$ (respectively $b \geq Ca$) for such a constant $C$. The notation $a=O(b)$ then means $|a|\lesssim b$. We employ the Kronecker $\delta$ notation:
$$
\delta_{a,b}= 1 \ \text{if} \ a=b, \ 0 \ \text{otherwise}.
$$

\noindent \emph{PDE notations.} We let the heat kernel be:
\be \label{eq:def Kt}
K_t(x):= \frac{1}{(4\pi t)^{\frac d 2}}e^{-\frac{|x|^2}{4t}}
\ee
We recall that the elliptic equation
\be
\label{eqnonline}
\Delta u + |u|^{p-1}=0,
\ee
admits a unique \cite{Gi} up to symmetries strictly positive $\dot H^1$ solution, equivalently \be \label{eq:elliptique}
u=Q_{z,\lambda}, \ \lambda>0, \ z\in \mathbb R^d.
\ee
where the radial soliton $Q$ is explicitely given by \fref{eq:def Q}. The linearized operator for \fref{eq:NLH} close to $Q$ is the Schr\"odinger operator:
\be \label{eq:def H}
H:=-\Delta -pQ^{p-1}=-\Delta +V
\ee
for the potential:
$$
V:=-pQ^{p-1}=-\frac{d+2}{d-2} \frac{1}{\left( 1+\frac{|x|^2}{d(d-2)}\right)^2}.
$$
The operator $H$ has only one negative eigenvalue $-e_0$, $e_0>0$ with multiplicity one associated to a non negative function $\mathcal Y$ that decays exponentially fast (and also its derivatives):
$$
H \mathcal Y=-e_0\mathcal Y,\ \ \mathcal Y>0, \ \ \int_{\mathbb R^d}\mathcal Y=1,
$$
see Proposition \ref{pr:H} for more details. We denote the nonlinearity by:
$$
f(u):=|u|^{p-1}u .
$$

\noindent \emph{Invariances.}
For $\lambda>0$ and $u:\mathbb{R}^d\rightarrow \mathbb R$, we define the rescaling
\be \label{intro:eq:def ulambda}
u_{\lambda} \ : \ x \ \mapsto \ \frac{1}{\lambda^{\frac{d-2}{2}}}u\left( \frac{x}{\lambda} \right).
\ee
The infinitesimal generator of this tranformation is $$
\Lambda u:=- \frac{\partial}{\partial \lambda}(u_{\lambda})_{|\lambda=1}= \frac{d-2}{2}u+x.\nabla u .
$$
Given a point $z\in \mathbb R^d$ and a function $u:\mathbb{R}^d\rightarrow \mathbb R$, we define the translation of vector $z$ of $u$ as:
\be \label{intro:eq:def tauzu}
\tau_z u \ : \ x \ \mapsto \ u(x-z).
\ee
The infinitesimal generator of this semi group is:
$$
\left[ \frac{\partial }{\partial z} (\tau_z u)\right]_{|z=0}= - \nabla u.
$$
One has for any $\lambda>0$ and $z\in \mathbb R^d$ that $\tau_z(u_{\lambda})=\left(\tau_{\frac{z}{\lambda}} u \right)_{\lambda}$. For a function $u:\mathbb{R}^d\rightarrow \mathbb R$ we use the  condensed notation:
$$
u_{z,\lambda}:=\tau_z (u_{\lambda})=x\mapsto \frac{1}{\lambda^{\frac{d-2}{2}}}u\left( \frac{x-z}{\lambda}\right).
$$
For $\lambda>0$, the original space variable will be referred to as $x\in \mathbb R^d$, the renormalized space variable will be referred to as $y\in \mathbb R^d$, and they are related by $$y=\frac{x-z}{\lambda}.$$ The manifold of ground states, being the orbit of $Q$ under the action of the two previous groups of symmetries, is:
$$
\mathcal M:=\{Q_{z,\lambda}, \ z\in \mathbb R^d, \ \lambda>0 \}.
$$
One has the following integration by parts and commutator formulas for smooth well localized functions :
\bea
 \label{eq:ipp}
&&\int (\Lambda u) v+\int u\Lambda v +2 \int uv =0,\\
\label{eq:commutateur}
&&H\Lambda =\Lambda H +2H-(2V+x.\nabla V),\\
&& \label{eq:ipp2}
H\nabla =\nabla H -\nabla V.
\eea

\noindent \emph{Functional spaces.} We will use the standard notation $H^s$, $\dot H^s$ and $H^s_{\text{loc}}$ for inhomogeneous and homogeneous Sobolev spaces, and for the space of distributions that are locally in $H^s$. The distance between an element $u\in \dot H^1$ and a subset $X\subset \dot H^1$ is denoted by:
$$
d(u,X):=\underset{v\in X}{\text{inf}} \parallel u-v\parallel_{\dot H^1}.
$$
For $u,v\in L^2(\mathbb R^d)$ real valued, the standard scalar product is
$$
\langle u,v \rangle= \int_{\mathbb R^d} uvdx
$$
and the orthogonality $u\perp v$ means $\langle u,v\rangle =0$. We extend these notations whenever $uv\in L^1(\mathbb R^d)$. For a subspace $X\subset \mathbb R^d$ and $0<\alpha<1$, the space $C^{0,\alpha}(X)$ is the space of H\"older $\alpha$-continuous functions on $X$. For $\alpha\in \mathbb N^d$, we will use the notations for the derivatives:
$$
\partial^{\alpha}f=\partial_{x_1}^{\alpha_i}...\partial_{x_d}^{\alpha_d}f, \ \ \nabla f=(\partial_{x_i}f)_{1\leq i \leq d}, \ \ \nabla^2f=(\partial_{x_ix_j}f)_{1\leq i,j\leq d} .
$$


\subsection{Strategy of the proof}

We now display the route map for the proofs of Theorem \ref{th:main} and \ref{th:liouville}, and stress the spectacular analogy with  the scheme of proof in \cite{MMR1}.\\

\noindent{\bf step 1} Coercivity of  the linearized operator and dissipation.  The dynamics near $Q$ is driven at the linear level by the operator $H=-\Delta -pQ^{p-1}$. This Schr\"odinger operator admits only one negative eigenvalue $-e_0$, $e_0>0$, with multiplicity $1$ associated to a positive, radially symmetric and exponentially well-localized profile $\mathcal Y$. It admits $d+1$ zeros that are given by the symmetries of the flow: $\Lambda Q,\partial_{x_1}Q,...,\partial_{x_d}Q$. A fundamental observation is that under the natural orthogonality conditions $\varepsilon \in \text{Span}(\mathcal Y,\Lambda Q,\partial_{x_1}Q,...,\partial_{x_d}Q)^{\perp}$, this operator behaves like the Laplace operator for the first three homogeneous Sobolev norms build on its iterates:
\be \la{coercivite}
\int_{\mathbb R^d} \varepsilon H^{i} \varepsilon dx \approx \para \varepsilon \para_{\dot H^i(\mathbb R^d)}^2 \ \ i=1,2,3
\ee
We emphasize that this is an essential consequence of the large dimension $d\ge 7$ assumption, and that this no longer holds for lower dimensions. This designs as in \cite{RaphRod, RSc1,RSc2} a natural set of norms to quantify the dissipative properties of the flow near $Q$.\\

\noindent{\bf step 2} Modulation, energy estimates and control of the scale near $Q$. We introduce a classical nonlinear modulated decomposition of  the solution close to the manifold of solitons $\mathcal M$:
$$
u=(Q+a\mathcal Y+\varepsilon)_{z,\lambda}, \ \ \varepsilon \in  \text{Span}(\mathcal Y,\Lambda Q,\partial_{x_1}Q,...,\partial_{x_d}Q)^{\perp}
$$
introducing the central point $z(t)$, the scale $\lambda (t)$, the projection of the perturbation on the instability direction $a(t)$ and the infinite dimensional remainder $\varepsilon (t)$. One has in particular:
\be \la{h1}
|a|+\para \varepsilon \para_{\dot H^1}\approx d(u,\mathcal M), \ \ |E(u)-E(Q)|\lesssim |a|^2+\para \varepsilon \para_{\dot H^1}^2.
\ee
Introducing the renormalized time $\frac{ds}{dt}=\frac{1}{\lambda^{2}}$, the evolution of $\e$ is to leading order:
$$
(a_s-e_0a)\mathcal Y+\varepsilon_s-\frac{\lambda_s}{\lambda}\Lambda Q-\frac{z_s}{\lambda}+H\varepsilon+NL=\rm{lot}
$$
where $NL$ stands for the purely nonlinear term. The geometrical parameters are computed through the choice of orthononality conditions:
\be \la{as}
a_s=e_0a+O(a^2+\para \varepsilon \para^2_{\text{loc}}), \ \ \left| \frac{d}{ds}(\text{log}(\lambda))\right|+\left| \frac{z_s}{\lambda} \right|\lesssim a^2+\para \varepsilon \para_{\dot H^2}^2,
\ee
and the infinite dimensional part $\e$ is controlled using an energy estimate at {\it both the $\dot{H}^1$ and $\dot{H}^2$ level} adapted to the linearized operator using \eqref{coercivite}:
\be \la{e1}
\frac{d}{ds}\left(\para \varepsilon \para_{\dot H^1}^2\right) \lesssim -\para \varepsilon \para_{\dot H^2}^2+a^2, \ \ \frac{d}{ds}\left( \frac{1}{\lambda^2} \para \varepsilon \para_{\dot H^2}^2 \right) \lesssim \frac{1}{\lambda^2}\left(-\para \varepsilon \para_{\dot H^3}^2+a^4+\para \varepsilon \para_{\dot H^2}^4\right).
\ee
These informations may be coupled to the critical control of the global energy:
\be \la{e}
\frac{d}{ds}E(u) \lesssim -\left(a^2+\para \varepsilon \para_{\dot H^2}^2\right).
\ee
and the fundamental observation is that as long as the solution stays near $\mathcal M$, its energy stays bounded from \fref{h1} and is a Lyapunov functional whose time evolution is given by \fref{e}. This implies the following space-time estimate for the perturbation:
\be \la{es}
\int_0^{s}(a^2+\para \varepsilon \para_{\dot H^2}^2)d\sigma\lesssim d(u,\mathcal M) \ll 1,
\ee
which injected into the modulation equation \fref{as} implies that the scale and the central point cannot move:
$$
|\lambda (t)-\lambda (0)|+|z(t)-z(0)|\lesssim d(u,\mathcal M)\ll 1.
$$
This already rules out type II blow near $Q$ in dimension $d\ge 7$.\\

\noindent{\bf step 3} Existence of minimal backwards elements. The bounds \fref{as} and \fref{e1} imply that for solutions close to $Q$, the projection $a$ onto the unstable mode displays a linear exponentially growing behaviour up to quadratic lower order errors. We then construct $Q^+$ and $Q^-$ as the generators of the instable manifold as $t\to -\infty$, i.e. solutions defined on $(-\infty,0]$ such that $a\approx \pm e^{e_0 t}$ and the stable perturbation is quadratic $\para \varepsilon \para_{\rm loc}\lesssim e^{2e_0t}$ as $t\rightarrow -\infty$. Their construction, similar to \cite{DM2}, relies on energy estimates and a compactness argument using parabolic regularization. A comparison principle, as $\mathcal Y$ is positive, gives $Q^+>Q$ and $0<Q^-<Q$. This forces $Q^+$ to blow up with type I blow up forward in time from a convexity argument and its positivity using Theorem 1.7 in \cite{MaM2}. Using parabolic regularity, we conclude that $Q^-$ converges to some stationary state as $t\rightarrow +\infty$, and this limit has to be zero using the classification of  positive solutions to the stationary elliptic equation \eqref{eqellipticq}, implying that $Q^-$ dissipates to $0$.\\

\noindent{\bf step 4} Uniqueness of minimal backwards elements. We now claim following \cite{Me9,DM2,RSmin,MMR2} the Liouville Theorem \ref{th:liouville} of classification of minimal backwards solutions. Indeed, let $u$ be a solution staying close to $Q$ as $t\rightarrow -\infty$, then the space-time integrability \fref{es} implies that $a\rightarrow 0$ and $\varepsilon \rightarrow 0$ as $t\rightarrow -\infty$. However, the dissipation \fref{e1} of the stable part $\varepsilon$ prevents it to go to zero if the size of $a$ is not comparable with $\varepsilon$. From the modulation equation \fref{as} for $a$ and the energy bounds \fref{e}, this forces the solution to have an exponential behavior: $a\approx e^{e_0t}$ and $\para \varepsilon \para_{\rm loc}\lesssim e^{2e_0t}$ as $t\rightarrow -\infty$, similar to the one of $Q^+$ and $Q^-$. This allows us to run a contraction argument in the space of solutions decaying exponentially as $t\to -\infty$, proving that either $u=Q^+$ or $u=Q^-$.\\

\noindent{\bf step 5} Stability of minimal solutions and (Exit) regimes. $Q_{\pm}$ design two mechanisms to exit a neighborhood of the solitary wave. We claim that these (Exit) dynamics are stable in the energy space. This is easily proven for the dissipative behaviour $Q^-$. For $Q^+$, this follows more generally from the stability of Type I blow up which is proven in \cite{Fe} in the subcritical regime $1<p<\frac{d+2}{d-2}$, and the arguments adapt to the energy critical case, Section \ref{sec:type I}.\\

\noindent{\bf step 6} Trapping and no return dynamics. We now turn to the heart of the analysis which is the ejection lemma. Let us introduce two characteristic smallness constants $0< \delta \ll \alpha\ll 1$ which measure respectively the distance to $Q$ of the initial data, and a universal tube of radius $\alpha$ around $\mathcal M$. We assume that the data $u(0)$ is at distance $\delta^4$ of $Q$ in the $\dot H^1$ topology, and we define the instable time $T_{\text{ins}}$ as the first time for which the instability is non negligible compared to the stable perturbation:
$$
|a(T_{\text{ins}})|\gtrsim \para \varepsilon (T_{\text{ins}}) \para_{\dot H^2}^2 .
$$
The energy estimates \fref{e}, the modulation equations \fref{as} and the space-time integrability \fref{es} ensure that the solution stays at distance $\delta^4$ of $Q$ in the $\dot H^1$ topology in $[0,T_{\text{ins}}]$. If $T_{\text{ins}}=+\infty$, then the stable part of the perturbation dominates for all times, the solution is trapped near $Q$ and is in fact asymptotically attracted to a possibility slightly shifted $Q$, this is the (Soliton) codimension one instable manifold. A contrario, if $T_{\text{ins}}<+\infty$, then the solution {\it enters a new regime} where $a$ will start to grow until the transition time $T_{\text{trans}}$ at which the stable perturbation is of quadratic size: 
$$
a^2(T_{\text{trans}})\gtrsim \para \varepsilon (T_{\text{trans}}) \para_{\dot H^2} .
$$
At this time, $a$ is approximately of size $\delta$ and $\varepsilon$ of size $\delta^2$. The solution then enters an exponential regime: $a\approx \pm \delta e^{e_0 (t-T_{\text{trans}})}$ and $\para \varepsilon \para\lesssim \delta^2e^{2e_0(t-T_{\text{trans})}}$ which will drive the solution out of the ball of size $\alpha$ around $Q$ at an exit time $T_{\text{exit}}$. In $[T_{\text{trans}},T_{\text{exit}}]$, we may compare the solution to the minimal elements $Q_{\pm}$ depending on the sign of $a(T_{\text{trans}})$ thanks to the exponential bounds and obtain:
$$
\para Q^{\pm}(T_{\text{exit}})-u(T_{\text{exit}})\para_{\dot H^1}\lesssim \delta .
$$
Since $Q^{\pm}(T_{\text{exit}})$ is by definition of size $\alpha$ which is universal, the solution is trapped at $T_{\text{exit}}$ in a neighborhood of $Q_\pm$ of size $\delta$ arbitrarily small, and hence step 5 ensures that the solution will either blow up type I or dissipate to zero forward in time  after $T_{\rm exit}$.\\

The paper is organized as follows. In Section \ref{sec:Q} we recall the standard results on the semilinear heat equation and on $Q$, leading to a preliminary study of solutions near $Q$. The local well-posedness is addressed in Proposition \ref{pr:cauchy}. In Proposition \ref{pr:H} we describe the spectral structure of the linearized operator. In particular, for functions orthogonal to its instable eigenfunction and to its kernel, it displays some coercivity properties stated in Lemma \ref{lem:coercivite}. The nonlinear decomposition Lemma \ref{lem:decomposition} then allows to decompose in a suitable way solutions around $Q$. Under this decomposition, the adapted variables are defined in Definition \ref{def:trapped}, the variation of the energy is studied in Lemma \ref{lem:variation energie}, the modulation equations are established in Lemma \ref{lem:modulation} and the energy bounds for the remainder on the infinite dimensional subspace are stated in Lemma \ref{lem:methode energie}. A direct consequence is the non-degeneracy of the scale and the central point, subject of the Lemma \ref{lem:lambda} ending the section. In the following one, Section \ref{sec:min}, we first construct the minimal solutions in Proposition \ref{pr:minimal} and then give the proof of the Liouville-type theorem for minimal solutions \ref{th:liouville}. Finally, in Section \ref{sec:th} we give the proof of the classification Theorem \ref{th:main}. In Lemma \ref{lem:caracterisation Tins} we characterize the instability time. If this time never happens, the solution is proved to dissipate to a soliton in Lemma \ref{lem:soliton}. If it happens, then it blows up with type I blow up or dissipate according to Lemma \ref{lem:exit}. Eventually, Section \ref{sec:type I} contains the proof of the stability of type I blow up.\\

The Appendix is organized as follows. First in Section \ref{sec:H} we study the kernel of the linearized operator in Lemma \ref{lem:zeros}. Then in Section \ref{sec:coercivite} we give the proof of the coercivity Lemma \ref{lem:coercivite}. In Section \ref{sec:decomposition} we prove the decomposition Lemma \ref{lem:decomposition}. Next, in Section \ref{sec:NL}, we state some useful estimate on the purely nonlinear estimate as $1<p<2$ in Lemma \ref{lem:NL}. In Section \ref{sec:para} we give certain parabolic type results, especially Lemma \ref{para:lem:max} that allows to propagate exponential bounds.


\section{Estimates for solutions trapped near $\mathcal M$} \la{sec:Q}


This section is devoted of the study of solutions which are globally trapped near the solitary wave. The heart of the proof is the coupling of the dissipative properties of the flow with coercivity estimates for the linearized operator $H$ and energy bounds which imply the control of the scaling parameter, Lemma \ref{lem:lambda}, and hence the impossibility of type II blow up near $Q$.


\subsection{Cauchy theory}


We first recall the standard Cauchy theory which takes care of the lack of differentiability of the nonlinearity in high dimensions. The Cauchy problem for \fref{eq:NLH} is well posed in the critical Lebesgue space $L^{\frac{2d}{d-2}}$, see \cite{Br} \cite{We} \cite{La}. A direct adaptation of the arguments used there implies that it is also well posed in $\dot H^1$. \fref{eq:NLH} being a parabolic evolution equation, it also possesses a regularizing effect. Namely, any solution starting from a singular initial datum will have an instantaneous gain of regularity, directly linked to the regularity of the nonlinearity which is only $C^1$ here as $1<p<2$. The proof of the following proposition is classical using for example the space time bounds of Lemma \ref{lem:strichartz}, and the details are left to the reader.

\begin{proposition}[Local well posedness of the energy critical semilinear heat equation in $\dot H^1$ and regularizing effects] \label{pr:cauchy}

Let $d\geq 7$. For any $u_0\in \dot H^1(\mathbb R^d)$ there exists $T(u_0)>0$ and a weak solution $u\in \mathcal C([0,T(u_0)),\dot H^1(\mathbb R^d))$ of \fref{eq:NLH}. In addition the following regularizing effects hold:
\begin{itemize}
\item[(i)] $u\in C^{(\frac 3 2,3)}((0,T_{u_0})\times \mathbb R^d)$, $u$ is a classical solution of \fref{eq:NLH} on $(0,T_{u_0})\times \mathbb R^d$. 
\item[(ii)] $u\in C((0,T(u_0)),W^{3,\infty}(\mathbb R^d))$.
\item[(iii)] $u\in C((0,T(u_0)),\dot H^3(\mathbb R^d))$, $u\in C^1((0,T(u_0)),\dot H^1 )$.
\end{itemize}
For any $0<t_1<t_2<T_{u_0}$ the solution mapping is continuous from $\dot H^1$ into $C^{(\frac 3 2,3)}([t_1,t_2]\times \mathbb R^d)$, $C([t_1,t_2],W^{3,\infty})$, $C((t_1,t_2),\dot H^1\cap \dot H^3)$ and $C^1((t_1,t_2),\dot H^1 )$ at $u_0$. 
\end{proposition}

Here $u\in C^{(\frac 3 2,3)}((0,T_{u_0})\times \mathbb R^d)$ means that $u$ is in the H\"older space $C^{\frac 3 2}((0,T_{u_0})\times \mathbb R^d)$ and that $u$ is three time differentiable with respect to the space variable $x$. The maximal time of existence of the solution $u$ will then be denoted by $T_{u_0}$. From Proposition \ref{pr:cauchy} we will always assume without loss of generality that the initial datum $u_0$ of a solution $u$ of \fref{eq:NLH} belongs to $\dot H^1\cap \dot H^3\cap W^{3,+\infty}$.


\subsection{The linearized operator $H$}


The spectral properties of $H=-\Delta-pQ^{p-1}$, the linearized operator of \fref{eq:NLH} close to $Q$, are well known, see \cite{Sc} and references therein. 

\begin{proposition}[Spectral theorem for the linearized operator] \label{pr:H}
Let $d\geq3$. Then $H:H^2\rightarrow L^2$ is self adjoint in $L^2$. It admits only one negative eigenvalue denoted by $-e_0$ where $e_0>0$, of multiplicity $1$, associated to a profile $\mathcal Y>0$  which decays together with its derivatives exponentially fast. For $d\geq 5$, $H$ admits $d+1$ zeros in $H^2$ given by the invariances of \eqref{eqnonline}:
\be \la{eq:ker H}
\text{Ker} H =\text{Span}(\Lambda Q,\partial_{x_1}Q,...,\partial_{x_d}Q).
\ee
The rest of the spectrum is contained in $[0,+\infty)$.
\end{proposition}

Our aim being to study perturbations of $Q$, we will decompose such perturbations in three pieces: a main term in $\mathcal M$, a part on the instable direction $\mathcal Y$, and a remainder being "orthogonal" to these latter. Unfortunately, in dimensions $7\leq d \leq 10$ the functions in the kernel do not decay quickly enough at infinity, what forces us to localize the natural orthogonality condition $u\in Ker(H)^{\perp}$. To do so we define:
\be \label{eq:def Psi0}
\Psi_0:=\chi_M \Lambda Q-\langle \chi_M \Lambda Q, \mathcal Y \rangle \mathcal Y,
\ee
\be \label{eq:def Psii}
\Psi_i:=\chi_M \partial_{x_i} Q, \ \ \text{for} \ 1\leq i \leq d .
\ee
for a constant $M\gg 1$ that can be chosen independently of the sequel. $\Psi_0,\Psi_1,...,\Psi_d$ look like $\Lambda Q, \partial_{x_1}Q,...,\partial_{x_d}Q$ for they satisfy the following orthogonality relations:
\be \label{eq:orthogonalite Psii}
\begin{array}{l l}
\langle \Psi_i,\mathcal Y \rangle=0, \ \ \text{for}  \ 0\leq i \leq d, \\
\langle \Psi_0,\Lambda Q  \rangle=\int \chi_M (\Lambda Q)^2, \ \ \text{and} \ \text{for} \ 1\leq i \leq d, \ \langle \Psi_i,\Lambda Q  \rangle=0, \\
\langle \Psi_i, \partial_{x_j} Q\rangle = \frac{1}{d} \int \chi_M |\nabla Q|^2 \delta_{i,j}  \ \ \text{for} \ 0\leq i \leq d, \ 1\leq j \leq d.
\end{array}
\ee
The potential $V=-pQ^{p-1}$ is in the Kato class, implying that the essential spectrum of $H$ is $[0,+\infty)$, the same as the one of the laplacian $-\Delta$. Consequently, $H$ is not $L^2$-coercive for functions orthogonal to its zeros and $\mathcal Y$. However, a weighted coercivity holds, similar to the one of the laplacian given by the Hardy inequality and its higher Sobolev versions.

\begin{lemma}[Weighted coercivity for $H$ on the stable subspace] 
\label{lem:coercivite}

Let $d\geq 7$. There exists a constant $C>0$ such that for all $v\in \dot H^1(\mathbb R^d)$ satisfying
\be \label{eq:orthogonalite}
v\in \text{Span}(\Psi_0,\Psi_1,...,\Psi_d,\mathcal Y)^{\perp} .
\ee
the following holds:\\
\noindent{\em (i) Energy bound}:
\be \label{eq:coerciviteA}
\frac 1 C \para v \para_{\dot H^1}^2\leq \int_{\mathbb R^d} |\nabla v|^2-pQ^{p-1}v^2 \leq C  \para v \para_{\dot H^1}^2
\ee
\noindent{\em (ii) $\dot{H}^2$ bound}: If $v\in \dot H^2(\mathbb R^d)$, then 
\be \label{eq:coercivite} 
\frac 1 C \para v \para_{\dot H^2}^2\leq \int_{\mathbb R^d} |Hv|^2 \leq C  \para v \para_{\dot H^2}^2
\ee
\noindent{\em (iii) $\dot{H}^3$ bound}: If $v\in \dot H^3(\mathbb R^d)$, then 
\be \label{eq:coercivite H3}
\frac 1 C \para v \para_{\dot H^3}^2 \leq \int_{\mathbb R^d} |\nabla Hv|^2-pQ^{p-1}|Hv|^2 \leq C \para v \para_{\dot H^3}^2
\ee
\end{lemma}

This result follows the scheme of proof in \cite{RaphRod} and details are given in Appendix \ref{sec:coercivite}.


\subsection{Geometrical decomposition of trapped solutions} 
From now on and for the rest of this section, we study solutions that are trapped near the ground state manifold $\mathcal M$.

\begin{definition}[Trapped solutions] \label{def:trapped}

Let $I \subset \mathbb R$ be a time interval containing $0$ and $0<\eta\ll1$. A solution $u$ of \fref{eq:NLH} is said to be trapped at distance $\eta$ on $I$ if:
\be \label{eq:proximite}
\underset{t\in I}{\text{sup}} \ d\left(u(t),\mathcal M\right)\leq \eta .
\ee
\end{definition}

A classical consequence of the orbital stability assumption \eqref{eq:proximite} is the existence of a geometrical decomposition of the flow adapted to the spectral structure of the linearized operator $H$ stated in Proposition \ref{pr:H} and Lemma \ref{lem:coercivite}. 

\begin{lemma}[Geometrical decomposition] \label{lem:decomposition}
Let $d\geq 7$. There exists $\delta>0$ such that for any $u_0\in \dot H^1(\mathbb R^d)$ with $\parallel u_0-Q\parallel_{\dot H^1}< \delta$, the following holds. If the solution of \fref{eq:NLH} given by Proposition \ref{pr:cauchy} is defined on some time interval $[0,T)$ and satisfies:
$$
\underset{0\leq t< T}{\text{sup}}d(u(t),\mathcal M)< \delta ,
$$
then there exists three functions $\lambda:[0,T)\rightarrow (0,+\infty)$, $z:[0,T)\rightarrow \mathbb R^d$, and $a:[0,T)\rightarrow \mathbb R$ that are $\mathcal C^1$ on $(0,T)$, such that
\be \label{eq:decomposition}
u=(Q+a\mathcal Y+v)_{z,\lambda},
\ee
where the function $v$ satisfies the orthogonality conditions \fref{eq:orthogonalite}. Moreover, one has the following estimate for each time $t\in [0,T)$:
\be \label{eq:estimation dotH1}
|a|+\parallel v \parallel_{\dot H^1}\lesssim \underset{\lambda>0, z\in \mathbb R^d}{\text{inf}}\parallel u-Q_{z,\lambda}\parallel_{\dot H^1} .
\ee

\end{lemma}

The proof of this result is standard and sketched in section \ref{sec:decomposition} for the sake of completeness. We therefore introduce the $C^1$ in time decomposition:
$$
u=(Q+a\mathcal Y+\varepsilon)_{z,\lambda} 
$$
with $\varepsilon$ satisfying the orthogonality conditions \fref{eq:orthogonalite}. We define the renormalized time $s=s(t)$ is by 
\be \label{eq:def s}
s(0)=0, \ \ \frac{d s}{d t}=\frac{1}{\lambda^2} .
\ee
and obtain the evolution equation in renormalized time
\be \label{eq:evolution}
a_s \mathcal Y+\varepsilon_s-\frac{\lambda_s}{\lambda}\left( \Lambda Q+a\Lambda \mathcal Y +\Lambda \varepsilon \right)-\frac{z_s}{\lambda}.\nabla \left( Q+a \mathcal Y + \varepsilon \right)=e_0 a\mathcal Y-H\varepsilon +NL 
\ee
where the nonlinear term is defined by:
$$
NL:= f(Q+a\mathcal Y+\varepsilon)-f(Q)-f'(Q)(a\mathcal Y+\varepsilon).
$$


\subsection{Modulation equations}


We start the analysis of the flow near $\mathcal M$ by the computation of the modulation equations.

\begin{lemma}[Modulation equations] \label{lem:modulation}
Let $d\geq 7$ and $I$ be a time interval containing $0$. There exists $0<\eta^*\ll 1$ such that if $u$ is trapped at distance $\eta$ for $0<\eta<\eta^*$ on $I$, then the following holds on $s(I)$:\\
\noindent{\em 1. Modulation equations}\footnote{The two "=O()" on each line represent two estimates, the first one being an explicit control in function of $a$ and $\varepsilon$, and the second one a uniform bound related to the distance to the manifold of ground states.}:
\bea
 \label{eq:modulation a}
&&a_s-e_0a=O(a^2+\parallel \varepsilon \parallel_{\dot H^2}^2)=O(\eta^2),\\
\label{eq:modulation lambda}
&&\frac{\lambda_s}{\lambda}=O(a^2+\parallel \varepsilon \parallel_{\dot H^2})=O(\eta ),\\
\label{eq:modulation z}
&&\frac{z_s}{\lambda} =O( a^2+\parallel \varepsilon \parallel_{\dot H^2})=O(\eta ).
\eea
\noindent{\em 2. Refined identities}\footnote{The notation $\frac{d}{ds}O(\cdot)$ means the derivative with respect to the renormalized time $s$ of a quantity satisfying the estimate $O(\cdot)$.}:
\be \label{eq:modulation lambda 2}
\left| \frac{\lambda_s}{\lambda}+\frac{d}{ds}O(\parallel \varepsilon \parallel_{\dot H^s})\right|\lesssim a^2+\parallel \varepsilon \parallel_{\dot H^2}^2 \ \ \text{for} \ \text{any} \ 1\leq s \leq \frac 4 3,
\ee
\be \label{eq:modulation z 2}
\left| \frac{z_s}{\lambda} +\frac{d}{ds}O(\parallel \varepsilon \parallel_{\dot H^s})\right|\lesssim a^2+\parallel \varepsilon \parallel_{\dot H^2}^2 \ \ \text{for} \ \text{any} \ 1\leq s \leq 2.
\ee
\end{lemma}

\begin{proof}[Proof of Lemma \ref{lem:modulation}] We compute the modulation equations as a consequence of the orthogonality conditions \eqref{eq:orthogonalite} for $\e$. However, since in dimension $7\leq d\le 10$ the $\Psi_j$ are not exact elements of the kernel of $H$, we will need the refined space time bounds \eqref{eq:modulation lambda 2}, \eqref{eq:modulation z 2}, see \cite{MR4,RSc2} for related issued.\\

\noindent{\bf step 1} First algebraic identities. We first prove \fref{eq:modulation a}, \fref{eq:modulation lambda} and \fref{eq:modulation z}.\\

\noindent\emph{First identity for $a$}. First, taking the scalar product between \fref{eq:evolution} and $\mathcal Y$ yields, using the orthogonality \fref{eq:orthogonalite}, \fref{eq:ker H}, the pointwise estimate for the nonlinearity \fref{eq:NL2}, the generalized Hardy inequality \fref{eq:hardy}, the integration by parts formula \fref{eq:ipp} and the fact that $\mathcal Y$ decays exponentially fast:
\be \label{eq:modulation a1}
\begin{array}{r c l}
a_s-e_0a & = & \frac{1}{\parallel \mathcal Y \parallel_{L^2}^2}\left(\frac{\lambda_s}{\lambda}\int (a\Lambda \mathcal Y+\Lambda \varepsilon )\mathcal Y+\int NL\mathcal Y +\int \frac{z_s}{\lambda}.\nabla \varepsilon \mathcal Y \right) \\
&=& \frac{\lambda_s}{\lambda} O\left( |a|\int e^{-C|x|}+\int |\varepsilon|e^{-C|x}\right)+O\left(\left| \frac{z_s}{\lambda}\right|\int e^{-C|x|}|\varepsilon| \right)\\
&&+O\left(a^2\int e^{-C|x|} Q^{p-2}+\int e^{-C|x|}Q^{p-2} |\varepsilon |^2 \right)\\
&=& O\left( \left| \frac{\lambda_s}{\lambda} \right| (a+\parallel \varepsilon \parallel_{\dot H^s})\right)+O(a^2+\parallel \varepsilon \parallel_{\dot H^s}^2)+O\left(\left| \frac{z_s}{\lambda}\right|\parallel \varepsilon \parallel_{\dot H^s} \right).\\
\end{array}
\ee
for $s\in [1,2]$. Taking $s=1$ and injecting \fref{eq:estimation dotH1} gives:
\be \label{eq:modulation adelta}
a_s-e_0a = \frac{\lambda_s}{\lambda}O(\eta)+O(\eta^2)+O\left(\left| \frac{z_s}{\lambda}\right|\eta \right).\\
\ee

\noindent \emph{First identity for $\lambda$}.Taking the scalar product between \fref{eq:evolution} and $\Psi_0$, using \fref{eq:orthogonalite}, \fref{eq:orthogonalite Psii} gives:
\be \label{eq:modulation lambda 7 expression}
-\frac{\lambda_s}{\lambda}\int \left[\Lambda Q+a\Lambda \mathcal Y+\Lambda \varepsilon \right]\Psi_0-\int \frac{z_s}{\lambda}.\nabla \varepsilon \Psi_0  = \int NL\Psi_0-\int H\varepsilon \Psi_0 
\ee
where we used the fact that $\int \partial_{x_i}\mathcal Y\Psi_0=0$ as $\mathcal Y$ and $\Psi_0$ are radial. \fref{eq:ipp} and the Hardy inequality \fref{eq:hardy} give, as $\Psi_0$ is exponentially decaying:
$$
\int \Lambda \varepsilon \Psi_0 =O\left(\int |\varepsilon| e^{-C|x|} \right)=O\left(\para \varepsilon \para_{\dot H^1} \right)=O(\eta).
$$
Hence from \fref{eq:estimation dotH1} and \fref{eq:orthogonalite Psii} the first term in the left hand side of \fref{eq:modulation lambda 7 expression} is:
$$
-\frac{\lambda_s}{\lambda}\int \left[\Lambda Q+a\Lambda \mathcal Y+\Lambda \varepsilon \right]\Psi_0 =\frac{\lambda_s}{\lambda}\left(-\int \chi_M (\Lambda Q)^2 +O(\eta) \right).
$$
For the second term in the left hand side of \fref{eq:modulation lambda 7 expression} one uses Hardy inequality \fref{eq:hardy} for $1\leq i \leq d$ and the fact that $\Psi_0$ is exponentially decreasing:
$$
\int \partial_{y_i} \varepsilon \Psi_0   = O(\para \varepsilon \para_{\dot H^1}) =O(\eta).
$$
For the right hand side we use the pointwise estimate \fref{eq:NL2} on the nonlinearity and the Hardy inequality \fref{eq:hardy} to obtain, as $\Psi_0$ is exponentially decreasing:
$$
\int NL\Psi_0 = O\left(\int |\Psi_0|Q^{p-2}|a\mathcal Y+\varepsilon|^2 \right) = O\left(a^2+\parallel \varepsilon \parallel^2_{\dot H^s} \right)
$$
for $s\in [1,2]$. The linear term is estimated using \fref{eq:coerciviteA}, \fref{eq:coercivite} and the fact that $\Psi_0$ is exponentially decreasing:
$$
\left| \int H\varepsilon \Psi_0 \right|=\left| \int \varepsilon H\Psi_0 \right|\lesssim \parallel \varepsilon \parallel_{\dot H^s}
$$
for $s\in [1,2]$. The five previous equations give then:
\be \label{eq:modulation lambda1}
\frac{\lambda_s}{\lambda}=O\left( a^2+\parallel \varepsilon \parallel_{\dot H^s} \right)+O\left(\eta \left| \frac{z_s}{\lambda}\right| \right).
\ee
for $s\in [1,2]$. Taking $s=1$ and injecting \fref{eq:estimation dotH1} one obtains:
\be \label{eq:modulation lambdadelta}
\frac{\lambda_s}{\lambda}=O\left( \eta \right)+O\left(\eta \left| \frac{z_s}{\lambda}\right| \right).
\ee

\noindent \emph{First identity for $z$}. We take the scalar product between \fref{eq:evolution} and $\Psi_i$ for $1\leq i \leq d$, using \fref{eq:orthogonalite} and \fref{eq:orthogonalite Psii}:
$$
-\frac{\lambda_s}{\lambda}\int \Lambda \varepsilon \Psi_i- \frac{z_{i,s}}{\lambda}\int \partial_{y_i} (Q+a\mathcal Y+\varepsilon) \Psi_i-\sum_{i\neq j} \frac{z_{j,s}}{\lambda} \int \partial_{x_j}\varepsilon \Psi_i  = \int NL\Psi_i-\int H\varepsilon \Psi_i .
$$
For the first term, using \fref{eq:ipp} and Hardy inequality \fref{eq:hardy} as $\Psi_i$ is compactly supported one finds:
$$
\int \Lambda \varepsilon \Psi_i =O(\eta ).
$$
For the second term, using Hardy inequality and \fref{eq:orthogonalite Psii}, as $\Psi_i$ is compactly supported one gets:
$$
\begin{array}{l l l l}
&\int \partial_{y_i} (Q+a\mathcal Y+\varepsilon) \Psi_i & = & \int \chi_M(\partial_{y_i}Q)^2+O(|a|)-\int \varepsilon \partial_{x_i}\Psi_i \\
 = & \int \chi_M(\partial_{y_i}Q)^2+O(\eta)+O\left(\parallel \varepsilon \parallel_{\dot H^1} \right) & = & \int \chi_M(\partial_{y_i}Q)^2+O(\eta)
\end{array}
$$
and similarly for $1\leq j \leq d$ with $j\neq i$:
$$
\int \partial_{y_j}\varepsilon \Psi_i=O(\eta).
$$
For the nonlinear term we use the pointwise estimate \fref{eq:NL2} on the nonlinearity, the Hardy inequality \fref{eq:hardy} and the fact that $\Psi_i$ is compactly supported:
$$
\begin{array}{r c l}
\int NL\Psi_i & = & O\left(\int |\Psi_i|Q^{p-2}|a\mathcal Y+\varepsilon|^2 \right) = O\left(a^2+\int \varepsilon ^2 |\Psi_i|Q^{p-2} \right)\\
&=& O\left(a^2+\parallel \varepsilon \parallel^2_{\dot H^s} \right)
\end{array}
$$
for $s\in [1,2]$. For the linear term, as $\Psi_i$ is compactly supported, using Hardy inequality \fref{eq:hardy}:
$$
\int H\varepsilon \Psi_i=O(\parallel \varepsilon \parallel_{\dot H^s})
$$
for $s\in [1,2]$. The four previous equations give then:
\be \label{eq:modulation z1}
\frac{z_{i,s}}{\lambda}=O\left( a^2+\parallel \varepsilon \parallel_{\dot H^s} \right)+\sum_{j\neq i} O\left( \eta \left| \frac{z_{j,s}}{\lambda}\right| \right)+O\left( \eta \left| \frac{\lambda_s}{\lambda}\right| \right).
\ee
for $s\in[1,2]$.Taking $s=1$ and injecting \fref{eq:estimation dotH1} one gets:
\be \label{eq:modulation zdelta}
\frac{z_{i,s}}{\lambda}=O\left( \eta \right)+\sum_{j\neq i} O\left( \eta \left| \frac{z_{j,s}}{\lambda}\right| \right)+O\left( \eta \left| \frac{\lambda_s}{\lambda}\right| \right).
\ee

\noindent \emph{Conclusion}. We gather the primary identities \fref{eq:modulation a1}, \fref{eq:modulation adelta}, \fref{eq:modulation lambda1}, \fref{eq:modulation lambdadelta}, \fref{eq:modulation z1}, \fref{eq:modulation zdelta} and bootstrap the information they contained for the mixed terms, yielding the identities \fref{eq:modulation a}, \fref{eq:modulation lambda} and \fref{eq:modulation z}.\\

\noindent{\bf step 2} Refined identities. We now show \fref{eq:modulation lambda 2} and \fref{eq:modulation z 2}. We perform an integration by part in time, improving the modulation equations by removing the derivative with respect to time of the projection of $\varepsilon$ onto $\Lambda Q$ and $\nabla Q$. Note that this procedure \emph{requires} sufficient decay of $\Lambda Q$ and hence the assumption $d\ge 7$.\\

\noindent \emph{Improved modulation equation for $z$}. For $1\leq i \leq d$ the quantity $\int \varepsilon \partial_{y_i} Q$ is well defined for $7\leq d$ via Sobolev embedding. We compute the following identity:
\be \label{eq:improved z expression}
\begin{array}{r c l}
\frac{d}{ds} \left[\frac{\langle \varepsilon,\partial_{y_i}Q\rangle}{\langle \partial_{y_i}(Q+a\mathcal Y+\varepsilon),\partial_{y_i}Q\rangle} \right]= \frac{\langle \varepsilon_s,\partial_{y_i}Q\rangle}{\langle \partial_{y_i}(Q+a\mathcal Y+\varepsilon),\partial_{y_i}Q\rangle}-\frac{\langle \varepsilon,\partial_{y_i}Q\rangle \langle \partial_{y_i}(a_s\mathcal Y+\varepsilon_s),\partial_{y_i}Q\rangle}{\langle \partial_{y_i}(Q+a\mathcal Y+\varepsilon),\partial_{y_i}Q\rangle^2 }
\end{array}
\ee
Using Hardy inequality \fref{eq:hardy} and \fref{eq:estimation dotH1} one has as $ 7\leq d$:
\be \label{eq:numerateur z}
\begin{array}{r c l}
\left| \langle \varepsilon,\partial_{y_i}Q\rangle \right| \leq \int \frac{|\varepsilon|}{1+|y|^{d-1}} &\leq& \left( \int \frac{|\varepsilon|^2}{1+|y|^4}\right)^{\frac 1 2 }\left( \int \frac{1}{1+|y|^{2d-6}}\right)^{\frac 1 2 }\\
&=&O(\parallel \varepsilon \parallel_{\dot H^s}) \ \text{for} \ \text{any} \ 1\leq s \leq 2,
\end{array}
\ee
\be \label{eq:denominateur z}
\langle \partial_{y_i}(Q+a\mathcal Y+\varepsilon),\partial_{y_i}Q\rangle= \int (\partial_{y_i}Q)^2+O(\eta) ,
\ee
and therefore the quantity on the left hand side of \fref{eq:improved z expression} is for $\eta^*$ small enough:
\be \label{eq:improved z 1}
\frac{\langle \varepsilon,\partial_{y_i}Q\rangle}{\langle \partial_{y_i}(Q+a\mathcal Y+\varepsilon),\partial_{y_i}Q\rangle}= O(\parallel \varepsilon \parallel_{\dot H^s}) \ \text{for} \ \text{any} \ 1\leq s \leq 2
\ee
From \fref{eq:evolution} the numerator of the first quantity in the right hand side of \fref{eq:improved z expression} is:
$$
\begin{array}{r c l}
\langle \varepsilon_s,\partial_{y_i}Q\rangle&=& \frac{\lambda_s}{\lambda} \langle \Lambda \varepsilon , \partial_{y_i}Q\rangle+\frac{z_{i,s}}{\lambda}\langle \partial_{y_i}(Q+a\mathcal Y+\varepsilon),\partial_{y_i} Q\rangle \\
&&+\sum_{j=1, \ j\neq i}^d \frac{z_{j,s}}{\lambda}\langle \partial_{y_j}(a\mathcal Y+\varepsilon),\partial_{y_i} Q\rangle-\langle H\varepsilon,\partial_{y_i}Q \rangle +\langle NL,\partial_{y_i}Q \rangle.
\end{array}
$$
Using \fref{eq:modulation lambda}, \fref{eq:ipp}, \fref{eq:estimation dotH1} and Hardy inequality \fref{eq:hardy} one gets for the first term:
$$
\begin{array}{r c l}
\left| \frac{\lambda_s}{\lambda} \langle \Lambda \varepsilon , \partial_{y_i}Q\rangle \right| & \leq & \left|\frac{\lambda_s}{\lambda} \right| \int \frac{|\varepsilon|}{1+|y|^{d-1}} \lesssim  (a^2+\parallel \varepsilon \parallel_{\dot H^2} ) \left( \int \frac{|\varepsilon|^2}{1+|y|^4}\right)^{\frac 1 2 } \left( \int \frac{1}{1+|y|^{2d-6}}\right)^{\frac 1 2 } \\
&\lesssim & a^2+\parallel \varepsilon \parallel_{\dot H^2}^2 .
\end{array}
$$
Similarly for the third term using \fref{eq:modulation z}, \fref{eq:estimation dotH1} and Hardy inequality \fref{eq:hardy}:
$$
\begin{array}{l l l l}
&\left| \sum_{j=1, \ j\neq i}^d \frac{z_{j,s}}{\lambda}\langle \partial_{y_j}(a\mathcal Y+\varepsilon),\partial_{y_i} Q\rangle \right| & \leq & \left|\frac{z_s}{\lambda} \right| (|a|+\int \frac{|\varepsilon|}{1+|y|^{d}}) \\
\lesssim & (a^2+\parallel \varepsilon \parallel_{\dot H^2} ) \left( |a|+ \int \frac{|\varepsilon|^2}{1+|y|^4}\right)^{\frac 1 2 } \left( \int \frac{1}{1+|y|^{2d-4}}\right)^{\frac 1 2 } &\lesssim & a^2+\parallel \varepsilon \parallel_{\dot H^2}^2 .
\end{array}
$$
For the fourth term, performing an integration by parts:
$$
\langle H\varepsilon,\partial_{y_i}Q\rangle =\int \varepsilon H\partial_{y_i}Q=0.
$$
For the last term, the nonlinear one, one computes using the hardy inequality \fref{eq:hardy}:
$$
\begin{array}{r c l}
\int NL\partial_{y_i}Q & = & O\left(\int |\partial_{y_i}Q|Q^{p-2}|a\mathcal Y+\varepsilon|^2 \right) = O\left(a^2+\int \varepsilon ^2 |\partial_{y_i}Q|Q^{p-2} \right)\\
&=&O\left(a^2+\int \frac{\varepsilon ^2}{1+|y|^{5}}  \right)  =  O\left(a^2+\parallel \varepsilon \parallel^2_{\dot H^s} \right)
\end{array}
$$
The five identities above plus \fref{eq:denominateur z} imply that the second term in \fref{eq:improved z expression} is:
\be \label{eq:improved z 2}
\frac{\langle \varepsilon_s,\partial_{y_i}Q\rangle}{\langle \partial_{y_i}(Q+a\mathcal Y+\varepsilon),\partial_{y_i}Q\rangle}= \frac{z_{i,s}}{\lambda}+O(a^2+\parallel \varepsilon \parallel_{\dot H^2}^2).
\ee
We now turn to the third term in \fref{eq:improved z expression}. One computes from \fref{eq:evolution}, using \fref{eq:modulation lambda} and \fref{eq:modulation z} that\footnote{We do not redo here the application of Hardy inequality and of \fref{eq:NL3} to control the linear and nonlinear term as we have used them numerous times before in the proof.}:
$$
\int \partial_{y_i}(a_s \mathcal Y+\varepsilon_s)\partial_{y_i}Q=O\left(\left| \frac{\lambda_s}{\lambda}\right|+\left| \frac{z_s}{\lambda}\right|+|a|+\parallel \varepsilon \parallel_{\dot H^2} \right)=O(|a|+\parallel \varepsilon \parallel_{\dot H^2}).
$$
With \fref{eq:denominateur z} and \fref{eq:numerateur z} the above estimate yields for the third term in \fref{eq:improved z expression}:
\be \label{eq:improved z 3}
\left| \frac{\langle \varepsilon,\partial_{y_i}Q\rangle \langle \partial_{y_i}(a_s\mathcal Y+\varepsilon_s),\partial_{y_i}Q\rangle}{\langle \partial_{y_i}(Q+a\mathcal Y+\varepsilon),\partial_{y_i}Q\rangle^2 } \right|\lesssim a^2+\parallel \varepsilon \parallel_{\dot H^2}^2.
\ee
We now come back to the identity \fref{eq:improved z expression}, and inject the estimates \fref{eq:improved z 1}, \fref{eq:improved z 2} and \fref{eq:improved z 3} we have found for each term, yielding the improved modulation equation \fref{eq:modulation z 2} claimed in the lemma.\\

\noindent \emph{Improved modulation equation for $\lambda$}. The quantity $\int \varepsilon \Lambda Q$ is well defined for $d\geq 7$ from the Sobolev embedding of $\dot H^1$ into $L^{\frac{2d}{d-2}}$. From \fref{eq:evolution} one computes first the identity:
\be \label{eq:improved lambda expression}
\begin{array}{r c l}
\frac{d}{ds} \left[\frac{\langle \varepsilon,\Lambda Q\rangle}{\langle \Lambda(Q+\varepsilon),\Lambda Q\rangle} \right]= \frac{\langle \varepsilon_s,\Lambda Q\rangle}{\langle \Lambda(Q+\varepsilon),\Lambda Q\rangle}-\frac{\langle \varepsilon,\Lambda Q\rangle \langle \Lambda(\varepsilon_s),\Lambda Q\rangle}{\langle \Lambda(Q+\varepsilon),\Lambda Q\rangle^2 }
\end{array}
\ee
We now compute all the terms in the previous identity. 

\noindent $\circ $ \emph{Left hand side of \fref{eq:improved lambda expression}}. One computes:
\be \label{eq:numerateur lambda}
\begin{array}{r c l}
\left| \langle \varepsilon,\Lambda Q\rangle \right| &\lesssim& \int \frac{|\varepsilon |}{1+|y|^{d-2}} \lesssim \left(\int \frac{|\varepsilon |^2}{1+|y|^{2+\frac 2 3}} \right)^{\frac 1 2} \left(\int \frac{1}{1+|y|^{2d -\frac{20}{3}}} \right)^{\frac 1 2}\\
&\lesssim& \parallel \varepsilon \parallel_{\dot H^s} \ \text{for} \ \text{any} \ 1\leq s \leq 1+\frac 1 3
\end{array}
\ee
as  since $d\geq 7$ one has indeed $\frac{1}{1+|x|^{2d -\frac{20}{3}}}\in L^1$. For the denominator one computes using \fref{eq:ipp}, Hardy inequality \fref{eq:hardy} and \fref{eq:estimation dotH1}:
\bea
\label{eq:denominateur lambda}
\nonumber \langle \Lambda(Q+\varepsilon),\Lambda Q\rangle & = & \int (\Lambda Q)^2-\int \varepsilon \Lambda^2Q-2\int \varepsilon Q=\int (\Lambda Q)^2+O(\parallel \varepsilon \parallel_{\dot H^1}) \\
& = & \int (\Lambda Q)^2+O(\eta).
\eea
We then conclude that the quantity in the left hand side of \fref{eq:improved lambda expression} is:
\be \label{eq:improved lambda 1}
\frac{\langle \varepsilon,\Lambda Q\rangle}{\langle \Lambda(Q+\varepsilon),\Lambda Q\rangle}=O(\parallel \varepsilon \parallel_{\dot H^s}) \ \ \text{for} \ \text{any} \ 1\leq s \leq 1+\frac 1 3 .
\ee

\noindent $\circ $ \emph{First term in the right hand side of \fref{eq:improved lambda expression}}. Using \fref{eq:evolution} one has:
$$
\begin{array}{r c l}
\langle \varepsilon_s,\Lambda Q\rangle &=& \frac{\lambda_s}{\lambda}\langle \Lambda (Q+a\mathcal Y+\varepsilon),\Lambda Q)-\langle H\varepsilon, \Lambda Q\rangle \\
&&+\langle \frac{z_s}{\lambda}.\nabla (a\mathcal Y+\varepsilon),\Lambda Q\rangle+\langle NL,\Lambda Q \rangle .
\end{array}
$$
For the second term of the right hand side we perform an integration by parts:
$$
\langle H\varepsilon, \Lambda Q\rangle=\langle \varepsilon,H\Lambda Q\rangle =0.
$$
For the third term, performing an integration by parts, using \fref{eq:modulation z} and Hardy inequality \fref{eq:hardy} one gets:
$$
\begin{array}{r c l}
&\left| \int \frac{z_s}{\lambda}.\nabla (a\mathcal Y+\varepsilon) \Lambda Q\right| \\
\leq & \left|\frac{z_s}{\lambda} \right| \int (|a|\mathcal Y+|\varepsilon|)|\nabla \Lambda Q| \lesssim  (a^2+\parallel \varepsilon \parallel_{\dot H^2}) (|a|+\int \frac{|\varepsilon|}{1+|y|^{d-1}}) \\
\lesssim & (a^2+\parallel \varepsilon \parallel_{\dot H^2}) \left(|a|+\left( \int \frac{|\varepsilon|^2}{1+|y|^{2}}\right)^{\frac 1 2}\left( \int \frac{1}{1+|y|^{2d-6}}\right)\right)^{\frac 1 2} \lesssim  a^2+\parallel \varepsilon \parallel_{\dot H^2}^2.
\end{array}
$$
For the nonlinear term using the Hardy inequality \fref{eq:hardy}
$$
\begin{array}{r c l}
\int NL\Lambda Q & = & O\left(\int |\Lambda Q|Q^{p-2}|a\mathcal Y+\varepsilon|^2 \right) = O\left(a^2+\int \varepsilon ^2 |\Lambda Q|Q^{p-2} \right)\\
&=&O\left(a^2+\int \frac{\varepsilon ^2}{1+|y|^{4}}  \right)  =  O\left(a^2+\parallel \varepsilon \parallel^2_{\dot H^2} \right)
\end{array}
$$
The four identities above, plus \fref{eq:modulation lambda}, and \fref{eq:denominateur lambda} give that the first term in the right hand side of \fref{eq:improved lambda expression} is:
\be \label{eq:improved lambda 2}
\frac{\langle \varepsilon_s,\Lambda Q\rangle}{\langle \Lambda(Q+\varepsilon),\Lambda Q\rangle}=\frac{\lambda_s}{\lambda}+O(a^2+\parallel \varepsilon \parallel_{\dot H^2}^2).
\ee

\noindent $\circ$ \emph{Second term in the right hand side of \fref{eq:improved lambda expression}}. From the asymptotic of the solitary wave \fref{eq:def Q} there exists a constant $\tilde c\in \mathbb R$ such that:
$$
\Lambda^2 Q+2\Lambda Q=\tilde c (\Lambda Q+R)
$$
where $R$ is a function satisfying an improved decay at infinity:
$$
|\partial^\alpha R|\lesssim \frac{1}{1+|x|^{d+|\alpha|}}.
$$
Using \fref{eq:ipp} one then computes that:
\bea
 \label{eq:improved lambda 3 expression}
\nonumber&&- \frac{\langle \varepsilon,\Lambda Q \rangle \langle \Lambda(\varepsilon_s),\Lambda Q\rangle}{\langle \Lambda(Q+\varepsilon),\Lambda Q\rangle^2 } \non  =\tilde c \frac{\langle \varepsilon ,\Lambda Q+R\rangle \langle \varepsilon_s ,\Lambda Q+R\rangle }{\int (\Lambda Q)^2+\langle \varepsilon ,\Lambda Q+R\rangle}-\tilde c \frac{\langle \varepsilon ,R\rangle \langle \varepsilon_s ,\Lambda Q+R\rangle }{\int (\Lambda Q)^2+\langle \varepsilon ,\Lambda Q+R\rangle} \\
\nonumber&= &\frac{\tilde c}{2}\frac{d}{ds}\left[ \frac{\langle \varepsilon ,\Lambda Q+R\rangle^2 }{\int (\Lambda Q)^2+\langle \varepsilon ,\Lambda Q+R\rangle}\right]+\frac{\tilde c}{2}\frac{\langle \varepsilon ,\Lambda Q+R\rangle^2\langle \varepsilon_s ,\Lambda Q+R\rangle }{\left(\int (\Lambda Q)^2+\langle \varepsilon ,\Lambda Q+R\rangle\right)^2} \\
&&-\tilde c \frac{\langle \varepsilon ,R\rangle \langle \varepsilon_s ,\Lambda Q+R\rangle }{\int (\Lambda Q)^2+\langle \varepsilon ,\Lambda Q+R\rangle} \\
\nonumber&= &\frac{\tilde c}{2}\frac{d}{ds}\left[ \frac{\langle \varepsilon ,\Lambda Q+R\rangle^2 }{\int (\Lambda Q)^2+\langle \varepsilon ,\Lambda Q+R\rangle}+\frac 1 3 \frac{\langle \varepsilon ,\Lambda Q+R\rangle^3 }{\left(\int (\Lambda Q)^2+\langle \varepsilon ,\Lambda Q+R\rangle\right)^2}\right]  \\
&+&\frac{\tilde c}{3}\frac{\langle \varepsilon ,\Lambda Q+R\rangle^3\langle \varepsilon_s ,\Lambda Q+R\rangle }{\left(\int (\Lambda Q)^2+\langle \varepsilon ,\Lambda Q+R\rangle\right)^3} -\tilde c \frac{\langle \varepsilon ,R\rangle \langle \varepsilon_s ,\Lambda Q+R\rangle }{\int (\Lambda Q)^2+\langle \varepsilon ,\Lambda Q+R\rangle}.
\eea
Using the Hardy inequality and \fref{eq:improved lambda 1} the first term in the right hand side is:
$$
\left| \frac{\langle \varepsilon ,\Lambda Q+R\rangle^2 }{\int (\Lambda Q)^2+\langle \varepsilon ,\Lambda Q+R\rangle}+\frac 1 3 \frac{\langle \varepsilon ,\Lambda Q+R\rangle^3 }{\left(\int (\Lambda Q)^2+\langle \varepsilon ,\Lambda Q+R\rangle\right)^2}\right| \lesssim \eta \parallel \varepsilon \parallel_{\dot H^s}
$$
for any $1\leq s \leq 1+\frac 1 3$. Next one computes from \fref{eq:evolution} using \fref{eq:modulation a}, \fref{eq:modulation lambda} and \fref{eq:modulation z} that:
$$
\left| \langle \varepsilon_s, \Lambda Q+R \rangle \right|\lesssim a^2+\parallel \varepsilon \parallel_{\dot H^2}.
$$
Using Hardy inequality \fref{eq:hardy} one has:
$$
\left|\int \varepsilon R\right|\lesssim \int \frac{|\varepsilon|}{1+|y|^{d}}\lesssim \left( \int \frac{|\varepsilon|}{1+|y|^4}\right)^{\frac 1 2}\left( \int \frac{1}{1+|y|^{2d-4}}\right)^{\frac 1 2}\lesssim \parallel \varepsilon \parallel_{\dot H^2}.
$$
Using the generalized Hardy inequality, interpolation and \fref{eq:estimation dotH1} one finds:
$$
\begin{array}{r c l}
\left| \langle \varepsilon ,\Lambda Q+R\rangle \right| & \lesssim & \int \frac{|\varepsilon|}{1+|y|^{d-2}} \lesssim \left(\int \frac{|\varepsilon |^2}{1+|y|^{2+\frac 2 3}} \right)^{\frac 1 2} \left(\int \frac{1}{1+|y|^{2d -\frac{20}{3}}} \right)^{\frac 1 2}\\
&\lesssim& \parallel \varepsilon \parallel_{\dot H^{1+ \frac 1 3}} \lesssim  \parallel \varepsilon \parallel_{\dot H^1}^{\frac{2}{3}} \parallel \varepsilon \parallel_{\dot H^2}^{\frac{1}{3}} \lesssim \eta^{\frac 2 3} \parallel \varepsilon \parallel_{\dot H^2}^{\frac 1 3}.
\end{array}
$$
Injecting the four equations above in \fref{eq:improved lambda 3 expression} one finds that the second term in the right hand side of \fref{eq:improved lambda expression} can be rewritten as:
\be \label{eq:improved lambda 3}
- \frac{\langle \varepsilon,\Lambda Q \rangle \langle \Lambda(\varepsilon_s),\Lambda Q\rangle}{\langle \Lambda(Q+\varepsilon),\Lambda Q\rangle^2 } = \frac{d}{ds}O\left( \eta \parallel \varepsilon \parallel_{\dot H^s} \right)+O(a^2+\parallel \varepsilon \parallel_{\dot H^2}^2).
\ee

\noindent $\circ$ \emph{End of the proof of the improved modulation equation for $\lambda$.} We now come back to the identity \fref{eq:improved lambda expression}. We computed each term appearing in \fref{eq:improved lambda 1}, \fref{eq:improved lambda 2} and \fref{eq:improved lambda 3}, implying the improved modulation equation \fref{eq:modulation lambda 2} claimed in the lemma and ending its proof.

\end{proof}


\subsection{Energy bounds for trapped solutions}


We now provide the necessary pointwise and space time parabolic bounds on the flow which will allow us to close the control of the modulation equations of Lemma \ref{lem:modulation} in the trapped regime near $\mathcal M$.\\
We first claim a global space time energy bound.

\begin{lemma}[Global energy bound] \label{lem:variation energie}
Let $d\geq 7$ and $I$ be a time interval containing $0$. There exists $0<\eta^*\ll 1$ such that if $u$ is trapped at distance $\eta$ for $0<\eta<\eta^*$ on $I$ then:
\be \label{eq:variation energie}
\int_{s(I)} \left( \parallel \varepsilon \parallel_{\dot H^2}^2+a^2 \right) ds\lesssim \eta ^2 
\ee
\end{lemma}

\begin{remark}
In dimension $d\geq 11$, one does not need to localize the orthogonality conditions in \fref{eq:orthogonalite} and $Q$ decays faster. One can then obtain simpler and stronger estimates, i.e. for \fref{eq:modulation lambda} and \fref{eq:modulation z} one would have $\para \varepsilon \para_{\dot H^2}^2$ instead of $\para \varepsilon \para_{\dot H^2}$, and for \fref{eq:modulation lambda 2} and \fref{eq:modulation z 2} one would remove the boundary term $\frac{d}{ds}O(\parallel \varepsilon \parallel_{\dot H^s})$, easing the sequel.
\end{remark}

\begin{proof}[Proof of Lemma \ref{lem:variation energie}]

As $Q$ solves $\Delta Q+Q^p=0$, from \fref{eq:NL10}, Sobolev embedding and Hardy inequality, for any $v\in \dot H^1$:
\bee
& & |E(Q+v)-E(Q)| \\
&=& \left| \frac{1}{2}\int |\nabla (Q+v)|^2-\frac{1}{p+1}\int |Q+v|^{p+1}-\frac{1}{2}\int |\nabla Q|^2+\frac{1}{p+1}\int Q^{p+1}  \right| \\
&\leq & \left| -\int (\Delta Q+Q^p) v+\frac 1 2 \int|\nabla v|^2-\frac{1}{p+1} \int (|Q+v|^{p+1}-Q^{p+1}-(p+1)Q^pv) \right| \\
&\leq & \frac 1 2 \int|\nabla v|^2+\frac{1}{p+1} \int ||Q+v|^{p+1}-Q^{p+1}-(p+1)Q^pv|  \\
&\leq & C \para v \para_{\dot H^1}^2+C \int |Q^{p-1}v^2|+|v|^{p+1} \leq  C \para v \para_{\dot H^1}^2+C \int \frac{v^2}{1+|x|^4}+\para v\para^{p+1}_{\dot H^1}  \\
&\leq & C \para v \para_{\dot H^1}^2+\para v\para^{p+1}_{\dot H^1}. \\
\eee
From the above identity, the closeness assumption \fref{eq:proximite} and the invariance of the energy via scale and translation, we then estimate:
$$
\forall t\in I, \ |E(u)-E(Q)|\lesssim \eta^2 .
$$
Using the regularizing effects from Proposition \ref{pr:cauchy} and \fref{eq:E} one computes:
$$
\int_I \parallel u_t \parallel_{L^2}^2 dt = \underset{t\rightarrow \text{inf}(I)}{\text{lim}} E(u(t))-\underset{t\rightarrow \text{sup}(I)}{\text{lim}} E(u(t)) \lesssim \eta ^2 .
$$
From \fref{eq:def s} and \fref{eq:evolution}, in renormalized variables, the left hand side is
\bee
&&\int_I \parallel u_t \parallel_{L^2}^2 dt = \int_I \int_{\mathbb R^d} \left(\Delta u +|u|^{p-1}u \right)^2dt\\
&=& \int_I \int_{\mathbb R^d} \left(\Delta (Q+a\mathcal Y+\varepsilon)_{z,\lambda} +|(Q+a\mathcal Y+\varepsilon)_{z,\lambda}|^{p-1}(Q+a\mathcal Y+\varepsilon)_{z,\lambda} \right)^2dt \\
&=& \int_I \int_{\mathbb R^d} \frac{1}{\lambda^2} \left(e_0a\mathcal Y-H\varepsilon +NL \right)^2dt= \int_{s(I)}  \para e_0a\mathcal Y-H\varepsilon +NL \para_{L^2}^2ds.
\eee
Therefore the two above equations imply:
\be \label{eq:energie0}
\int_{s(I)} \parallel e_0 a\mathcal Y-H\varepsilon +NL \parallel_{L^2}^2 ds \lesssim \eta ^2 .
\ee
We now show that the contribution of the two linear terms are decorelated, and that they control the nonlinear one. Young's inequality $ab\leq \frac{a}{4}+2b$ yields:
\begin{align*}
\parallel e_0 a\mathcal Y-H\varepsilon +NL \parallel_{L^2}^2 
=& \int \left( e_0 a\mathcal Y-H\varepsilon +NL \right)^2  \\
=& \int \left( e_0 a\mathcal Y-H\varepsilon \right)^2 +2 \int \left( e_0 a\mathcal Y-H\varepsilon\right) NL  +\int NL^2  \\
\geq & \frac{1}{2} \int (e_0a\mathcal Y-H\varepsilon)^2-\int (NL)^2.
\end{align*}
From the orthogonality \fref{eq:orthogonalite} and the coercivity estimate \fref{eq:coercivite} the linear term controls the following quantities:
$$
\int \left( e_0 a\mathcal Y-H\varepsilon \right)^2 \gtrsim a^2 + \parallel \varepsilon \parallel_{\dot H ^2}^2.
$$
From the estimate \fref{eq:NL1} on the nonlinearity, Sobolev embedding, interpolation and \fref{eq:estimation dotH1}:
$$
\begin{array}{r c l}
\int NL^2 & \lesssim & \int |a\mathcal Y+\varepsilon |^{2p} \lesssim  \int a^{2p}\mathcal Y^{2p}+|\varepsilon|^{2p} \lesssim   a^{2p}+\parallel \varepsilon \parallel^{2p}_{\dot H^{\frac{2d}{d+2}}} \\
&\lesssim  & a^{2p}+\parallel \varepsilon \parallel^{2p\times \frac{4}{d+2}}_{\dot H^1}\parallel \varepsilon \parallel^{2p\times \frac{d-2}{d+2}}_{\dot H^2} \lesssim   a^2\delta^{2(p-1)}+\delta^{\frac{8p}{d+2}}\parallel \varepsilon \parallel_{\dot H^2}^2.
\end{array}
$$
For $\delta$ small enough, the three previous equations give:
$$
\parallel e_0 a\mathcal Y-H\varepsilon +(a\mathcal Y+\varepsilon)^2 \parallel_{L^2}^2 \gtrsim a^2+\parallel \varepsilon \parallel_{\dot H^2}^2.
$$
In turn, injecting this estimate in the variation of energy formula \fref{eq:energie0} yields the identity \fref{eq:variation energie} we claimed.\\

\end{proof}

The $\dot{H}^1$ bound of Lemma \ref{lem:variation energie} is not enough to control the modulation equations of Lemma \ref{lem:modulation}, and we claim as a consequence of the coercivity bounds of Lemma \ref{lem:coercivite} higher order $\dot{H^2}, \dot{H}^3$ bounds which lock the dynamics.

\begin{lemma}[Higher order energy bounds] \label{lem:methode energie}
Let $d\geq 7$ and $I$ be a time interval containing $0$. There exists $0<\eta^*\ll 1$ such that if $u$ is trapped at distance $\eta$ for $0<\eta<\eta^*$ on $I$ then the following holds on $s(I)$.
\begin{itemize}
\item[(i)] \emph{$\dot H^1$ monotonicity:}
\be \label{eq:dotH1}
\frac{d}{ds}\left[\frac 1 2 \int \varepsilon H\varepsilon \right] \leq -\frac{1}{C} \int (H\varepsilon)^2+Ca^4 .
\ee
\item[(ii)] \emph{$\dot H^2$ monotonicity:} 
\be \label{eq:dotH2}
\begin{array}{r c l}
\frac{d}{ds} \left[\frac 1 2 \int (H \varepsilon)^2 \right] -\frac{\lambda_s}{\lambda} \int |H\varepsilon |^2 \leq -\frac{1}{C}\int H\varepsilon H^2\varepsilon +Ca^4+Ca^2\parallel \varepsilon \parallel^2_{\dot H^2}
\end{array}
\ee
\end{itemize}

\end{lemma}

\begin{remark} The notation $\int H^2\varepsilon H\varepsilon $ means $\int |\nabla H\varepsilon|^2-\int V |H\varepsilon|^2$, this later formula given by an integration by parts makes sense from Proposition \ref{pr:cauchy} whereas the first one does not, but we keep it to ease notations. One also has from Proposition \ref{pr:cauchy} that $u\in C^1((0,T),\dot H^1)$, hence the identity \fref{eq:dotH1} makes sense. For \fref{eq:dotH2}, the quantity $\int |H \varepsilon |^2$ is well defined from Proposition \ref{pr:cauchy} but this does not give its time differentiability. \fref{eq:dotH2} should then be understood as an abuse of notation for its integral version using a standard procedure of regularization of the nonlinearity.
\end{remark}

\begin{proof}[Proof of Lemma \ref{lem:methode energie}]

\noindent{\bf step 1} Proof of the $\dot H^1$ bound. One computes from \fref{eq:evolution} using the orthogonality conditions \fref{eq:orthogonalite}, \fref{eq:commutateur}, \fref{eq:ipp} and \fref{eq:ipp2}:
\be \label{eq:dotH1 expression}
\begin{array}{r c l}
\frac{d}{ds}\left[\frac 1 2 \int \varepsilon H\varepsilon \right] &= & -\int (H\varepsilon)^2+\int NL H\varepsilon+\frac{1}{2} \frac{\lambda_s}{\lambda} \int \varepsilon^2 (V+y.\nabla V) \\
&&-\int \varepsilon ^2 \frac{z_s}{\lambda}.\nabla V+a\frac{\lambda_s}{\lambda} \int H \varepsilon \Lambda \mathcal Y+a\int H\varepsilon \frac{z_s}{\lambda}.\nabla \mathcal Y 
\end{array}
\ee
and we now estimate each term in the right hand side. Using Young inequality, the estimate \fref{eq:NL3} on the nonlinearity, Sobolev embedding, interpolation, \fref{eq:estimation dotH1} and the fact that $\mathcal Y$ is exponentially decreasing one has for any $0<\kappa \ll 1$:
\bee
&& \left| \int NL H\varepsilon \right| \leq   \int \frac{\kappa (H\varepsilon)^2}{2}+ \frac{NL^2}{2\kappa} \lesssim  \frac{\kappa}{2} \int (H\varepsilon)^2+ \int \frac{a^4}{2\kappa}\mathcal Y^4 Q^{2(p-2)}+\frac{1}{2\kappa} \int \varepsilon^{2p} \\
&\lesssim & \frac{\kappa}{2} \int (H\varepsilon)^2+\frac{a^4}{2\kappa} +\frac{1}{2\kappa} \parallel \varepsilon \parallel_{\dot H^{\frac{2d}{d+2}}}^{2p} \lesssim  \frac{\kappa}{2} \int (H\varepsilon)^2+\frac{a^4}{2\kappa} +\frac{1}{2\kappa} \parallel \varepsilon \parallel_{\dot H^1}^{\frac{8p}{d+2}}\parallel \varepsilon \parallel_{\dot H^2}^2 \\
&\lesssim & \frac{\kappa}{2} \int (H\varepsilon)^2+\frac{a^4}{2\kappa} +\frac{\eta^{\frac{8p}{d+2}}}{2\kappa} \int (H\varepsilon)^2 .
\eee
Now using \fref{eq:modulation lambda}, \fref{eq:modulation z}, and the coercivity estimate \fref{eq:coerciviteA}:
$$
 \left| \frac{1}{2} \frac{\lambda_s}{\lambda} \int \varepsilon (V+y.\nabla V)\varepsilon -\int \varepsilon ^2 \frac{z_s}{\lambda}.\nabla V \right| \lesssim  \eta \int \frac{|\varepsilon|^2}{1+|y|^4} \lesssim \eta \int (H\varepsilon)^2 .
$$
Using Young inequality, the estimates \fref{eq:modulation lambda}, \fref{eq:modulation z}, and \fref{eq:estimation dotH1} one has for any $0<\kappa \ll 1$:
$$
\ba{r c l}
\left| a\frac{\lambda_s}{\lambda} \int H \varepsilon \Lambda \mathcal Y+a\int H\varepsilon \frac{z_s}{\lambda}.\nabla \mathcal Y  \right| & \lesssim & \frac{a^2}{\kappa} \left( \left| \frac{\lambda_s}{\lambda} \right|^2+\left| \frac{z_s}{\lambda} \right|^2\right)+\kappa \int (H\varepsilon)^2 \\
& \lesssim & \frac{a^6}{\kappa} +\left( \frac{\eta^2}{\kappa}+\kappa\right) \int (H\varepsilon)^2
\ea
$$
We now inject the three estimates above in the identity \fref{eq:dotH1 expression}. One can choose $\kappa $ small enough independently of $\eta^*$ and then $\eta^*$ small enough such that the desired identity \fref{eq:dotH1} holds.\\

\noindent{\bf step 2} Proof of the $\dot H^2$ bound. One computes first the following identity using the evolution equation \fref{eq:evolution}, the orthogonality \fref{eq:orthogonalite}, the identities \fref{eq:ipp}, \fref{eq:commutateur} and $H\nabla=\nabla H-\nabla V$:
\bea \label{eq:dotH2 expression}
\nonumber &&\frac{d}{ds} \left(\frac 1 2 \parallel H\varepsilon \parallel_{L^2}^2 \right)\\
\nonumber  & = & \frac{\lambda_s}{\lambda} \int H\varepsilon H(a\Lambda \mathcal Y+\Lambda \varepsilon)+\int H\varepsilon H(\frac{z_s}{\lambda}.\nabla(a\mathcal Y+\varepsilon))-\int \varepsilon H^3\varepsilon+\int H\varepsilon H(NL)\\
\nonumber &=& \frac{\lambda_s}{\lambda} \int |H\varepsilon |^2+\frac{\lambda_s}{\lambda} a \int H \varepsilon H\Lambda \mathcal Y-\int \varepsilon H^3\varepsilon+\int H\varepsilon H(NL)\\
&+& \frac{\lambda_s}{\lambda}\int H\varepsilon (2V+y.\nabla V)\varepsilon +a\int H\varepsilon H(\frac{z_s}{\lambda}.\nabla \mathcal Y)-\int H\varepsilon \frac{z_s}{\lambda}.\nabla V \varepsilon.
\eea
\noindent \emph{Estimate for the nonlinear term}. One first computes the influence of the nonlinear term. Performing an integration by parts we write:
\be \label{eq:bound NL expression}
\int H\varepsilon H(NL) = \int H \varepsilon VNL + \int \nabla H\varepsilon .\nabla (NL) .
\ee
The potential term is estimated directly as it has a very strong decay, using the estimate \fref{eq:NL3} for the nonlinear term, the coercivity \fref{eq:coercivite}, \fref{eq:estimation dotH1}, Sobolev embedding, interpolation and the fact that $\mathcal Y$ is exponentially decaying:
\bea \label{eq:bound NL 1}
\non && \left|  \int H(\varepsilon)VNL \right| \lesssim  \int \frac{|H\varepsilon||NL|}{1+|y|^4} \lesssim  \int \frac{|H\varepsilon|(|a|^2\mathcal Y^2Q^{p-2}+|\varepsilon|^p)}{1+|y|^4} \\
\non & \lesssim & \parallel \varepsilon \parallel_{\dot H^3}a^2+\left( \int \frac{|H\varepsilon|^2}{1+|y|^2}\right)^{\frac 1 2} \left( \int \frac{|\varepsilon|^{2p}}{1+|y|^6} \right)^{\frac 1 2} \\
\non & \lesssim & \parallel \varepsilon \parallel_{\dot H^3}a^2+\left(\int H\varepsilon H^2\varepsilon \right)^{\frac 1 2} \left( \parallel |\varepsilon|^{2p} \parallel_{L^{\frac{d}{d-2}}}\parallel \frac{1}{1+|y|^{6}}\parallel_{L^{\frac d 2}} \right)^{\frac 1 2} \\
\non & \lesssim & \parallel \varepsilon \parallel_{\dot H^3}a^2+\left(\int H\varepsilon H^2\varepsilon \right)^{\frac 1 2} \parallel \varepsilon \parallel_{L^{\frac{2d(d+2)}{(d-2)^2}}}^p \\
\non&\lesssim &  \parallel \varepsilon \parallel_{\dot H^3}a^2+\left(\int H\varepsilon H^2\varepsilon \right)^{\frac 1 2} \parallel \varepsilon \parallel_{\dot H^{3-\frac{8}{d+2}}}^p  \\
\non & \lesssim & \parallel \varepsilon \parallel_{\dot H^3}a^2+\left(\int H\varepsilon H^2\varepsilon \right)^{\frac 1 2} \parallel \varepsilon \parallel_{\dot H^1}^{\frac{4}{d+2}p}\parallel \varepsilon \parallel_{\dot H^3}^{\frac{d-2}{d+2}p} \\
\non&\lesssim & \parallel \varepsilon \parallel_{\dot H^3}a^2+\left(\int H\varepsilon H^2\varepsilon \right)^{\frac 1 2} \eta^{\frac{4}{d+2}p} \parallel \varepsilon \parallel_{\dot H^3}\\
& \lesssim & \parallel \varepsilon \parallel_{\dot H^3}a^2+\eta^{\frac{4}{d+2}p}\int H\varepsilon H^2\varepsilon . 
\eea
Now for the other term one first computes the first derivatives of the nonlinear term:
\be \label{eq:bound NL 2 expression}
\begin{array}{r c l}
\nabla (NL)&=&p\left( |Q+a\mathcal Y+\varepsilon|^{p-1}-Q^{p-1} \right)\nabla (a\mathcal Y+\varepsilon)\\
&&+p\left(|Q+a\mathcal Y+\varepsilon|^{p-1}-Q^{p-1}-(p-1)Q^{p-2}(a\mathcal Y+\varepsilon) \right)\nabla Q \\
&=:& A_1+A_2 .
\end{array}
\ee
We estimate the first term pointwise using the estimates on the nonlinearity \fref{eq:NL4} and \fref{eq:NL5}:
\bee
|A_1|&=&p\left| \left( |Q+a\mathcal Y+\varepsilon|^{p-1}-Q^{p-1} \right)\nabla (a\mathcal Y+\varepsilon) \right| \\
&\leq & \left( Q^{p-2}|a| \mathcal Y +|\varepsilon|^{p-1}|\right)|a||\nabla \mathcal Y|+\left( Q^{p-2}|a| \mathcal Y +|\varepsilon|^{p-1}|\right)|\nabla \varepsilon | \\
&\lesssim & \left( Q^{p-2}|a| \mathcal Y +|\varepsilon||\right)|a||\nabla \mathcal Y|+\left( Q^{p-2}|a| \mathcal Y +|\varepsilon|^{p-1}|\right)|\nabla \varepsilon | \\
&\lesssim & Q^{p-2}a^2 \mathcal Y |\nabla \mathcal Y| +Q^{p-2}|\varepsilon| |a| |\nabla \mathcal Y |+Q^{p-2}|a| \mathcal Y |\nabla \varepsilon |+|\varepsilon|^{p-1}|\nabla \varepsilon |
\eee
Hence, using the coercivity \fref{eq:coercivite H3} for the second and third terms, and Sobolev embedding plus \fref{eq:estimation dotH1} for the last one:
\bea
 \label{eq:bound NL 2 1}
\nonumber \int |A_1|^2 &\lesssim & \int Q^{2(p-2)}a^4 \mathcal Y^2 |\nabla \mathcal Y|^2 +Q^{2(p-2)}|\varepsilon|^2 a^2 |\nabla \mathcal Y |^2\\
\nonumber &+&\int Q^{2(p-2)}a^2 \mathcal Y^2 |\nabla \varepsilon |^2+|\varepsilon|^{2(p-1)}|\nabla \varepsilon |^2 \\
\nonumber &\lesssim & a^4+a^2 \parallel \varepsilon \parallel_{\dot H^3}^2+\parallel | \nabla \varepsilon |^2 \parallel_{L^{\frac{d}{d-4}}} \parallel |\varepsilon |^{2(p-1)}\parallel_{L^{\frac{d}{4}}} \\
\nonumber &\lesssim & a^4+a^2 \parallel \varepsilon \parallel_{\dot H^3}^2+\parallel \nabla \varepsilon \parallel_{L^{\frac{2d}{d-4}}}^2 \parallel \varepsilon \parallel_{L^{\frac{2d}{d-2}}}^{2(p-1)} \\
\nonumber &\lesssim & a^4+a^2 \parallel \varepsilon \parallel_{\dot H^3}^2+\parallel  \varepsilon \parallel_{\dot H^3}^2 \parallel \varepsilon \parallel_{\dot H^1}^{2(p-1)}\\
&  \lesssim  & a^4+(a^2+\eta^{2(p-1)}) \parallel \varepsilon \parallel_{\dot H^3}^2
\eea
We now turn to the second term $A_2$. We use the pointwise estimate for the nonlinearity \fref{eq:NL6}, yielding:
$$
|A_2|\lesssim Q^{p-2-\frac{2}{d-2}}|\varepsilon|^{1+\frac{2}{d-2}}|\nabla Q|+a^2\mathcal Y^2Q^{p-3}|\nabla Q| .
$$
We then estimate $A_2$ using Sobolev embedding, the coercivity estimate \fref{eq:coercivite H3} and \fref{eq:estimation dotH1}:
\be \label{eq:bound NL 2 2}
\begin{array}{r c l}
\int |A_2|^2 & \lesssim &  a^4 +\int \frac{|\varepsilon |^{\frac{2d}{d-2}}}{1+|y|^6} \lesssim  a^4+\parallel |\varepsilon |^{\frac{2d}{d-2}} \parallel_{L^{\frac{d}{d-4}}}\parallel \frac{1}{1+|y|^{6}}\parallel_{L^{\frac{d}{4}}} \\
&\lesssim & a^4+\parallel \varepsilon \parallel_{L^{\frac{2d^2}{(d-2)(d-4)}}}^{\frac{2d}{d-2}} \lesssim  a^4+\parallel \varepsilon \parallel_{\dot H^{3-\frac 4 d}}^{\frac{2d}{d-2}} \\
&\lesssim & a^4+\parallel \varepsilon \parallel_{\dot H^1}^{\frac{2d}{d-2}\frac{2}{d}}\parallel \varepsilon \parallel_{\dot H^3}^{\frac{2d}{d-2}\frac{d-2}{d}}  \lesssim a^4+\eta^{\frac{4}{d-2}}\parallel \varepsilon \parallel_{\dot H^3}^2.
\end{array}
\ee
Putting together the two estimates \fref{eq:bound NL 2 1} and \fref{eq:bound NL 2 2} we have found for the two terms in \fref{eq:bound NL 2 expression} yields:
$$
\int |\nabla NL |^2\lesssim a^4+(a^2+\eta^{\frac{4}{d-2}})\parallel \varepsilon \parallel_{\dot H^3}^2.
$$
We now come back to the original nonlinear term \fref{eq:bound NL expression} we had to treat, and inject the above estimate and the estimate \fref{eq:bound NL 1} we just found for each term, yielding:
\bea
 \label{eq:bound NL}
\nonumber&& | \int H\varepsilon H(NL)|\lesssim a^2 \parallel \varepsilon \parallel_{\dot H^3} +(a+\eta^{\frac{2}{d-2}})\parallel \varepsilon \parallel_{\dot H^3}^2\\
&\lesssim& \frac{a^4}{\kappa } +(a+\eta^{\frac{2}{d-2}}+\kappa)\parallel \varepsilon \parallel_{\dot H^3}^2
\eea
where we used Young inequality with a small parameter $0< \kappa \ll1$ to be chosen later on.\\

\noindent\emph{Remainders from scale and space change}. We put some upper bounds on the following term appearing in \fref{eq:dotH2 expression}. From \fref{eq:modulation lambda} and the coercivity estimate \fref{eq:coercivite H3}:
\bea \label{eq:bound remainderlambdas}
\non \left |\frac{\lambda_s}{\lambda}\int H\varepsilon (2V+y.\nabla V)\varepsilon \right|&  \lesssim & \eta \int \frac{|\varepsilon||H\varepsilon|}{1+|y|^4} \leq  \eta \left(\int \frac{|\varepsilon|^2}{1+|y|^6}\right)^{\frac 1 2}\left(\int \frac{|H\varepsilon|^2}{1+|y|^2}\right)^{\frac 1 2} \\
&\lesssim &  \eta \int H\varepsilon H^2\varepsilon .
\eea
Similarly, from \fref{eq:modulation z}:
\be \label{eq:bound remainderzs}
\left| \int H\varepsilon \frac{z_s}{\lambda}.\nabla V \varepsilon \right| \lesssim \eta \int \frac{|H\varepsilon |\varepsilon}{1+|x|^5}  \lesssim \eta^2 \int H\varepsilon H^2\varepsilon .
\ee
Now, from the fact that $\mathcal Y$ decays exponentially fast and the coercivity estimate \fref{eq:coercivite H3}, one has for any $0<\kappa\ll1$ to be chosen later:
\bea
 \label{eq:bound force}
&&\left| \frac{\lambda_s}{\lambda} a \int H \varepsilon H\Lambda \mathcal Y+a\int H\varepsilon H(\frac{z_s}{\lambda}.\nabla \mathcal Y)\right|\\
\nonumber  & \lesssim & \frac{a^2}{\kappa}\left(\left| \frac{\lambda_s}{\lambda} \right|^2+\left| \frac{z_s}{\lambda} \right|^2 \right)+\kappa \int H\varepsilon H^2\varepsilon .
\eea

\noindent\emph{End of the proof of the energy identity}. We come back to the identity \fref{eq:dotH2 expression} and inject all the bounds we found so far, \fref{eq:bound NL}, \fref{eq:bound remainderlambdas}, \fref{eq:bound remainderzs} and \fref{eq:bound force}. Using the two different types of estimates for the modulation of $\lambda$ and $z$, \fref{eq:modulation lambda} and \fref{eq:modulation z}, one sees that there exists $\kappa$ and $\eta^*$ such that the desired bound \fref{eq:dotH2} holds.
\end{proof}


\subsection{No type II blow up near the soliton}


A spectacular consequence of Lemma \ref{lem:modulation}, Lemma \ref{lem:variation energie} and Lemma \ref{lem:methode energie} is the uniform control of the scale for trapped solutions.

\begin{lemma}[Non degeneracy of the scale for trapped solutions] \label{lem:lambda}

Let $d\geq 7$ and $I$ be a time interval containing $0$. There exists $0<\eta^*\ll 1$ such that if $u$ is trapped at distance $\eta$ for $0<\eta<\eta^*$ on $I$, then:
\be \label{eq:controle lambda}
\lambda (t)=\lambda (0)(1+O(\eta ))
\ee

\end{lemma}

In particular, this rules out type II blow up near the solitary wave $Q$ as in \cite{NT2} for the large homotopy number corotational harmonic heat flow.

\begin{proof}[Proof of Lemma \ref{lem:lambda}]

We reason in renormalized time. The modulation equation \fref{eq:modulation lambda} for $\lambda$ can be written as:
$$
\left| \frac{d}{ds}\left[ \text{log}(\lambda)+O(\parallel \varepsilon (s) \parallel_{\dot H^1}) \right] \right| \lesssim a^2+\parallel \varepsilon \parallel_{\dot H^2}^2
$$
on $s(I)$. We now reintegrate it in time:
$$
\left| \text{log}\left(\frac{\lambda (s)}{\lambda (0)}\right)+\parallel \varepsilon (s) \parallel_{\dot H^1}+O(\parallel \varepsilon (0) \parallel_{\dot H^1}) \right| \lesssim \int_0^s a^2+\parallel \varepsilon \parallel_{\dot H^2}^2
$$ 
and inject the estimate \fref{eq:variation energie} for the right hand side coming from the variation of energy and \fref{eq:estimation dotH1}, yielding:
$$
\lambda (s)=\lambda (0) e^{O(\eta)} =\lambda (0) (1+O(\eta))
$$
and \eqref{eq:controle lambda} is proved.
\end{proof}

\begin{remark} The estimate \fref{eq:controle lambda} implies that for solutions trapped at distance $\eta$ on $I$, the original time variable $t$ is equivalent to the renormalized time variable $s$, namely there exists $c>0$ such that for all $t\in I$:
\be \label{eq:equivalence temps}
\lambda (0)(1-c\eta)t\leq s \leq \lambda (0)(1+c\eta)t .
\ee
\end{remark}


\section{Existence and uniqueness of minimal solutions} \la{sec:min}


This section is devoted to the proof of existence and uniqueness of the minimal elements which are trapped backwards in time near the soliton manifold, together with the complete description of their forward behaviour. 


\subsection{Existence} 


The existence of $Q^\pm$ follows from a simple brute force fixed point argument near $-\infty$, while the derivation of their forward (exit) behavior follows from the maximum principle. We may here relax the dimensional assumption and assume $d\ge 3$ only.

\begin{proposition}[Existence of $Q^\pm$] \label{pr:minimal}
Let $d\geq 3$. There exists two strictly positive, $\mathcal C^{\infty}$ and radial solutions of \fref{eq:NLH}, $Q^+$ and $Q^-$, defined on $(-\infty,0]\times \mathbb R^d$, such that 
$$\lim_{t\to-\infty}\|Q^{\pm}-Q\|_{\dot{H}^1}=0.$$ Moreover:\\
{\em 1. Trapping near $Q$ for $t\leq 0$}: there holds the expansion for $t\in (-\infty,0]$:
\be \label{eq:expansion Qpm}
Q^{\pm}(t)=Q\pm \epsilon e^{e_0 t} \mathcal Y+v, \ \ \parallel v \parallel_{\dot H^1}+\parallel v \parallel_{L^{\infty}}\lesssim \epsilon^2 e^{2e_0 t}
\ee
for some $0< \epsilon \ll 1$.\\
{\em 2. Forward (exit)}: $Q^+$ explodes forward in finite time in the ODE type I blow up regime, while $Q^-$ is global and dissipates $\lim_{t\to+\infty}\|Q^-\|_{\dot{H}^1}=0.$
\end{proposition}

\begin{proof}[Proof of Proposition \ref{pr:minimal}]

We prove the existence of $Q^+$ and $Q^-$ via a compactness argument. Let $0<\epsilon \ll 1$ be a small enough strictly positive constant. We look at two solutions $u_n^{\pm}$ of \fref{eq:NLH} with initial data at the initial time $t_0(n)=-n$:
\be \la{Q+:eq:def un+}
u^{\pm}_n(-n)=Q\pm \epsilon e^{-ne_0}\mathcal Y .
\ee
The sign $+$ (resp. $-$) will give an approximation of $Q^+$ (resp. $Q^-$). The results and techniques employed in this step being exactly the same for the two cases, we focus on the $+$ sign case. As $u^{\pm}_n(-n)$ is radial, $u^{\pm}_n$ is radial for all times.\\

\noindent {\bf step 1} Forward propagation of smallness. We claim that there exists constants $C_a, \  C_2, \ C_{\infty}>0$ such that for $\epsilon$ small enough for any $n\in \mathbb N$ the solution is at least defined on $[-n,0]$ and can be written on this interval under the form:
\be \la{Q+:eq:id un+}
u^+_n(t)=Q+ \epsilon (1+a) e^{te_0}\mathcal Y+\epsilon v, \ \ a:[-n,0]\rightarrow \mathbb R, \ \ v\in L^2\cap L^{\infty}, \ \ v\perp \mathcal Y
\ee
where the corrections $a$ and $v$ satisfy the following bounds on $[-n;0]$:
\be \label{eq:bounds bootstrap}
\parallel a \parallel_{L^{\infty}}\leq \epsilon C_a e^{e_0 t}, \ \ \parallel v \parallel_{L^2}\leq \epsilon C_2 e^{2 t e_0}, \ \ \parallel v \parallel_{L^{\infty}}\leq \epsilon C_{\infty} e^{2te_0} .
\ee
Indeed, the decomposition and the regularity of the solution is ensured by the fact that one can solve the Cauchy problem for $H^1\cap L^{\infty}$ perturbations of $Q$ arguing like for the proof of Proposition \ref{pr:cauchy}. We now prove \eqref{eq:bounds bootstrap} using a bootstrap argument. Let $\mathcal T\subset [-n,0]$ be the set of times $-n\leq t \leq 0$ such that \fref{eq:bounds bootstrap} holds on $[-n,t]$. $\mathcal T$ is not empty as it contains $-n$, and it is closed from a continuity argument. We now show that it is open, implying $\mathcal T=[-n,0]$ using connectedness. To do so, we are going to show that under some compatibility conditions on the constants $C_a, \ C_2, \ C_{\infty}$ and for $\epsilon$ sufficiently small, the inequalities in \fref{eq:bounds bootstrap} are strict in $\mathcal T$, implying that this latter is open by continuity. From \fref{eq:NLH} the time evolution of $v$ is given by:
\be \la{Q+:eq:vt}
v_t+Hv=-a_te^{te_0}\mathcal Y+NL, \ \ NL:= \frac{1}{\epsilon}\left[f(u_n^+)-p\epsilon Q^{p-1}(e^{te_0}(1+a)\mathcal Y+ v) \right].
\ee
Using the bound \fref{eq:NL3} on the nonlinearity, the bootstrap bounds \fref{eq:bounds bootstrap}, the fact that $\mathcal Y$ decays exponentially and interpolation, one gets for $q\in [2,+\infty]$ and $t\in \mathcal T$:
\be \la{Q+:eq:bd NL}
\ba{r c l}
\parallel NL \parallel_{L^q}&\leq & C \para Q^{p-2}\epsilon e^{2e_0t}(1+a)^2\mathcal Y^2+\epsilon^{p-1}|v|^{p} \para_{L^q} \\
&\leq & C\epsilon e^{2e_0t}+CC_a^2\epsilon^2e^{4e_0t}+C\epsilon^{2p-1}C_2^{\frac 2 q}C_{\infty}^{\frac{q-2}{q}} \\
&\leq & C\epsilon e^{2te_0}
\ea
\ee
for a constant $C$ independent of the others, for any choice of $C_a$, $C_2$ and $C_{\infty}$ if $\epsilon$ is chosen small enough, as $t\leq 0$ and $p>1$.\\

\noindent \emph{Bound for $a$}. Projecting \fref{Q+:eq:vt} onto $\mathcal Y$, using the orthogonality $v\perp \mathcal Y$ and the above estimate then yields that for $t\in \mathcal T$:
\be \la{Q+:eq:bd at}
|a_t|\leq C\epsilon e^{e_0t}
\ee
for $C$ independent of $C_a, \ C_2, \ C_{\infty}, \ \epsilon$. Reintegrated in time this gives:
\be \la{Q+:eq:bd a}
|a(t)|\leq |a(-n)|+C\epsilon e^{e_0t}\leq C\epsilon e^{e_0t}<C_a e^{e_0t} \ \ \text{if} \ \ C<C_a
\ee
for $C$ independent of $C_a, \ C_2, \ C_{\infty}, \ \epsilon$ for any $t\in \mathcal T$ as $a(-n)=0$ from \fref{Q+:eq:def un+} and \fref{Q+:eq:id un+}.\\

\noindent \emph{$L^2$ bound for $v$}. We multiply \fref{Q+:eq:vt}, by $v$ and integrate using \fref{Q+:eq:bd NL}, \fref{Q+:eq:bd at}, Cauchy-Schwarz inequality and the fact $\int vHv\geq 0$ from Proposition \ref{pr:H} to obtain the following energy estimate on $\mathcal T$:
$$
\ba{r c l}
\frac{d}{dt}\parallel v \parallel_{L^2}^2 & = & -2\int vHv+2\int v(-a_te^{te_0}\mathcal Y+NL) \leq C\para v \para_{L^2}(\para NL \para_{L^2}+|a_t|) \\
&\leq & CC_2 \epsilon e^{e_0t} .
\ea
$$
Reintegrated in time, as $v(-n)=0$ from \fref{Q+:eq:def un+} and \fref{Q+:eq:id un+} one obtains that for $t\in \mathcal T$:
\be \la{Q+:eq:bd v L2}
\parallel v \parallel_{L^2} \leq C\sqrt{C_2}\epsilon e^{e_0t}<C_2\epsilon e^{e_0t} \ \ \text{if} \ \ C<C_2
\ee 
for $C$ independent of $C_a, \ C_2, \ C_{\infty}, \ \epsilon$.\\

\noindent\emph{$L^{\infty}$ bound for $v$}. Using Duhamel formula one has:
\be \la{Q+:id v}
v(t)=\int_{-n}^t K_{t-t'}*(-a_te^{t'e_0}\mathcal Y+NL+Vv)dt',
\ee
where $K_t$ is defined by \fref{eq:def Kt}. For the first two terms, thanks to the bounds \fref{Q+:eq:bd NL} and \fref{Q+:eq:bd at} on $a_t$ and $NL$ and the fact that $K_t*:L^{\infty}\rightarrow L^{\infty}$ is unitary, one gets:
$$
\parallel \int_{-n}^t K_{t-t'}*(-a_te^{t'e_0}\mathcal Y+NL) \parallel_{L^{\infty}}\leq \int_{-n}^t C\epsilon e^{2t'e_0} \leq C\epsilon e^{2te_0}
$$
for $C$ independent of $C_a, \ C_2, \ C_{\infty}, \ \epsilon$ and $t\in \mathcal T$. Now for the last term, for $\delta=\frac{1}{2\parallel V \parallel_{L^{\infty}}}$, making the abuse of notation $t-\delta=\text{min}(t-\delta,-n)$, using H\"older inequalities, \fref{eq:bd Kt}, the fact that $\para K_t\para_{L^1}=1$ for all $t\geq 0$ and interpolation one computes for $t\in \mathcal T$:
\bee
&& \parallel \int_{-n}^t K_{t-t'}*(Vv)dt' \parallel_{L^{\infty}} \\
& \leq & \int_{-n}^{t-\delta} \parallel K_{t-t'}*(Vv)\parallel_{L^{\infty}}dt' +\int_{t-\delta}^{t} \parallel K_{t-t'}*(Vv)\parallel_{L^{\infty}}dt' \\
& \leq &  \int_{-n}^{t-\delta} \parallel K_{t-t'}\parallel_{L^{\frac{d}{d-1}}}\parallel Vv \parallel_{L^d}dt' +\int_{t-\delta}^{t} \parallel K_{t-t'}\parallel_{L^{1}}\parallel Vv\parallel_{L^{\infty}}dt' \\
& \leq &  \int_{-n}^{t-\delta}\frac{C}{|t-t'|^{\frac 1 2}}\parallel v \parallel_{L^{2d}}\parallel V \parallel_{L^{2d}} dt' +\int_{t-\delta}^{t} \parallel V\parallel_{L^{\infty}} \parallel v\parallel_{L^{\infty}}dt' \\
& \leq &  \int_{-n}^{t-\delta} \frac{C}{|t-t'|^{\frac 1 2}} \epsilon e^{2e_0t'} C_2^{\frac 1 d} C_{\infty}^{1-\frac 1 d}dt'  +\int_{t-\delta}^{t} \parallel V\parallel_{L^{\infty}} \epsilon C_{\infty}e^{2e_0t'}dt' \\
& \leq & C\epsilon C_2^{\frac 1 d}C_{\infty}^{1-\frac 1 d} e^{2e_0t}+\delta \parallel V\parallel_{L^{\infty}} \epsilon C_{\infty}e^{2e_0t} \leq  \epsilon e^{2e_0t} (CC_2^{\frac 1 d}C_{\infty}^{1-\frac 1 d} +\frac 1 2 C_{\infty}).
\eee
The three previous equations then imply that for $t\in \mathcal T$:
\be \la{Q+:eq:bd v Linfty}
\parallel v \parallel_{L^{\infty}} \leq \epsilon e^{2e_0t} (C+CC_2^{\frac 1 d}C_{\infty}^{1-\frac 1 d} +\frac 1 2 C_{\infty})<\epsilon C_{\infty}e^{2e_0t}C 
\ee
 provided $C+CC_2^{\frac 1 d}C_{\infty}^{1-\frac 1 d}< C_{\infty}$.\\

\noindent\emph{Conclusion}. The estimates \fref{Q+:eq:bd a}, \fref{Q+:eq:bd v L2} and \fref{Q+:eq:bd v Linfty} ensure the existence of $\epsilon_0>0$, $C_a,C_2,C_{\infty}>0$ such that for $0<\epsilon\leq \epsilon_0$, \fref{eq:bounds bootstrap} is strict on $\mathcal T$ hence this latter is open, and \eqref{eq:bounds bootstrap} is proved.\\

\noindent{\bf step 2} Propagation of smoothness. We claim that one has the following additional bounds on $[-n+1,0]$:
\be \la{Q+:eq:bd v W2infty}
\parallel  v \parallel_{W^{2,\infty}} +\parallel  v \parallel_{\dot H^2} \leq  \epsilon e^{te_0},
\ee
\be \la{Q+:eq:bd v holder}
\parallel  \nabla^2 v \parallel_{C^{0,\frac 1 4}(\mathbb R^d\times [-n+1,0])}+\parallel  \partial_t v \parallel_{L^{\infty}(\mathbb R^d\times [-n+1,0])}+\parallel  \partial_t v \parallel_{C^{0,\frac 1 4}(\mathbb R^d\times [-n+1,0])}\leq C
\ee
and that for all $t\in[-n,0]$:
\be \la{Q+:eq:bd v dotH1}
\para v \para_{\dot H^1}\leq C \epsilon e^{2e_0t}.
\ee
where $C^{0,\frac 1 4}$ denotes the H\"older $\frac 1 4$-norm, for $C$ independent of $n$ and $\epsilon$. The first two bounds \fref{Q+:eq:bd v W2infty} and \fref{Q+:eq:bd v holder} are direct consequences of \fref{Q+:eq:id un+}, \fref{eq:bounds bootstrap} and the parabolic estimates \fref{para:eq:bd v W2infty} and \fref{para:eq:bd v holder} from Lemma \ref{para:lem:para}. To prove the last bound \fref{Q+:eq:bd v dotH1} one computes using \fref{Q+:id v}, \fref{Q+:eq:bd at}, \fref{eq:bounds bootstrap}, \fref{eq:NL1}, Young and H\"older inequalities that for $t\in [-n,0]$:
\bee
 \para v(t)\para_{\dot H^1} & \leq &  \int_{-n}^t \para  K_{t-t'}\para_{L^1}\para a_te^{t'e_0}\mathcal Y\para_{\dot H^1}dt'\\
&&+ \int_{-n}^t \para \nabla K_{t-t'}\para_{L^1} (\para NL\para_{L^2}+\para Vv\para_{L^2})dt'\\
& \leq &  C\epsilon \int_{-n}^t \left(e^{2e_0t'}+\frac{ e^{2e_0 t'}}{|t-t'|^{\frac 1 2}}+\frac{e^{2e_0t'}}{|t-t'|^{\frac 1 2}}\right)dt'  \leq C\epsilon e^{2e_0t}
\eee
for $C>0$ independent of the other constants, ending the proof of the claim.\\

\noindent{\bf step 3} Maximum principle for $u_n^{\pm}$. We claim that 
\be \la{min:prop un}
\left|\begin{array}{ll} u_n^+\geq Q, \ \ \partial_t(u_n^+)\geq 0\\ 0\leq u_n^-\leq Q \ \ \text{and} \ \ \partial_t(u_n^-)\leq 0\end{array}\right.
\ee
We prove it for $u_n^+$, the proof being similar for $u_n^-$. This is true at initial time $-n$, because $\mathcal Y>0$ and $\epsilon>0$, implying $u_n^+(-n)-Q=\epsilon e^{-ne_0}\mathcal Y>0$ and 
$$
\begin{array}{r c l}
\partial_t (u_n^+)(-n) & = & -H(\epsilon e^{-ne_0}\mathcal Y)+[f(Q+\epsilon e^{-n e_0}\mathcal Y)-f(Q)-f'(Q)\epsilon e^{-n e_0}\mathcal Y] \\
& = & e_0\epsilon e^{-ne_0}\mathcal Y+[f(Q+\epsilon e^{-n e_0}\mathcal Y)-f(Q)-f'(Q)\epsilon e^{-n e_0}\mathcal Y] \\
&>&0
\end{array}
$$
as the nonlinearity $f$ is strictly convex on $[0,+\infty)$ from $p>1$. Therefore, from the maximum principle, see Lemma \ref{para:lem:max}, one has that $u_n^+\geq 0$ and $\partial_t u_n^+\geq 0$ on $[-n,0]$. \\

\noindent {\bf step 4} Compactness. From \fref{eq:bounds bootstrap}, \fref{Q+:eq:bd at}, \fref{Q+:eq:bd v W2infty}, \fref{Q+:eq:bd v holder}, Arzela-Ascoli theorem and a diagonal argument, there exist subsequences of $(u_n^+)_{n\in \mathbb N}$ and $(u_n^-)_{n\in \mathbb N}$ that converge in $C^{1,2}_{\text{loc}}((-\infty,0]\times \mathbb R^d)$ toward some functions that we call $Q^+$ and $Q^-$ respectively. The equation \fref{eq:NLH} then passes to the limit using \eqref{Q+:eq:bd v W2infty}, \eqref{Q+:eq:bd v holder}, implying that $Q^+$ and $Q^-$ are also solutions of \fref{eq:NLH} on $(-\infty,0]$. For each time $t\le 0$, the $L^{\infty}$ bound in \fref{eq:bounds bootstrap} and \fref{Q+:eq:bd v W2infty} passes to the limit via pointwise convergence, and the $\dot H^1$ bound \fref{Q+:eq:bd v dotH1} passes to the limit via lower semi- continuity of the $\dot H^1$ norm. This implies that $Q^+$ and $Q^-$ can be decomposed according to:
\be \la{min:id Qpm}
Q^{\pm}=Q\pm \epsilon e^{te_0}\mathcal Y+a \epsilon e^{te_0}\mathcal Y+\epsilon v
\ee
with $v\perp \mathcal Y$ satisfying:
\be \label{eq:bound Qpm}
|a|\lesssim \epsilon e^{2te_0}, \ \ \parallel v \parallel_{\dot H^1} +\parallel v \parallel_{L^{\infty}} \lesssim \epsilon e^{2te_0}.
\ee
Moreover, as $u^{\pm}_n$ is radial, so is $Q^\pm$.\\

\noindent{\bf step 5} $Q^+$ blows up forward Type I. The estimate \fref{min:prop un} ensures $Q^+(t)\geq Q(t)$ on the whole space-time domain $(-\infty,0]\times \mathbb R^d$. This then propagates forward in time according to the maximum principle, see Lemma \ref{para:lem:max}, giving that $Q^+\geq Q$ on $[0,T)$ where $T$ denotes the maximal time of existence of $Q^+$. We recall that $\mathcal Y>0$, $\int \mathcal Y=1$, and hence since the function $g:[0,+\infty)\rightarrow \mathbb R$ defined by $g(x)=(Q+x)^p-Q^p-pQ^{p-1}x$ is convex from $p>1$, Jensen inequality implies:
$$
g \left( \int (Q^+-Q)\mathcal Y\right) \leq \int g(Q^+-Q)\mathcal Y .
$$
This gives the following polynomial lower bound for the derivative of the component along $\mathcal Y$:
$$
\begin{array}{r c l}
\partial_t \left( \int (Q^+-Q)\mathcal Y \right) &=& e_0\int (Q^+-Q)\mathcal Y+ \int g(Q^+-Q) \mathcal Y \\
&\geq & e_0\int (Q^+-Q)\mathcal Y+g \left( \int (Q^+-Q)\mathcal Y\right) \\
\end{array}
$$
As this quantity is strictly positive at time $0$ from \fref{min:id Qpm} and \fref{eq:bound Qpm}, this implies that $Q^+$ blows up in finite time because $g(x)\sim x^p$ as $x\rightarrow +\infty$. Moreover, from Theorem 1.7 in \cite{MaM2}, $u >0$ implies that the blow up is of type $1$. Hence $Q^+$ blows up with a type I blow up forward in time.\\

\noindent{\bf step 6} $Q^-$ dissipates forward. As $\para u_n^-(0)- Q\para_{L^{\infty}}\geq \frac{\epsilon}{2}$ from \fref{Q+:eq:id un+} and \fref{eq:bounds bootstrap}, $0\leq u_n^-(0)\leq Q$ and $\partial_tu_n^-(0)\leq 0$, in the limit one obtains that $Q^-\neq Q$, $Q^-\neq 0$, $\partial_t Q^-(0)\leq 0$ and $0\leq Q^-(0)\leq Q$. Using the maximum principle for $Q^-$ and $\partial_t Q^-$, see Lemma \ref{para:lem:max}, one has that $0\leq Q^-\leq Q$ and $\partial_t Q^-\leq 0$ for all times $t\in [0,T)$ where $T$ is the maximal time of existence of $Q^-$. The $L^{\infty}$ Cauchy theory then ensures that $Q^-$ is a global solution. As $0\leq Q^- \leq Q$ on $(0,+\infty)$, from the parabolic estimates \fref{para:eq:bd v W2infty} and \fref{para:eq:bd v holder} one deduces that for any $t>1$, 
\be \label{eq:bound Q-}
\parallel Q^- \parallel_{W^{2,\infty}}+\parallel \partial_t Q^-\para_{L^{\infty}} \lesssim 1 ,
\ee
\be \label{eq:bound Q-2}
\parallel \nabla^2 Q^- \parallel_{C^{0,\frac 1 4}([t,t+1]\times \mathbb R^d)}+\parallel \partial_t Q^-\para_{C^{0,\frac 1 4}([t,t+1]\times \mathbb R^d)} \lesssim 1 
\ee
where $C^{0,\frac 1 4}$ denotes the H\"older norm. We define $u_{\infty}(x):=\underset{t\rightarrow +\infty}{\text{lim}}Q^-(t,x)$ which exists as $\partial_t Q^-\leq 0$ and $Q^-\geq 0$, satisfies $0\leq u_{\infty}\leq Q$, $u_{\infty}\neq Q$ and is radial. For $n\in \mathbb N$ let the sequence of functions 
\be \la{min:def vn}
v_n(t,x)=Q^-(n+t,x), \ \ (t,x)\in (0,1)\times \mathbb R^d.
\ee
As $Q^-$ is decreasing, for any $0\leq t_1\leq t_2\leq 1$ and $x\in \mathbb R^d$ there holds 
$$
Q(n+t_1,x)=v_n(t_1,x)\geq v_n(t_2,x)=Q(n+t_2,x)
$$
which implies:
$$
\underset{n\rightarrow +\infty}{\text{lim}} v_n(t_1,x) =\underset{n\rightarrow +\infty}{\text{lim}} v_n(t_2,x) = \underset{t\rightarrow +\infty}{\text{lim}} Q^-(t,x)= u_{\infty}(x),
$$
meaning that $v_n$ converges to the constant in time function $u_{\infty}$. From its definition \fref{min:def vn} and the bounds \fref{eq:bound Q-} and \fref{eq:bound Q-2}, using Arzela-Ascoli theorem, $u_{\infty}\in W^{2,\infty}$ and $v_n\rightarrow u_{\infty}$ in $C^{1,2}_{\text{loc}}((0,1)\times \mathbb R^d)$. From \fref{min:def vn}, $v_n$ solves $\partial_t v_n=\Delta v_n+|v_n|^{p-1}v_n$, and therefore at the limit one obtains:
$$
0=\partial_t u_{\infty}=\Delta u_{\infty}+|u_{\infty}|^{p-1}u_{\infty},
$$
$u_{\infty}$ is consequently a stationary solution of \fref{eq:NLH}. From the classification of all the radial solutions \fref{eq:elliptique}, one has that either $u_{\infty}=Q_{\lambda}$ for some $\lambda>0$, or $u_{\infty}=0$. For $\lambda_1>\lambda_2>0$ from the formula \fref{eq:def Q} one sees that the radial $Q_{\lambda_1}$ and $Q_{\lambda_2}$ must intersect at some radius $r^*$ and that one is strictly above the other for $0<r<r^*$ and conversely for $r>r^*$. Since $u_{\infty}<Q$ necessarily $u_{\infty}=0$. One has proven that for any $x\in \mathbb R^d$, $Q ^-(t,x)\rightarrow 0$.

\noindent One notices that as $Q^-\rightarrow 0$ and $0<Q^-<Q$, the nonlinear term $(Q^-)^p$ is in every Lebesgue space $L^p$ for $p\geq 1$ in which it then converges to zero as $t\rightarrow +\infty$ from Lebesgue's dominated convergence theorem. For any $T>0$ and $t>T$, we write using Duhamel formula and Young inequality:
$$
\begin{array}{r c l}
\nabla(Q^-) & = & \nabla(K_t *Q^-(0))+\int_0^t (\nabla K_{t-t'})*(Q^-)^pdt' \\
& = & o_{L^2, \ t\rightarrow +\infty}(1)+\int_0^{t-T} (\nabla K_{t-t'})*(Q^-)^pdt' +\int_{t-T}^{t} (\nabla K_{t-t'})*(Q^-)^pdt' \\
& = & o_{L^2, \ t\rightarrow +\infty}(1)+\int_0^{t-T}O_{L^2}\left( \parallel \nabla K_{t-t'}\parallel_{L^2}\parallel (Q^-)^p\parallel_{L^1} \right)dt' \\
&&+\int_{t-T}^{t} O_{L^2}\left( \parallel \nabla K_{t-t'}\parallel_{L^1}\parallel (Q^-)^p\parallel_{L^2}\right)dt' \\
& = & o_{L^2, \ t\rightarrow +\infty}(1)+\int_0^{t-T}O_{L^2}\left( \frac{1}{(t-t')^{\frac{d+2}{4}}}\times C  \right)dt' \\
&&+\int_{t-T}^{t} O_{L^2}\left( \frac{1}{\sqrt{t-t'}} \underset{t'\in [t-T,t]}{\text{sup}}\parallel (Q^-)^p\parallel_{L^2}\right) dt'\\
& = & o_{L^2, \ t\rightarrow +\infty}(1)+O_{L^2}(T^{-\frac{d-2}{4}} ) + O_{L^2}\left( \underset{t'\in [t-T,t]}{\text{sup}}\parallel Q^-\parallel_{L^{2p}}^p\right)\\
& = & O_{L^2}(T^{-\frac{d-2}{4}} ) \ \ \text{as} \ \ t\rightarrow +\infty.
\end{array}
$$
As this is valid for any $T>0$, this implies that $\nabla Q^-$ tends to $0$ in $L^2$. Hence $\lim_{t\to +\infty}\parallel Q^-(t)\parallel_{\dot H^1}= 0$.

\end{proof}


\subsection{Uniqueness}


We now conlude the proof of Theorem \ref{th:liouville} by proving that $Q^\pm$ are the only solutions uniformly trapped near $\mathcal M$ on $(-\infty,0]$ in the sense of Definition \ref{def:trapped}, up to the symmetries of the equation.

\begin{proof}[Proof of Theorem \ref{th:liouville}] let $u$ be trapped at distance $0<\delta\ll 1$ of $\mathcal M$ on $(-\infty,0]$ in the sense of Definition \ref{def:trapped}. We follow the strategy designed in \cite{RSmin,MMR2}, First we use a dissipation argument to show that the instability must be the main term at $-\infty$, the stable part of the perturbation being of quadratic size. This then implies an exponential decay of the whole perturbation, which hence enters the regime where we constructed the solution, and a simple contraction like argument will close the proof.\\ 

The primary information to notice is that the scale cannot diverge as $t\rightarrow -\infty$ from \fref{eq:controle lambda}, i.e. there exists $0<\lambda_1<\lambda_2$ such that for $t\leq 0$:
$$
\lambda_1\leq \lambda \leq \lambda_2.
$$
We then rewrite the $\dot H^2$ energy bound \fref{eq:dotH2} as:
$$
\frac{d}{ds}\left[\frac{1}{\lambda^2}\int (H \varepsilon)^2 \right]\leq -\frac{1}{C\lambda^2}\int H\varepsilon H^2\varepsilon+\frac{Ca^4}{\lambda^2}+\frac{Ca^2}{\lambda^2} \parallel \varepsilon \parallel_{\dot H^2}^2 .
$$
From the fact \fref{eq:controle lambda} that $\lambda$ does not diverge we rewrite it as:
\be \label{eq:modifieddotH2}
\frac{d}{ds}\left[\frac{1}{\lambda^2}\parallel H \varepsilon \parallel_{L^2}^2\right]\leq C(a^4+a^2\parallel \varepsilon \parallel_{\dot H^2}^2) -\frac{1}{C}\int \varepsilon H^3\varepsilon
\ee
Finally the global energy bound \eqref{eq:variation energie} ensures 
\be
\label{integrability}
\int_{-\infty}^0 (\|\e(s)\|_{\dot{H}^2}^2+a^2(s))ds\lesssim \delta^2.
\ee

\noindent{\bf step 1} Dominance of the instability. We claim that there for any constant $K\gg 1$ for any $0< \delta \ll 1$ small enough there holds the following bound:
\be \label{eq:domination a}
\forall t<0, \ \ \parallel \varepsilon \parallel_{\dot H^2}^2(t)\leq \frac{|a(t)|}{K} .
\ee
\noindent\emph{Sequential control}. We first claim that \eqref{eq:domination a} holds on a subsequence in time $t_n\to -\infty$. We argue by contradiction and assume that there exists $s_0\leq 0$ such that:
\be \label{eq:dominance a contradiction}
\forall s\leq s_0, \ \ \parallel \varepsilon (s) \parallel_{\dot H^2}^2> \frac{|a(s)|}{4K}.
\ee
On the one hand, there exists $s_n\to-\infty$ such that $\lim_{s_n\to -\infty}\|\e(s_n)\|_{\dot{H}^2}=0$ from \eqref{integrability}. On the other hand, injecting \fref{eq:dominance a contradiction} in the energy identity \fref{eq:modifieddotH2}, using \fref{eq:controle lambda} and \fref{eq:estimation dotH1}, yields:
$$
\ba{r c l}
\frac{d}{ds}\left[\frac{1}{\lambda^2}\parallel H\varepsilon \parallel_{L^2}^2\right] \leq  C(K|a|^3+a^2) \parallel H \varepsilon \parallel_{L^2}^2  \leq  C(K\delta+1)a^2 \parallel H \varepsilon \parallel_{L^2}^2 \leq Ca^2 \frac{\parallel H \varepsilon \parallel_{L^2}^2}{\lambda^2}
\ea
$$
which, integrated in time on $[s,s_0]$ for any $s\leq s_0$, using the integrability of $a^2$ coming from the variation of the energy \fref{eq:variation energie}, gives:
$$
\text{log}\left(\frac{1}{\lambda^2(s_0)} \para H\varepsilon (s_0) \para_{L^2}^2 \right)-\text{log}\left(\frac{1}{\lambda^2(s)} \para H\varepsilon (s) \para_{L^2}^2 \right) \leq C\int_s^{s_0}a^2ds\leq C \delta^2.
$$
This can be rewritten, for any $s\leq s_0$ as:
$$
\para H\varepsilon(s) \para_{L^2}^2\geq e^{-C\delta^2} \frac{\lambda^2 (s)}{\lambda^2 (s_0)} \para H \varepsilon (s_0)\para_{L^2}^2\geq c >0
$$
for some constant $c>0$, from \fref{eq:dominance a contradiction} and as the scale does not diverge from \fref{eq:controle lambda}. The above identity then contradicts the convergence to $0$ of $H\varepsilon$ along a subsequence $s_n\rightarrow -\infty$ and gives the desired contradiction.\\

\noindent\emph{No return}. We claim that there exists $K,\delta_0>0$ such that for all $0<\delta\leq \delta_0$, if at some time $s_0$ there holds:
\be \la{min:comparaison varepsilon dotH2 a}
\parallel \varepsilon \parallel_{\dot H^2}^2(s_0)\leq \frac{|a(s_0)|}{4K},
\ee
then 
\be \la{min:comparaison varepsilon dotH2 a 2}
\forall s\in (s_0,0), \ \ \parallel \varepsilon \parallel_{\dot H^2}^2(s)\leq \frac{|a(s)|}{K}
\ee
which concludes the proof of  \eqref{eq:domination a} given that we just showed \fref{min:comparaison varepsilon dotH2 a} for a sequence of times $s_n\rightarrow -\infty$. Indeed, if $a(s_0)=0$, then the solution is $Q$ up to symmetries and the claim follows. We now assume $a(s_0)\neq 0$. We look at $\mathcal S$, the set of times $s_0\leq s \leq 0$ such that:
$$
|a(s)|\geq \frac{2K}{\lambda^2(s)} \parallel \varepsilon \parallel_{\dot H^2}^2
$$
$\mathcal S$ is closed and non empty as it contains $s_0$ from \fref{min:comparaison varepsilon dotH2 a} and as $|\lambda (s_0)-1|\lesssim \delta$ from \fref{eq:controle lambda}. Now inside $\mathcal S$ we compute from the modulation equation \fref{eq:modulation a} for $a$, the modified energy estimate \fref{eq:modifieddotH2} for the error, the size estimate $|a|\lesssim \delta$ and the non divergence of the scale \fref{eq:controle lambda}:
$$
\begin{array}{r c l}
& \frac{d}{ds}\left( |a(s)|-K\frac{2}{\lambda^2(s)}\parallel \varepsilon \parallel_{\dot H^2}^2(s)\right) \\
 \geq  & |a(s)| \left(e_0 -C\delta-\frac{C}{K}-KC\delta^3-C\delta^2 \right) > 0
\end{array}
$$
for $K$ large enough and $\delta $ small enough. Consequently $\mathcal S$ is open, and via connectedness $\mathcal S=[s_0,0]$. One then has on $[s_0,s]$ the bound $\parallel \varepsilon \parallel_{\dot H^2}^2\leq \frac{\lambda^2(s)}{2K}|a(s)|$ which gives \fref{min:comparaison varepsilon dotH2 a 2} from \fref{eq:controle lambda}.\\

\noindent {\bf step 2} Primary refined asymptotics at $-\infty$. We claim that the solution enters the regime of exponential smallness: there exists $\lambda_{\infty}>0$ and $z_{\infty}\in \mathbb R^d$ such that:
\be \label{eq:decroissance exponentielle}
a=a_0e^{\frac{e_0}{\lambda_{\infty}^2} t}+O(\delta^2e^{\frac{2e_0}{\lambda_{\infty}^2} t}), \ \ \parallel \varepsilon \parallel_{\dot H^2}\lesssim \delta^2 e^{\frac{2e_0}{\lambda_{\infty}^2} t}, \ \ 0\neq |a_0|\lesssim \delta ,
\ee
\be \label{eq:decroissance exponentielle2}
\lambda = \lambda_{\infty}+O(\delta^2e^{\frac{2e_0}{\lambda_{\infty}^2} t}), \ \ z=z_{\infty}+O(\delta^2e^{\frac{2e_0}{\lambda_{\infty}^2}t}).
\ee
\noindent \emph{Intermediate estimate}. We first claim the intermediate estimate:
\be \label{eq:decroissance exponentielle inter}
a(s)=a_0e^{e_0 s}+O(\delta^2e^{2e_0 s}), \ \ \parallel \varepsilon \parallel_{\dot H^2}\lesssim \delta^2e^{2e_0 s}, \ \ 0\neq |a_0|\lesssim \delta ,
\ee
which we will now prove by bootstraping the informations one can obtain from \fref{eq:modulation a} and \fref{eq:modifieddotH2} starting from the first bound \fref{eq:domination a}. Taking $K$ large enough and $\delta$ in a small enough range, \fref{eq:modulation a} and \fref{eq:domination a} yields:
$$
\frac{d}{ds}|a(s)|>\frac{2e_0}{3}|a(s)|.
$$
Integrating this on $[s,0]$ using the global bound $|a(s)|\lesssim\delta$ from \eqref{eq:estimation dotH1} ensures:
 $$|a(s)|e^{-\frac{2e_0}3s}\lesssim \delta$$ and hence
\be \label{eq:a 1}
|a|\lesssim \delta e^{\frac{2e_0}{3}s}.
\ee
 From \fref{eq:domination a} this implies 
\be \label{eq:e 1}
\parallel \varepsilon \parallel_{\dot H^2}^2\leq C\delta e^{\frac{2e_0}{3}s}.
\ee
We inject these two estimates in \fref{eq:modifieddotH2} to find:
$$
\frac{d}{ds}\left[O(1)\parallel H\varepsilon \parallel_{L^2}^2 \right]\lesssim \delta^3 e^{2e_0s}
$$
which implies the refined estimate using the coercivity \fref{eq:coercivite} 
\be \label{eq:e 2}
\parallel \varepsilon \parallel_{\dot H^2}\lesssim \delta^{\frac{3}{2}}e^{e_0s}.
\ee
In a similar way, we inject this and \fref{eq:a 1} in \fref{eq:modulation a} to find after reintegration on $[s,0]$:
\be \label{eq:a 2}
a=a_0e^{e_0s}+O(\delta^2e^{\frac{4}{3}e_0s}+\delta^{\frac 3 2}e^{2e_0s}), \ \ a_0=O(\delta).
\ee
Again, injecting the above estimate and \fref{eq:e 2} in \fref{eq:modifieddotH2} gives:
\be \label{eq:e 3}
\parallel \varepsilon \parallel_{\dot H^2}\lesssim \delta^2e^{2e_0s}.
\ee
Injecting this estimate and \fref{eq:a 2} in \fref{eq:modulation a} gives:
\be \label{eq:a 3}
a=a_0e^{e_0s}+O(\delta^2e^{2e_0s}), \ \ a_0=O(\delta).
\ee
The above two estimates are the bound \fref{eq:decroissance exponentielle inter} we had to show.\\

\noindent\emph{Conclusion}. We rewrite the modulation equations \fref{eq:modulation lambda}, and \fref{eq:modulation z} and for $\lambda$ and $z$ using the exponential decay in renormalized time \fref{eq:decroissance exponentielle inter} as:
$$
\frac{\lambda_s}{\lambda}=O(\delta^2 e^{2e_0s}), \ \ \frac{z_s}{\lambda}=O(\delta^2 e^{2e_0s}) .
$$
This implies that there exists $\lambda_{\infty}$ and $z_{\infty}$ such that:
$$
\lambda=\lambda_{\infty}+O(\delta^2 e^{2e_0s}), \ \ z=z_{\infty}+O(\delta^2 e^{2e_0s}).
$$
From the fact that $s$ solves $\frac{ds}{dt}=\lambda^{-2}$, one gets $t=\lambda_{\infty}^2s+O(\delta^2e^{2e_0s})$. The primary exponential bound \fref{eq:decroissance exponentielle inter} in renormalized time $s$ and the above identity then become the desired bounds \fref{eq:decroissance exponentielle} and \fref{eq:decroissance exponentielle2} in original time $t$.\\

\noindent{\bf step 3} Additional exponential decay. Up to scale change and translation, one can assume that $\lambda_{\infty}=1$ and $z_{\infty}=0$. Up to time translation\footnote{As a consequence the solution is maybe no more defined on $(-\infty,0[$ but just on $(-\infty,t_0)$ for some $t_0\in \mathbb R$, which is not a problem as we are working at the asymptotic $t\rightarrow -\infty$.} one can assume that $u$ is not a ground state and that $a_0=\pm 1$ in \fref{eq:decroissance exponentielle}. We write the solution $u$ as:
\be \la{liou:id u}
u(t)=Q\pm e^{e_0 t}\mathcal Y+\tilde ae^{e_0 t}\mathcal Y+v, \ \ v\perp \mathcal Y.
\ee
Then from \fref{eq:decroissance exponentielle}, \fref{eq:decroissance exponentielle2} and \fref{eq:modulation lambda} one has for small enough times:
\be \label{eq:exp}
|\tilde a|\lesssim e^{e_0 t}, \ \ |\tilde a_t|\lesssim e^{e_0 t}, \ \ \parallel v \parallel_{\dot H^1}\lesssim \delta, \ \ \parallel v \parallel_{\dot H^2}\lesssim e^{2e_0t}.
\ee
We claim that in addition\footnote{The $L^{\infty}$ norm being finite from Proposition \ref{pr:cauchy}} that for small enough times:
\be \label{eq:exp2}
\parallel v \parallel_{L^{\infty}}\lesssim e^{e_0t}, \ \ \parallel v\parallel_{\dot H^1}\lesssim e^{2e_0t}.
\ee
From \fref{eq:exp} and Sobolev embedding $v$ satisfies $\para v \para_{L^{\frac{2d}{d-4}}}\lesssim e^{2e_0t}$. This, with \fref{liou:id u} and \fref{eq:exp} implies that the whole perturbation satisfies $\para u-Q\para_{L^{\frac{2d}{d-4}}}\lesssim e^{e_0 t}$. Therefore, one can apply the parabolic estimate of Lemma \ref{para:lem:para}, \fref{para:eq:bd v W2infty}, to obtain the $L^{\infty}$ estimate in \fref{eq:exp2} on some time interval $(-\infty,T]$ for $T$ small enough. We now prove the $\dot H^1$ bound in \fref{eq:exp2}. The evolution equation for $v$ is:
\be \label{eq:evolution v}
v_t+Hv=-\tilde a _te^{te_0}\mathcal Y+NL, \ \ NL:=f(u)-f(Q)-f'(Q)u.
\ee
The Duhamel formula for the solution of \fref{eq:evolution v} yields for $t$ small enough and $t_0>0$:
\be \label{eq:duhamel v}
\nabla v(t)=(\nabla K_{t_0})*v(t-t_0)+\int_{t-t_0}^t (\nabla K_{t-\tau}* (-Vv-\tilde a _te^{\tau e_0}\mathcal Y+NL)d\tau.
\ee
We estimate the nonlinear term using Young inequality, the estimate \fref{eq:NL3} on the nonlinearity, Sobolev embedding, interpolation, \fref{eq:exp},\fref{eq:estimation dotH1} and the fact that $\mathcal Y$ is exponentially decreasing:
$$
\begin{array}{r c l}
\int NL^2 &\lesssim &  \int (e^{e_0t}+\tilde a)^4\mathcal Y^4 Q^{2(p-2)}+\int \varepsilon^{2p} \lesssim  e^{4e_0t}+ \parallel \varepsilon \parallel_{\dot H^{\frac{2d}{d+2}}}^{2p} \\
&\lesssim &  e^{4e_0t}+ \parallel \varepsilon \parallel_{\dot H^1}^{\frac{8p}{d+2}}\parallel \varepsilon \parallel_{\dot H^2}^2 \lesssim  e^{4e_0t} +\delta^{\frac{8p}{d+2}}e^{4e_0 t} .
\end{array}
$$
We come back to the above Duhamel formula. For the first term we use H\"older inequality and \fref{eq:exp}, for the second the fact that $\mathcal Y$ is exponentially decaying and \fref{eq:exp} and for the third the estimate we just proved, yielding:
\bee
 && \parallel \int_{t-t_0}^t (\nabla K_{t-\tau})* (-Vv-\tilde a _te^{\tau e_0}\mathcal Y+NL)d\tau \parallel_{L^2} \\
& \leq & \int_{t-t_0}^t \parallel \nabla K_{t-\tau} \parallel_{L^1}\parallel-Vv-\tilde a _te^{\tau e_0}\mathcal Y+NL\parallel_{L^2}d\tau \\
& \lesssim & \int_{t-t_0}^t \frac{1}{\sqrt{t-\tau}} \left(\parallel V\parallel_{L^{\frac d 2}}\parallel v\parallel_{L^{\frac{2d}{d-4}}}+|\tilde a _t|e^{e_0 \tau}+\parallel NL \parallel_{L^{2p}}^{p}\right)d\tau \\
& \lesssim & \int_{t-t_0}^t \frac{e^{2e_0 \tau }}{\sqrt{t-\tau}}d\tau  \lesssim  e^{2e_0 t}.
\eee
The very same computations shows that the second term in \fref{eq:duhamel v} converges strongly at speed $e^{2e_0 t}$ in $L^2$ as $t_0\rightarrow +\infty$. This implies that $(\nabla K_{t_0})*v(t-t_0)$ converges strongly in $L^2$ as $t_0 \rightarrow +\infty$. As $v$ is uniformly bounded in $L^4$, one has that $(\nabla K_{t_0})*v(t-t_0)$ converges weakly to $0$ as $t_0\rightarrow +\infty$. Therefore, $(\nabla K_{t_0})*v(t-t_0)$ converges strongly in $L^2$ to $0$ as $t_0 \rightarrow +\infty $ and one has the following formula:
$$
\nabla v(t)=\int_{-\infty }^t (\nabla K_{t-\tau}* (-Vv-\tilde a _te^{\tau e_0}\mathcal Y+NL)d\tau =O(e^{2e_0 t})
$$
where the upper bound is implied by the above estimate. \\

\noindent{\bf step 5} Uniqueness via contraction argument. Let $u_1$ and $u_2$ be two solutions of \fref{eq:NLH} that are trapped on $(-\infty,0]$ at distance $\delta$ and that are not ground states. From all the previous results of the previous steps, we assume that they have been renormalized and from \fref{eq:exp} and \fref{eq:exp2} they can be decomposed as:
\be \label{eq:orthogonalite v1-v2}
u_i=Q \pm e^{e_0t}\mathcal Y+\tilde a_i e^{e_0 t}\mathcal Y+v_i+b_i\Lambda Q+z_i.\nabla Q, \ \ v_i\in \text{Span} (\mathcal Y,\Psi_0,\Psi_1,...,\Psi_d)^{\perp},
\ee
the profiles $\Psi_0,...,\Psi_i$ being defined by \fref{eq:def Psi0} and \fref{eq:def Psii} with
\be \label{eq:decroissance expo totale}
|b_i|+|z_i|+|\tilde a_i|e^{e_0t}+\parallel v_i \parallel_{\dot H^1}+\parallel v_i \parallel_{\dot H^2}\lesssim e^{2e_0t}, \ \ \ \parallel v_i \parallel_{L^{\infty}}\lesssim e^{e_0t}.
\ee
Then if they have the same sign at first order on their projection onto the unstable mode (the $\pm$ in \fref{eq:orthogonalite v1-v2}) we claim $u_1=u_2$. This will end the proof of the proposition as $Q^-$ and $Q^+$ are indeed trapped at any distance of $\mathcal M$ as $t\rightarrow -\infty$ from \fref{eq:bound Qpm}. Without loss of generality we chose a $+$ sign. For $T\ll 0$ we define the following norm for the difference:
\be \label{eq:def norme u1-u2}
\begin{array}{r c l}
\parallel u_1-u_2 \parallel_T &:=& \underset{t\leq T}{\text{sup}} \ e^{-2e_0t}\parallel v_1-v_2\parallel_{\dot H^1}+\underset{t\leq T}{\text{sup}} \ e^{-e_0t}|\tilde a_1-\tilde a_2|\\
&&+\underset{t\leq T}{\text{sup}} \ e^{-2e_0t}|b_1-b_2|+\underset{t\leq T}{\text{sup}} \ e^{-2e_0t}|z_1-z_2|
\end{array}
\ee
which is finite from \fref{eq:decroissance expo totale}. The evolution of the difference $u_1-u_2$ is given by:
\be \label{eq:evolution v1-v2}
(v_1-v_2)_t+H(v_1-v_2)=-(\tilde a_1-\tilde a_2)_te^{e_0t}\mathcal Y-(b_1-b_2)_t\Lambda Q-(z_1-z_2)_t.\nabla Q+NL_1-NL_2,
\ee
where:
$$
NL_i:=f(u_i)-f(Q)-f'(Q)u_i.
$$
From \fref{eq:decroissance expo totale}, \fref{eq:NL7}, H\"older inequality and Sobolev embedding one gets the following bounds for the nonlinear term for $t\leq T$:
\bee \label{eq:bound NL1-NL2}
&&\parallel NL_1-NL_2 \parallel_{L^2}\\
&\lesssim & \Bigl\Vert |(\tilde a_1-\tilde a_2) e^{e_0 t}\mathcal Y+v_1-v_2+(b_1-b_2)\Lambda Q+(z_1-z_2).\nabla Q| \\
&& \times (|\tilde a_1 e^{e_0 t}\mathcal Y+v_1+b_1\Lambda Q+z_1.\nabla Q |^{p-1}+|\tilde a_2 e^{e_0 t}\mathcal Y+v_2+b_2\Lambda Q+z_2.\nabla Q |^{p-1}\Bigr\Vert_{L^2} \\
&\lesssim & \para (\tilde a_1-\tilde a_2) e^{e_0 t}\mathcal Y+v_1-v_2+(b_1-b_2)\Lambda Q+(z_1-z_2).\nabla Q\para_{L^{\frac{2d}{d-2}}} \\
&& \times \Bigl( \para (1+\tilde a_1) e^{e_0 t}\mathcal Y+v_1+b_1\Lambda Q+z_1.\nabla Q \para_{L^{(p-1)d}}^{p-1}\\
&&+\para \tilde (1+\tilde a_2) e^{e_0 t}\mathcal Y+v_2+b_2\Lambda Q+z_2.\nabla Q \para_{L^{(p-1)d}}^{p-1}\Bigr) \\
&\lesssim & \para (\tilde a_1-\tilde a_2) e^{e_0 t}\mathcal Y+v_1-v_2+(b_1-b_2)\Lambda Q+(z_1-z_2).\nabla Q\para_{\dot H^1}  \times e^{(p-1)t} \\
&\lesssim & \parallel u_1-u_2 \parallel_T e^{(p+1)t} . \\
\eee

\noindent \emph{Energy estimate for the difference of errors}. From \fref{eq:evolution v1-v2}, the orthogonality conditions \fref{eq:orthogonalite v1-v2} and the bound \fref{eq:bound NL1-NL2} on the nonlinear term one gets the following energy estimate:
$$
\begin{array}{r c l}
 \frac{d}{dt}\left[\int (v_1-v_2)H(v_1-v_2) \right] &=& 2 \int H(v_1-v_2)(NL_1-NL_2)-2\int H(v_1-v_2)^2 \\
 &\leq & \parallel NL_1-NL_2\parallel_{L^2}^2 \lesssim  e^{2(p+1)e_0 t} \parallel u_1-u_2 \parallel_T.
\end{array}
$$
From the coercivity property of the linearized operator \fref{eq:coercivite} one has:
$$
\int (v_1-v_2)H(v_1-v_2)\lesssim \parallel v_1-v_2\parallel_{\dot H^1}^2 \lesssim \int (v_1-v_2)H(v_1-v_2).
$$
Therefore, $ \int (v_1-v_2)H(v_1-v_2)$ goes to zero as $t\rightarrow -\infty$, and reintegrating in time the energy estimate gives from the coercivity:
\be \label{eq:contraction 2}
\underset{t\leq T}{\text{sup}}  \parallel v_1-v_2 \parallel_{\dot H^1}e^{-2e_0t}  \lesssim e^{(p-1)e_0T}\parallel u_1-u_2\parallel_T .
\ee

\noindent \emph{Modulation equations for the differences of parameters}. We take the scalar products of \fref{eq:evolution v1-v2} with $\Lambda Q$, $\mathcal Y$ and $\nabla Q$, using the bound \fref{eq:bound NL1-NL2} for the nonlinear term and obtain for any $t\leq T$
$$
\ba{r c l}
 \left|\frac{d}{dt}\left[b_1-b_2+\int (v_1-v_2)\Lambda Q \right] \right|+\left|\frac{d}{dt}\left[z_1-z_2+\int (v_1-v_2) \nabla Q \right] \right| \\
+|\frac{d}{dt}(\tilde a_1-\tilde a_2)|e^{e_0t}  \lesssim   e^{(p+1)e_0 t} \parallel u_1-u_2 \parallel_T .
\ea
$$
The boundary term involving $v_1-v_2$ satisfies from Sobolev embedding as $d\geq 7$:
$$
\ba{r c l}
&\left|\int (v_1-v_2)\Lambda Q  \right|+\left|\int (v_1-v_2)\nabla Q  \right|  \lesssim  \int \frac{|v_1-v_2|}{1+|x|^{d-2}} \\
\lesssim & \para v_1-v_2\para_{L^{\frac{2d}{d-2}}} \para (1+|x|)^{-d+2}\para_{L^{\frac{2d}{d+2}}} \lesssim \para v_1-v_2\para_{\dot H^1}.
\ea
$$
The two above equations, after reintegration in time, as the left hand side goes to $0$ as $t\rightarrow -\infty$, imply:
\bea \label{eq:contraction 1}
\non & \underset{t\leq T}{\text{sup}} (|b_1-b_2|+|z_1-z_2|+|\tilde a_1-\tilde a_2|e^{e_0t})e^{-2e_0t} \\
\lesssim & e^{(p-1)e_0T} \parallel u_1-u_2 \parallel_T+ \underset{t\leq T}{\text{sup}} \para v \para_{\dot H^1}e^{-2e_0t}  \lesssim e^{(p-1)e_0T} \parallel u_1-u_2 \parallel_T
\eea
where we used the estimate \fref{eq:contraction 2}.

\noindent \emph{Conclusion}. From the definition \fref{eq:def norme u1-u2} of the norm of the difference that is adapted to the exponential decay, and the estimates \fref{eq:contraction 1} and \fref{eq:contraction 2} one obtains:
$$
\parallel u_1-u_2\parallel_{T} \lesssim e^{(p-1)e_0T} \parallel u_1-u_2\parallel_{T} .
$$
For $T\ll 0$ small enough this implies: $\parallel u_1-u_2\parallel_{T}=0 $. This means that the solutions $u_1$ and $u_2$ are equal.

 \end{proof}


\section{Classification of the flow near the ground state} \la{sec:th}


We are now in position to conclude the proof of Theorem \ref{th:main}.


\subsection{Set up} \la{sub:strategy}


Let
$$
0<\delta \ll \alpha \ll \alpha^* \ll 1.
$$
be three small strictly positive constants to be fixed later on. Let $u_0\in \dot H^1$ with
\be \label{eq:distance u0 delta2}
\parallel u_0-Q\parallel_{\dot H^1}\leq \delta^4
\ee
and let $u$ be the solution of \fref{eq:NLH} starting from $u_0$ (see Proposition \ref{pr:cauchy} for the local wellposedness result) with maximal time of existence $T_{u_0}$. To prove Theorem \ref{th:main} we are going to study $u$ for times where it is close to the manifold of ground states $\mathcal M$, that is to say in the set
$$
\left\{v\in \dot H^1, \ d(v,\mathcal M)=\underset{\lambda>0, \ z\in \mathbb R^d}{\text{inf}}\para v-Q_{z,\lambda}\para_{\dot H^1}\leq \alpha^* \right\}
$$
using the variables $\lambda$, $s$, $z$, $a$ and $\varepsilon$ introduced in Definition \ref{def:trapped} to decompose it in a suitable way. We introduce three particular times related to the trajectory of the solution $u$ starting from $u_0$. For a constant $K \gg 1$ big enough to be fixed later we define
\be \label{eq:def Tins}
\begin{array}{r c l}
T_{\text{ins}}  := \text{sup} \Bigl\{ & 0\leq t <T_{u_0}, \ \ \underset{0\leq t'\leq t}{\text{sup}} \ d(u(t'),\mathcal M) \leq \delta^2,  &  \\
& \forall t'\in [0,t], \ \ \parallel \varepsilon (t') \parallel_{\dot H^2}^2> \frac{|a(t')|}{K} &\Bigr\} .
\end{array}
\ee
Let $\tilde K >0$ be a constant, independent of the other constants, such that for any $\nu>0$ with $0<\tilde K \nu \leq \alpha^* $ and  $v\in \dot H^1$,
\be \label{eq:def tildeK}
\tilde K \nu \leq d(v,\mathcal M) \leq \alpha^* \ \ \Rightarrow \ \ \text{either} \ \   \parallel \varepsilon \parallel_{\dot H^1}\geq \nu  \ \ \text{or} \ \ |a|\geq\nu .
\ee
Such a constant $\tilde K$ exists since for $v\in \dot H^1$ with  $d(v,\mathcal M) \leq \alpha$ one has from \fref{eq:decomposition}
$$
d(v,\mathcal M)\leq \parallel \varepsilon \parallel_{\dot H^1}+C|a| .
$$
We then define:
\be \label{eq:def Ttrans}
\begin{array}{r c l}
T_{\text{trans}}  := \text{sup} \Bigl\{ & T_{\text{ins}} \leq t <T_{u_0}, \ \ \underset{T_{\text{ins}}\leq t'\leq t}{\text{sup}} \ d(u(t'),\mathcal M)\leq \tilde K \delta,  &  \\
& \forall t'\in [0,t], \ \ |a(t')|\leq \delta &\Bigr\} ,
\end{array}
\ee
\be \label{eq:def Texit}
\begin{array}{r c l}
T_{\text{exit}}  := \text{sup} \Bigl\{ & T_{\text{trans}}\leq t <T_{u_0}, \ \ \underset{T_{\text{trans}}\leq t'\leq t}{\text{sup}} \ d(u(t'),\mathcal M)\leq  \tilde K \alpha,  &  \\
& \forall t'\in [0,t], \ \ |a(t')|\leq \alpha &\Bigr\} .
\end{array}
\ee
with the convention that $\text{sup}(\emptyset)=T_{u_0}$. Our strategy of the proof is the following. In Lemma \ref{lem:caracterisation Tins} we characterize $T_{\text{ins}}$ as the time, if not $+\infty$, for which the instability has started to take control over the solution. In Lemma \ref{lem:soliton} we show that if it never happens, i.e. $T_{\text{ins}}=+\infty$, then the solution converges to some soliton. Indeed, in that case the main part of the renormalized perturbation is located on the stable infinite dimensional direction of perturbation $\approx (\Lambda Q,\partial_{x_1}Q,...,\partial_{x_d}Q,\mathcal Y)^{\perp}$ and undergoes dissipation. In Lemma \ref{lem:exit} we show that if it happens, i.e. $T_{\text{ins}}<+\infty$, then the instability will drive the solution toward type I blow up or dissipation. The analysis is done in three times. First we characterize the time interval $[T_{\text{ins}},T_{\text{trans}}]$ as the transition period in which the solution stays trapped at distance $\delta^2$ and is such that at $T_{\text{trans}}$ the stable perturbation is quadratic compared to the instability. For later times, on $[T_{\text{trans}},T_{\text{exit}}]$ this implies an exponential growth of the instability, with a stable perturbation being still quadratic. In this exponential instable regime we can compare the solution with the minimal solutions $Q^{\pm}$ introduced in Proposition \ref{pr:minimal} and compute that they are close at the exit time $T_{\text{exit}}$. As Type I blow up and dissipation are stable behaviors, $u$ will undergo one or the other.\\

We now proceed to the detailed proof of Theorem \ref{th:main}.


\subsection{Caracterization of $T_{\text{ins}}$}


 We first characterize the time $T_{\text{ins}}$.

\begin{lemma}[$T_{\text{ins}}$ as the instability time] \label{lem:caracterisation Tins}
There exists $K^*\gg 1$, such that for any $K\geq K^*$, there exists $0<\delta^*(K)\ll 1$, such that for any $0<\delta <\delta^*(K)$, on $[0,T_{\text{ins}})$ there holds:
\be \label{eq:bound delta2 Tins}
|a(t)|\lesssim K\delta^4, \ \ \parallel \varepsilon (t) \parallel_{\dot H^1}\lesssim \delta^4.
\ee
Moreover, if $T_{\text{ins}}<T_{u_0}$ then:
\be \label{eq:caracterisation Tins}
\parallel \varepsilon (T_{\text{ins}}) \parallel_{\dot H^2}^2= \frac{|a(T_{\text{ins}})|}{K}.
\ee
\end{lemma}

\begin{proof}[Proof of Lemma \ref{lem:caracterisation Tins}]

To prove the lemma, from the definition \fref{eq:def Tins} of $T_{\text{ins}}$ , it suffices to show that \fref{eq:bound delta2 Tins} holds, which will automatically imply that the other identity \fref{eq:caracterisation Tins} holds for $\delta$ small enough. We will prove it by computing the time evolution of $a$ and performing an energy estimate adapted to the linear level for $\varepsilon$ in the regime $0\leq t \leq T_{\text{ins}}$. $u$ being trapped at distance $\delta^2$ on $[0,T_{\text{ins}})$ in the sense of Definition \ref{def:trapped}, we reason in renormalized time \fref{eq:def s} and define: $S_{\text{ins}}:=s(T_{\text{ins}})$. \\

\noindent{\bf step 1} Bound for $a$. From the definition \fref{eq:def Tins} of $T_{\text{ins}}$, the modulation estimate \fref{eq:modulation a} for $a$ on $0\leq s \leq S_{\text{ins}}$ and \fref{eq:estimation dotH1} one has:
$$
|a_s|\leq e_0 |a|+Ca^2+C\parallel \varepsilon \parallel_{\dot H^2}^2\leq \parallel \varepsilon \parallel_{\dot H^2}^2(K |e_0| +CK |\delta^2 |+C)\lesssim K\parallel \varepsilon \parallel_{\dot H^2}^2 
$$
for $K$ large enough. We reintegrate in time this identity using the variation of energy formula \fref{eq:variation energie}:
$$
|a(s)| \leq |a(0)|+CK \int_0^{S_{\text{ins}}} \para \varepsilon (s)\para_{\dot H^2}^2ds \leq C|\delta |^4+CK\delta^4\lesssim K\delta^4
$$
for $K$ large enough as initially $|a(0)|\lesssim \delta^4$ from \fref{eq:distance u0 delta2} and \fref{eq:estimation dotH1}.\\

\noindent{\bf step 2} Bound for $\varepsilon$. From the definition of $T_{\text{ins}}$ \fref{eq:def Tins}, the $\dot H^1$ bound \fref{eq:dotH1},  \fref{eq:estimation dotH1} and the coercivity \fref{eq:coercivite}, one has for $0\leq s \leq S_{\text{ins}}$:
\be \label{eq:energie dotH1 Tins} 
\ba{r c l}
\frac{d}{ds}\left[\frac 1 2 \int \varepsilon H\varepsilon \right] & \leq &  -\frac{1}{C}\int (H\varepsilon)^2+Ca^4 \leq -\frac{1}{C}\int (H\varepsilon)^2+CK |a|^3\para \varepsilon \para_{\dot H^2}^2 \\
& \leq &  -\frac{1}{C}\int (H\varepsilon)^2+CK\delta^6 \int (H \varepsilon)^2 \leq 0 \\
\ea
\ee
for $\delta$ small enough (depending on $K$) meaning that on $[0,S_{\text{ins}}]$, the quantity $ \int \varepsilon H\varepsilon$ is a Lyapunov functional. We integrate in time using the coercivity \fref{eq:coerciviteA} to find for $0\leq s \leq S_{\text{ins}}$:
$$
\parallel \varepsilon (s) \parallel_{\dot H^1}\lesssim \left(\int \varepsilon (s) H\varepsilon (s) \right)^{\frac 1 2} \leq \left(\int \varepsilon (0) H\varepsilon (0) \right)^{\frac 1 2} \lesssim \parallel \varepsilon (0) \parallel_{\dot H^1} \lesssim \delta^4
$$
from \fref{eq:distance u0 delta2} and \fref{eq:estimation dotH1}, ending the proof of the lemma. 
\end{proof}


\subsection{Soliton regime} We now claim that $T_{\text{ins}}=T_{u_0}$ is the (Soliton) regime. 

\begin{lemma}[Soliton regime for $T_{\text{ins}}=T_{u_0}$] \label{lem:soliton}

There exists $K^*\gg 1$, such that for any $K\geq K^*$, there exists $0<\delta^*(K) \ll 1$, such that for any $0<\delta <\delta^*(K)$, if $T_{\text{ins}}=T_{u_0}$ then  $T_{u_0}=+\infty$ and there exists $z_{\infty}\in \mathbb R^d$ and $\lambda_{\infty}>0$ such that:
$$
u\rightarrow Q_{z_{\infty},\lambda_{\infty}} \ \ \text{as} \ t\rightarrow +\infty \ \text{in} \ \dot{H}^1.
$$

\end{lemma}

\begin{proof}[Proof of Lemma \ref{lem:soliton}] Let $u$ satisfy \fref{eq:distance u0 delta2} such that $T_{\text{ins}}=T_{u_0}$. From \fref{eq:bound delta2 Tins}, there exists $\tilde C>0$ (depending on $K$) such that $u$ is trapped at distance $\tilde C\delta^4$ in the sense of Definition \ref{def:trapped} on $[0,T_{u_0})$. We reason in renormalized time \fref{eq:def s} and define $S(u_0)=\underset{t\rightarrow T_{u_0}}{\text{lim}} s(t)$. \\

\noindent{\bf step 1} Global existence.
We claim that $u$ is a global solution, i.e. that $T_{u_0}=+\infty$. Indeed, recall from \fref{eq:equivalence temps} that the times $t$ or $s$ are equivalents, i.e. $\frac{s}{C}\leq t \leq Cs$ for $C>0$. Injecting the bound \eqref{eq:bound delta2 Tins} on $a$ into \eqref{eq:dotH2} and using Gronwall's lemma ensures that $\|\e(t)\|_{\dot{H}^2}<c(t)<+\infty$ for all bounded time $t$, and hence $\| u(t)\|_{\dot{H}^2}<c(t)<+\infty$ an $T(u_0)=+\infty$ from the blow up criterion \eqref{beibeibvei}.\\

\noindent{\bf step 2} Convergence of the perturbation to $0$ in $\dot H^2$. We claim that:
\be \label{eq:convergence dotH2 soliton}
\parallel \varepsilon \parallel_{\dot H^2}\rightarrow 0 \ \ \text{as} \ \ s \rightarrow +\infty .
\ee
Indeed, the $\dot H^2$ bound \fref{eq:dotH2} and the smallness \eqref{eq:bound delta2 Tins} ensure:
\be
\label{eobnveneo}
\frac{d}{ds}\left[\frac{1}{\lambda^2}\parallel H\varepsilon \parallel_{L^2}^2\right]\leq \frac{C}{\lambda^2}(a^4+a^2\parallel \varepsilon \parallel_{\dot H^2}^2) -\frac{1}{C\lambda^2}\int \varepsilon H^3\varepsilon\lesssim a^2+\|\e\|_{\dot{H}^2}^2.
\ee
Now since $0<\l_0<\l(s)<\l_2$, the right hand side is in $L^1([0,+\infty))$ from \eqref{eq:variation energie}, and hence there exists a sequence $s_n\to +\infty$ with $\|\e(s_n)\|_{\dot{H}^2}\to0$, and integrating \eqref{eobnveneo} on $[s_n,s]$ yields $$\forall s\geq s_n, \ \ \\|\e(s)\|^2_{\dot{H}^2}\lesssim \|\e(s_n)\|^2_{\dot{H}^2}+\int_{s_n}^{+\infty} (a^2(\tau)+\|\e(\tau)\|_{\dot{H}^2}^2)d\tau$$ and \eqref{eq:convergence dotH2 soliton} follows.\\

\noindent{\bf step 3} Convergence of the central point and the scale. We claim that there exist $\lambda_{\infty}>0$ and $z_{\infty}\in \mathbb R^d$ such that:
\be \label{eq:convergence lambda z soliton}
\lambda \rightarrow \lambda_{\infty}, \ \ z\rightarrow z_{\infty} \ \ \text{as} \ s \rightarrow +\infty .
\ee
The evolution for these parameters is given by the modulation equations \eqref{eq:modulation lambda 2}, \eqref{eq:modulation z 2}. After integration in time this gives:
$$
|\lambda (s)-\lambda(0)|= O(\parallel \varepsilon (s) \parallel_{\dot H^{1+\frac{1}{3}}})+\int_0^s O(a^2(\tau)+\parallel \varepsilon (\tau)\parallel_{\dot H^2}^2)d\tau ,
$$
$$
|z(s)-z(0)| = O(\parallel \varepsilon \parallel_{\dot H^{1+\frac 1 3}})+\int_0^s O(a^2(\tau)+\parallel \varepsilon (\tau) \parallel_{\dot H^2}^2)d\tau .
$$
From the uniform bound \fref{eq:bound delta2 Tins} at the $\dot H^1$ level and the convergence to $0$ at the $\dot H^2$ level \fref{eq:convergence dotH2 soliton}, using interpolation one has that the first term in the above identities converges to $0$:
$$
\parallel \varepsilon \parallel_{\dot H^{1+\frac{1}{3}}} \lesssim \para \varepsilon \para_{\dot H^1}^{\frac 2 3} \para \varepsilon \para_{\dot H^2}^{\frac 1 3}\leq (\tilde C\delta^4)^{\frac 2 3} \para \varepsilon \para_{\dot H^2}^{\frac 1 3}  \rightarrow 0 \ \ \text{as} \ s\rightarrow +\infty .
$$
From the variation of energy identity \fref{eq:variation energie} one has that $(a^2+\parallel \varepsilon \parallel_{\dot H^2}^2)\in L^1([0,+\infty))$, implying that the second term is convergent. These two facts imply \fref{eq:convergence lambda z soliton}.\\

\noindent{\bf step 4} Convergence of the perturbation to $0$ in $\dot H^1$. The convergence \eqref{eq:convergence dotH2 soliton} and the bound $a\lesssim \parallel \varepsilon \parallel_{\dot H^2}^2$ from \fref{eq:def Tins} ensure $$a(s)\to 0\ \ \mbox{as}\ \ s\to+\infty.$$
It remains to show the convergence to $0$ for $\varepsilon$ in the $\dot H^1$ energy norm. We come back to the original time variable $t$. As $a$ converges to $0$ as $t\rightarrow +\infty$ with $\int_0^{+\infty}a^2(t)dt<+\infty$ from \eqref{eq:variation energie}, and as $\varepsilon$ converges to $0$ as $t\rightarrow +\infty$ in $\dot H^2$ with $\int_0^{+\infty}\parallel \varepsilon (t)\parallel_{\dot H^2}^2dt<+\infty$, and is bounded in $\dot H^1$ from \fref{eq:estimation dotH1} and the fact that the solution is trapped at distance $\tilde C \delta^4$ on $[0,+\infty)$ from \fref{eq:bound delta2 Tins}, we can gather the instable and stable parts and write our solution as:
\be \label{eq:decomposition soliton}
u=Q_{z,\lambda}+\tilde{\varepsilon}, \ \ \tilde{\varepsilon}=(a\mathcal Y +\varepsilon)_{z,\lambda}
\ee
with a perturbation $\tilde \varepsilon $ satisfying, as the scale does not diverge from \fref{eq:controle lambda}:
\be \label{eq:convergence 0 soliton}
\parallel \tilde \varepsilon \parallel_{\dot H^1}\ll 1, \ \ \lim_{t\to +\infty} \|\tilde \varepsilon(t)\|_{\dot H^2}= 0, \ \ \int_0^{+\infty} \parallel \tilde \varepsilon (t)\parallel_{\dot H^2}^2dt<+\infty .
\ee
The last space time integrability property, via Sobolev embedding yields:
$$
\int_0^{+\infty} \parallel \nabla \tilde \varepsilon \parallel_{L^{\frac{2d}{d-2}}}^2<+\infty .
$$
Hence one has the boundedness of Strichartz type norms for $\nabla \tilde \varepsilon$: 
\be \label{eq:strichartz varepsilon}
\parallel \nabla \tilde \varepsilon \parallel_{L^{\infty}([0,+\infty),L^2(\mathbb R^d))}+\parallel \nabla \tilde \varepsilon \parallel_{L^2([0,+\infty),L^{\frac{2d}{d-2}}(\mathbb R^d))}<+\infty .
\ee
The evolution of $\tilde \varepsilon$ is given by:
$$
\tilde \varepsilon_t-\Delta \tilde \varepsilon=pQ_{z,\lambda}^{p-1}\tilde \varepsilon+\frac{\lambda_t}{\lambda}(\Lambda Q)_{z,\lambda}+\frac{z_t}{\lambda}.(\nabla Q)_{z,\lambda}+NL, \ \ NL:=f(u)-f(Q)-f'(Q)\tilde \varepsilon .
$$
The Duhamel formula then gives, with $K_t$ being defined by \fref{eq:def Kt}:
\be \label{eq:duhamel nabla tilde varepsilon}
\begin{array}{r c l}
\nabla \tilde \varepsilon &=& K_t*(\nabla \tilde \varepsilon (0))+\int_0^t (\nabla K_{t-t'})*(pQ_{z,\lambda}^{p-1}\tilde \varepsilon)dt'+\int  K_{t-t'}*(\nabla NL)dt' \\
&&+\int_0^t (\nabla K_{t-t'})*(\frac{\lambda_t}{\lambda}(\Lambda Q)_{z,\lambda}+\frac{z_t}{\lambda}.(\nabla Q)_{z,\lambda}  )dt'. \\
\end{array}
\ee
We now estimate each term in the right hand side of the previous identity and prove that it goes to $0$ in $\dot H^1$ as $t\rightarrow +\infty$.\\

\noindent \emph{Free evolution term}. The first one undergoes dissipation:
\be \label{eq:lineaire nabla tilde varepsilon}
K_t*(\nabla \tilde \varepsilon (0)) \rightarrow 0 \ \text{in} \ L^2(\mathbb R^d) \ \text{as} \ t\rightarrow +\infty .
\ee

\noindent \emph{Potential term}. Let $0<\epsilon \ll 1$ be small enough and $T>0$. Using Young and H\"older inequalities, \fref{eq:convergence 0 soliton} and the fact that $\lambda$ and $z$ converge as $t\rightarrow +\infty$, for $t>T$ one computes:
\bee
&& \parallel \int_0^t (\nabla K_{t-t'})*(pQ_{z,\lambda}^{p-1}\tilde \varepsilon)dt' \parallel_{L^2} \\
& \leq &  \int_0^{t-T} \parallel (\nabla K_{t-t'})*(pQ_{z,\lambda}^{p-1}\tilde \varepsilon) \parallel_{L^2}dt'+\int_{t-T}^{t} \parallel (\nabla K_{t-t'})*(pQ_{z,\lambda}^{p-1}\tilde \varepsilon) \parallel_{L^2}dt'  \\
& \lesssim &  \int_0^{t-T} \parallel \nabla K_{t-t'} \parallel_{L^{1+\epsilon}} \parallel pQ_{z,\lambda}^{p-1}\tilde \varepsilon \parallel_{L^{\frac{2+2\epsilon}{1+3\epsilon}}}dt'+\int_{t-T}^{t} \parallel \nabla K_{t-t'}\parallel_{L^1}\parallel pQ_{z,\lambda}^{p-1}\tilde \varepsilon \parallel_{L^2} \\
& \lesssim &  \int_0^{t-T} \frac{1}{(t-t')^{\frac 1 2 + \epsilon\frac{ d}{2+2\epsilon}}} \parallel pQ_{z,\lambda}^{p-1} \parallel_{L^{\frac{d+\epsilon d}{2+\epsilon (2+d)}}} \parallel \tilde \varepsilon \parallel_{L^{\frac{2d}{d-4}}}dt' \\
 &&+\int_{t-T}^{t} \frac{1}{\sqrt{t-t'}} \parallel pQ_{z,\lambda}^{p-1} \parallel_{L^{\frac d 2}} \parallel \tilde \varepsilon \parallel_{L^{\frac{2d}{d-4}}} dt' \\
 &\lesssim &  \int_0^{t-T} \frac{1}{(t-t')^{\frac 1 2 +\frac{ \epsilon d}{2+2\epsilon}}} \parallel \tilde \varepsilon  \parallel_{\dot H^2}dt'+\int_{t-T}^{t} \frac{1}{\sqrt{t-t'}}  \parallel \tilde \varepsilon  \parallel_{\dot H^2}dt'  \\
& \lesssim &  \left( \int_0^{t-T} \frac{dt'}{(t-t')^{1 +\frac{\epsilon d}{1+\epsilon}}}\right)^{\frac 1 2} \left( \int_0^{t-T} \parallel \tilde  \varepsilon \parallel_{\dot H^2}^2dt'\right)^{\frac 1 2}+ \int_{t-T}^{t} \frac{\underset{t-T\leq t'\leq t}{\text{sup}} \parallel \tilde \varepsilon (t') \parallel_{\dot H^2}dt'}{\sqrt{t-t'}}   \\
& \lesssim &  \frac{1}{T^{\frac{\epsilon d}{2+2\epsilon}}}\left( \int_0^{+\infty} \parallel \tilde \varepsilon (t') \parallel_{\dot H^2}^2dt'\right)^{\frac 1 2}+ C(T)\underset{t-T\leq t'\leq t}{\text{sup}} \parallel \tilde \varepsilon (t') \parallel_{\dot H^2}   \\
 & \lesssim &  \left( \frac{1}{T^{\frac{\epsilon d}{2+2\epsilon}}}+ C(T)\underset{t-T\leq t'\leq t}{\text{sup}} \parallel \tilde \varepsilon (t') \parallel_{\dot H^2}\right)\rightarrow \frac{1}{T^{\varepsilon(d+1)}} \ \text{as} \ t\rightarrow +\infty.
\eee
Now, this computation being valid for any fixed $T>0$, one obtains the convergence to $0$ for this term: 
\be \label{eq:potentiel nabla tilde varepsilon}
\underset{t\rightarrow +\infty}{\text{lim}} \left\Vert \int_0^t (\nabla K_{t-t'})*(pQ_{z,\lambda}^{p-1}\tilde \varepsilon)dt' \right\Vert_{L^2} =0 .
\ee

\noindent \emph{Nonlinear term}. We now turn to the third term in \fref{eq:duhamel nabla tilde varepsilon}. Using the estimates \fref{eq:NL8} and \fref{eq:NL9} for the nonlinearity one obtains first the pointwise bound:
$$
\begin{array}{r c l}
|\nabla NL| &=& \Bigl| p(|Q_{z,\lambda}+\tilde \varepsilon|^{p-1}-Q_{z,\lambda}^{p-1})\nabla \tilde \varepsilon \\
&& +p(|Q_{z,\lambda}+\tilde \varepsilon|^{p-1}-Q_{z,\lambda}^{p-1}-(p-1)Q^{p-2}_{z,\lambda}\tilde \varepsilon)\nabla (Q_{z,\lambda}) \Bigr| \\
&\lesssim & |\tilde \varepsilon |^{p-1}|\nabla \tilde \varepsilon| +\frac{|\tilde \varepsilon|^p}{1+|x|} .
\end{array}
$$
Now, using H\"older inequality, the generalized Hardy inequality in Lebesgue spaces and Sobolev embedding we estimate this term via:
\bee
&& \parallel \nabla NL \parallel_{L^{\frac{2d}{(d-2)+(d-4)(p-1)}}} \\
& \lesssim & \parallel |\tilde \varepsilon |^{p-1}|\nabla \tilde \varepsilon| \parallel_{L^{\frac{2d}{(d-2)+(d-4)(p-1)}}} +  \left\Vert \frac{|\tilde \varepsilon|^p}{1+|x|} \right\Vert_{L^{\frac{2d}{(d-2)+(d-4)(p-1)}}} \\
& \lesssim & \parallel \tilde \varepsilon \parallel_{L^{\frac{2d}{d-4}}}^{p-1} \parallel \nabla \tilde \varepsilon \parallel_{L^{\frac{2d}{d-2}}} +  \left\Vert \frac{\tilde \varepsilon }{1+|x|} \right\Vert_{L^{\frac{2d}{d-2}}} \parallel \tilde \varepsilon \parallel_{L^{\frac{2d}{d-4}}}^{p-1} \\
&\lesssim & \parallel \tilde \varepsilon \parallel_{\dot H^2}^{p-1}  \parallel \tilde \varepsilon \parallel_{\dot H^2} + \parallel \nabla \tilde \varepsilon \parallel_{L^{\frac{2d}{d-2}}} \parallel \tilde \varepsilon \parallel_{\dot H^2}^{p-1} \lesssim  \parallel \tilde \varepsilon \parallel_{\dot H^2}^p .
\eee
As $\parallel \tilde \varepsilon \parallel_{\dot H^2}\in L^2([0,+\infty))$ this means that:
$$
\parallel \nabla NL \parallel_{L^{\frac{2}{p}}\left( [0,+\infty),L^{\frac{2d}{(d-2)+(d-4)(p-1)}}(\mathbb R^d)\right)}<+\infty . 
$$
We let $(q,r)$ be the conjugated exponents of $\frac{2}{p}$ and $\frac{2d}{(d-2)+(d-4)(p-1)}$ respectively:
$$
q=\frac{2}{2-p}>2, \ \ r=\frac{2d}{d+2-(d-4)(p-1)}>2 .
$$
They satisfy the Strichartz relation $ \frac{2}{q}+\frac{d}{r}=\frac d 2 $. Therefore, using \fref{strichartz inhomogene} one gets for any $0\leq T\leq t$:
$$
\left\Vert \int_0^T  K_{T-t'}*(\nabla NL)dt'  \right\Vert_{L^2} \leq \parallel \nabla NL \parallel_{L^{\frac{2}{p}}\left([0,T],L^{\frac{2d}{(d-2)+(d-4)(p-1)}}\right)}<+\infty ,
$$
$$
\ba{r c l}
\parallel \int_T^t  K_{t-t'}*(\nabla NL)dt'  \parallel_{L^2} &\leq & \parallel \nabla NL \parallel_{L^{\frac{2}{p}}\left([T,t],L^{\frac{2d}{(d-2)+(d-4)(p-1)}}\right)} \\
&\leq & \parallel \nabla NL \parallel_{L^{\frac{2}{p}}\left([T,+\infty],L^{\frac{2d}{(d-2)+(d-4)(p-1)}}\right)} \\
&\rightarrow  &0  \ \text{as} \ T\rightarrow +\infty . 
\ea
$$
We now write:
$$
\int_0^t  K_{t-t'}*(\nabla NL)dt'=K_{t-T}*\left(  \int_0^{T}  K_{T-t'}*(\nabla NL)dt' \right)+\int_T^t K_{t-t'}*(\nabla NL)dt' 
$$
and the two previous inequalities imply that, for $T$ fixed the first term goes to $0$ in $\dot H^1$ as $t\rightarrow +\infty$, and the second goes to $0$ in $\dot H^1$ as $T\rightarrow +\infty$ uniformly in $t\geq T$. Therefore one gets the convergence to $0$ of the nonlinear term in $\dot H^1$ as $t\rightarrow +\infty$:
\be \label{eq:NL nabla tilde varepsilon}
\left\Vert \int  K_{t-t'}*(\nabla NL)dt'  \right\Vert_{L^2} \rightarrow 0 \ \text{as} \ t \rightarrow +\infty .
\ee

\noindent\emph{Remainders from scale and space translations}. We now turn to the last two terms in \fref{eq:duhamel nabla tilde varepsilon}. From the modulation equations \fref{eq:modulation lambda}, \fref{eq:modulation z}, the variation of energy formula \fref{eq:variation energie}, the fact that $\lambda$ converges and the convergence to $0$ of $a$ and $\varepsilon$ in $\dot H^2$ \fref{eq:convergence lambda z soliton} and \fref{eq:convergence dotH2 soliton} one has:
$$
\lambda_t , z_t \in L^2([0,+\infty))\cap L^{\infty}([0,+\infty)), \ \ \text{with} \ |\lambda_t|+|z_t|\rightarrow 0 \ \text{as} \ t\rightarrow +\infty. 
$$
Moreover one has $\Lambda Q,\nabla Q\in L^2\cap L^{\frac 3 2}$ as $d\geq 7$. One then deduces that:
$$
\left\Vert \frac{\lambda_t}{\lambda}(\Lambda Q)_{z,\lambda}+\frac{z_t}{\lambda}.(\nabla Q)_{z,\lambda}\right\Vert_{L^2}\rightarrow 0 \ \text{as} \ t\rightarrow +\infty ,
$$
$$
\left\Vert \frac{\lambda_t}{\lambda}(\Lambda Q)_{z,\lambda}+\frac{z_t}{\lambda}.(\nabla Q)_{z,\lambda}\right\Vert_{L^2} \in L^{2}([0,+\infty),L^{\frac 3 2}(\mathbb R^d)).
$$
Therefore one has for any $t\geq T>0$, using Young inequality for convolution and H\"older inequality:
\bee
&& \left\Vert \int_0^t (\nabla K_{t-t'})*\left[\frac{\lambda_t}{\lambda}(\Lambda Q)_{z,\lambda}+\frac{z_t}{\lambda}.(\nabla Q)_{z,\lambda}  \right]dt' \right\Vert_{L^2} \\
&\leq & \int_0^{t-T} \parallel \nabla K_{t-t'} \parallel_{L^{\frac 6 5}} \left\Vert \frac{\lambda_t}{\lambda}(\Lambda Q)_{z,\lambda}+\frac{z_t}{\lambda}.(\nabla Q)_{z,\lambda} \right\Vert_{L^{\frac 3 2}} dt'\\
&&+ \int_{t-T}^{t} \parallel \nabla K_{t-t'} \parallel_{L^1} \left\Vert \frac{\lambda_t}{\lambda}(\Lambda Q)_{z,\lambda}+\frac{z_t}{\lambda}.(\nabla Q)_{z,\lambda} \right\Vert_{L^{2}}dt' \\
&\lesssim & \int_0^{t-T}\frac{dt'}{(t-t')^{\frac{1}{2}+\frac{d}{12}}} \left\Vert \frac{\lambda_t}{\lambda}(\Lambda Q)_{z,\lambda}+\frac{z_t}{\lambda}.(\nabla Q)_{z,\lambda} \right\Vert_{L^{\frac 3 2}} \\
&&+ \int_{t-T}^{t} \frac{dt'}{\sqrt{t-t'}} \para\frac{\lambda_t}{\lambda}(\Lambda Q)_{z,\lambda}+\frac{z_t}{\lambda}.(\nabla Q)_{z,\lambda} \parallel_{L^{2}} \\
&\lesssim & T^{-\frac{d}{12}} \left(\int_0^{t-T} \parallel \frac{\lambda_t}{\lambda}(\Lambda Q)_{z,\lambda}+\frac{z_t}{\lambda}.(\nabla Q)_{z,\lambda} \parallel_{L^{\frac 3 2}}^2dt'\right)^{\frac 1 2} \\
&&+ \sqrt T \underset{t\in [t-T,t]}{\text{sup}} \parallel \frac{\lambda_t}{\lambda}(\Lambda Q)_{z,\lambda}+\frac{z_t}{\lambda}.(\nabla Q)_{z,\lambda} \parallel_{L^{2}} \rightarrow  T^{-\frac{d}{12}} \ \text{as} \ t\rightarrow +\infty .
\eee
from what we deduce as this is valid for any $T>0$ that
\be \label{eq:zt lambdat nabla tilde varepsilon}
\begin{array}{r c l}
\parallel \int_0^t (\nabla K_{t-t'})*(\frac{\lambda_t}{\lambda}(\Lambda Q)_{z,\lambda}+\frac{z_t}{\lambda}.(\nabla Q)_{z,\lambda}  )dt' \parallel_{L^2}\rightarrow 0 \ \text{as} \ t\rightarrow +\infty .
\end{array}
\ee

\noindent \emph{Conclusion} We now come back to the Duhamel formula \fref{eq:duhamel nabla tilde varepsilon}. We showed in \fref{eq:lineaire nabla tilde varepsilon}, \fref{eq:potentiel nabla tilde varepsilon}, \fref{eq:NL nabla tilde varepsilon} and \fref{eq:zt lambdat nabla tilde varepsilon} that each terms in the right hand side converges to zero strongly in $L^2$ as $t\rightarrow +\infty$. Hence $ \tilde \varepsilon$ converges strongly to $0$ in $\dot H^1$ as $t\rightarrow +\infty$. Going back to \fref{eq:decomposition soliton}, this, with the convergence of $\lambda$ and $z$ as $t\rightarrow +\infty$ showed in Step 3, implies that $u\rightarrow Q_{z_{\infty},\lambda_{\infty}}$ strongly in $\dot H^1$ as $t\rightarrow +\infty$, ending the proof of the Lemma.

\end{proof}


\subsection{Transition regime and no return}


We now study the (exit) regime $T_{\text ins}<T_{u_0}$ and start with the fundamental no return lemma:

\begin{lemma}[No return lemma] 
There holds $$T_{\text{trans}}<+\infty, \ \ T_{\text{trans}}< T_{u_0}$$ and the bounds for all $T_{\text{ins}}\leq t \leq T_{\text{trans}}$:
\be \label{eq:varepsilon Ttrans}
\parallel \varepsilon (t) \parallel_{\dot H^1}+\parallel \varepsilon (t) \parallel_{\dot H^2}\lesssim \delta^2 
\ee
 \be \label{eq:domination Ttrans}
\parallel \varepsilon \parallel_{\dot H^2}^2(t)\lesssim \frac{|a(t)|}{K}.
\ee
Moreover, 
\be
\label{cenovbeinone}
|a(T_{\text{trans}})|=\delta.
\ee
\end{lemma}

\begin{proof}[Proof of Lemma \ref{lem:exit}] First notice is from their definition \fref{eq:def Ttrans} and \fref{eq:def Texit} one has $T_{\text{trans}}\leq T_{\text{exit}}\leq T_{u_0}$ and that the solution is trapped at distance $\tilde K \alpha$ on $[0,T_{\text{exit}})$. From \fref{eq:controle lambda} the scale does not degenerate:
\be \label{eq:controle lambda Ttrans}
\lambda = 1+O(\alpha).
\ee
To reason in renormalized time we define $$S_{\text{trans}}:=\underset{t\rightarrow T_{\text{trans}}}{\text{lim}}s(t), \ \ S_{\text{exit}}:=\underset{t\rightarrow T_{\text{exit}}}{\text{lim}}s(t).$$ We recall that from its definition, on $[0,S_{\text{trans}})$ the solution is trapped at distance $\tilde K \delta$.\\

\noindent {\bf step 1} Proof of the $\dot H^1$ bound. On $[0,S_{\text{trans}})$ the $\dot H^1$ energy bound \fref{eq:dotH1} gives:
$$
\frac{d}{ds} \left[\frac 1 2 \int \varepsilon H\varepsilon \right]\lesssim a^4.
$$
We integrate it in time on using the fact that $\int_0^{S_{\text{trans}}} a^2\lesssim \delta^2$ from \fref{eq:variation energie}, the fact that $|a(s)|\leq \delta$ for all $0\leq s \leq S_{\text{trans}}$ from \fref{eq:def Tins} and \fref{eq:def Ttrans} and \fref{eq:estimation dotH1}, the coercivity \fref{eq:coerciviteA} and \fref{eq:distance u0 delta2}:
$$
\parallel \varepsilon (s) \parallel_{\dot H^1}^2 \lesssim \int \varepsilon (s)H\varepsilon (s)\lesssim  \int \varepsilon (0)H\varepsilon (0)+\int_0^s a^4(s)ds \lesssim \parallel \varepsilon (0) \parallel_{\dot H^1}^2+\delta^4 \lesssim \delta^4.
$$
This proves the first bound in \fref{eq:varepsilon Ttrans}.\\

\noindent {\bf step 2} Proof of the $\dot H^2$ bound. On $[0,S_{\text{trans}})$ the $\dot H^2$ energy bound \fref{eq:dotH2}, gives:
$$
\frac{d}{ds} \left[\frac{1}{\lambda^2} \int (H\varepsilon)^2 \right]\lesssim \frac{1}{\lambda^2}\left(a^4+a^2\int (H \varepsilon)^2 \right).
$$
We integrate it in time on $[S_{\text{ins}},S_{\text{trans}})$ using the facts that 
$$
\int_{S_{\text{ins}}}^{S_{\text{trans}}} \left(\parallel H\varepsilon (s) \parallel^2+a^2(s)\right)ds \lesssim \delta^2
$$
from \fref{eq:variation energie} and because on $[0,S_{\text{trans}})$ the solution is trapped at distance $\tilde K \delta$ from \fref{eq:def Ttrans}, that $|a(s)|\leq \delta$ for all $S_{\text{ins}}\leq s \leq S_{\text{trans}}$ from \fref{eq:def Ttrans} and \fref{eq:estimation dotH1}, the non degeneracy of the scale \fref{eq:controle lambda Ttrans} and \fref{eq:coercivite}:
$$
\begin{array}{r c l}
\parallel \varepsilon (s)\para_{\dot H^2}^2&\lesssim & \int (H\varepsilon (s))^2 \lesssim  \int (H\varepsilon (S_{\text{ins}}))^2+\int_{S_{\text{ins}}}^s \left(a^4+a^2\int (H \varepsilon)^2 \right)ds\\
&\lesssim & \para \varepsilon (S_{\text{ins}})\para_{\dot H^2}^2+\delta^2\int_{S_{\text{ins}}}^s (a^2+\para \varepsilon \para_{\dot H^2}^2)ds\lesssim \delta^4 
\end{array}
$$
as $\para \varepsilon (T_{\text{ins}})\para_{\dot H^2}^2\lesssim \delta^4 $ from \fref{eq:caracterisation Tins} and \fref{eq:bound delta2 Tins}. This proves the second bound in \fref{eq:varepsilon Ttrans}.\\

\noindent {\bf step 3} No return. We now turn to the proof of \eqref{eq:domination Ttrans} using a bootstrap argument. Let $\tilde C>0$ be a constant such that:
$$
\frac{1}{\tilde C} \parallel \varepsilon \parallel_{\dot H^2}^2 \leq \int (H\varepsilon)^2\leq \tilde C \para \varepsilon \para_{\dot H^2}^2
$$
which exists from \fref{eq:coercivite} and is independent of the other constants. Let $\mathcal S$ be the set of times:
$$
\mathcal S:=\left\{s\in [S_{\text{ins}},S_{\text{trans}}], \ \forall S_{\text{ins}}\leq s'\leq s, \  \int (H\varepsilon(s'))^2\leq \frac{\tilde C |a(s')| \lambda^2(s')}{K\lambda^2(S_{\text{ins}})} \right\} .
$$
$\mathcal S$ is non empty as it contains $S_{\text{ins}}$ from the two previous inequalities. It is closed by a continuity argument. We now show that it is open in $[S_{\text{ins}},S_{\text{trans}})$. For each $s\in \mathcal S$ using \fref{eq:dotH2}, \fref{eq:modulation a}, \fref{eq:controle lambda Ttrans}, the fact that $|a(s)|\leq \delta$ for all $S_{\text{ins}}\leq s \leq S_{\text{trans}}$ from \fref{eq:def Ttrans} and \fref{eq:estimation dotH1}, and the coercivity \fref{eq:coercivite}:
\bee
&& \frac{d}{ds}\left( |a(s)|-\frac{K\lambda^2(S_{\text{ins}}) }{\tilde C \lambda^2}\int (H\varepsilon)^2 \right) \\
&\geq  & e_0|a(s)|-C|a(s)|^2-C\parallel \varepsilon (s) \parallel_{\dot H^2}^2-\frac{C K}{\tilde C}\frac{\lambda^2(S_{\text{ins}})}{\lambda^2(s)}(|a(s)|^4+|a(s)|^2\parallel \varepsilon (s) \parallel_{\dot H^2}^2 ) \\
&\geq  & |a(s)| \left(e_0 -C\delta-\frac C K -CK\delta^3-C\delta^2\right) >0
\eee
for $K$ large enough and $\delta $ small enough, where the constant $C$ is independent of the other constants. Consequently $\mathcal S$ is open, which implies that $\mathcal S=[S_{\text{ins}},S_{\text{trans}}]$. From the definition of $\mathcal S$, \fref{eq:controle lambda Ttrans} and \fref{eq:coercivite} one has proven \fref{eq:domination Ttrans}.\\

\noindent {\bf step 4} Proof of $T_{\text{trans}}<T_{u_0}$. We claim that $T_{\text{trans}}<+\infty$. Indeed, from \fref{eq:domination Ttrans} and the modulation equation \fref{eq:modulation a} for $a$ one gets:
$$
|a|_s\geq |a(s)|\left(e_0-C\delta-\frac{C}{K}\right)
$$
for a constant $C$ independent of the other constants. Hence for $K$ large enough, the function $|a|$ satisfies $|a|_s>c>0$ on $[S_{\text{ins}},S_{\text{trans}}]$. Therefore $T_{\text{trans}}<+\infty$ because if not $a$ would be unbounded, which is a contradiction to the very definition of $T_{\text{trans}}$ \fref{eq:def Ttrans}. This implies $T_{\text{trans}}<T_{u_0}$. This is obvious if $T_{u_0}=+\infty$ and is otherwise a consequence of \eqref{eq:varepsilon Ttrans} and the control of the scale which implies a uniform $\dot{H}^2$ bound on u in $[0,T_{\text trans}]$ and hence $T_{\text trans}<T_{u_0}$ from \eqref{beibeibvei}. The estimate \eqref{cenovbeinone} now follows by continuity and \eqref{eq:domination Ttrans}.
\end{proof}


\subsection{(Exit) dynamics}


We now classify the (Exit) dynamics and show that $Q^\pm$ are the attractors.

\begin{lemma}[Classification of the (Exit) dynamics] \label{lem:exit}

There exists $K^*\gg 1$, such that for any $K\geq K^*$, there exists $0<\delta^*(K) \ll 1$, such that for any $0<\delta <\delta^*(K)$, if $T_{\text{ins}}<T_{u_0}$ then either $u$ will blow up with type I blow up forward in time, or $u$ will converge to $0$ strongly in $\dot H^1$.
\end{lemma}

\begin{proof}[Proof of lemma \ref{lem:exit}] At time $T_{\text{trans}}$, there holds from \eqref{cenovbeinone}, \fref{eq:varepsilon Ttrans}: $$|a(T_{\text{trans}})|=\delta, \ \ \parallel \varepsilon (t) \parallel_{\dot H^1}+\parallel \varepsilon (t) \parallel_{\dot H^2}\lesssim \delta^2.$$  This is the cruising instable regime. We recall that the exit time $T_{\text{trans}}< T_{\text{exit}}$ is defined by \fref{eq:def Texit}.\\

\noindent {\bf step 1} First exponential bounds in renormalized time. We claim that $T_{\text{exit}}<+\infty$, $T_{\text{exit}}\neq T_{u_0}$ and that the following holds on $[S_{\text{trans}},S_{\text{exit}}]$:
\be \label{eq:regime TtransTexit}
\begin{array}{l l}
a(s)=\pm (\delta+\tilde a) e^{e_0(s-S_{\text{trans}})}, \ \text{with} \ |\tilde a| \leq \tilde C' \delta^2 e^{e_0(s-S_{\text{trans}})}, \\
\parallel \varepsilon \parallel_{\dot H^1}+\parallel \varepsilon \parallel_{\dot H^2}\leq \tilde C' \delta^2e^{2e_0(s-S_{\text{trans}})}
\end{array}
\ee
for some constant $\tilde C'>0$ independent of the other constants. We use a bootstrap method to prove the above bound. Fix $\tilde C>0$ and define $\mathcal S \subset [S_{\text{trans}},S_{\text{exit}}]$ as the set of times $s$ such that \fref{eq:regime TtransTexit} holds on $[S_{\text{trans}},s]$. For $\tilde C'$ large enough independently on the other constants, $\mathcal S$ is non empty as it contains $S_{\text{trans}}$ from \fref{eq:varepsilon Ttrans}. It is closed by a continuity argument. We claim that for $\tilde C'$ big enough independently of the other constants, it is open. The first thing to notice is that from \fref{eq:regime TtransTexit} and \fref{eq:def Texit}:
\be \label{eq:controle s exit}
\delta e^{e_0(s-S_{\text{trans}})}\leq 2\alpha .
\ee

\noindent  \emph{Estimate for $\tilde a$}. For $s\in \mathcal S$ we compute from \fref{eq:modulation a}, using \fref{eq:regime TtransTexit} and \fref{eq:controle s exit}:
$$
\begin{array}{r c l}
|\tilde a_s| & = & e^{-e_0(s-S_{\text{trans})}}[O(a^2)+O(\parallel \varepsilon \parallel_{\dot H^2}^2)] \\
&\leq & e^{-e_0 (s-S_{\text{trans}})}[C\delta^2 e^{2(e_0s-S_{\text{trans}})}+C(\tilde C')^2\delta^2e^{2e_0(s-S_{\text{trans}})}\alpha^2] \\
&\leq & \delta^2  e^{e_0(s-S_{\text{trans}})}(C+C\alpha^2(\tilde C')^2).
\end{array}
$$
Reintegrating it in time between $S_{\text{trans}}$ and $s\in \mathcal S$ we find, as $\tilde a (S_{\text{trans}})=0$:
\be \label{eq:expo exit a}
|\tilde a(s)| \lesssim  \delta^2  e^{e_0(s-S_{\text{trans}})}(C+C\alpha^2(\tilde C')^2)<\delta^2 \tilde C'  e^{e_0(s-S_{\text{trans}})}
\ee
for $\tilde C'$ large enough and $\alpha$ small enough.

\noindent \emph{Estimate for $\parallel \varepsilon \parallel_{\dot H^1}$}. For $s\in \mathcal S$ we compute from \fref{eq:dotH1}, using \fref{eq:regime TtransTexit} and \fref{eq:controle s exit}:
$$
\frac{d}{ds}\left[\int \varepsilon H\varepsilon \right]\leq C|a|^4\leq C\delta^4 e^{4e_0(s-S_{\text{trans}})}[1+(\tilde C')^4\alpha^4].
$$
Reintegrating it in time between $S_{\text{trans}}$ and $s\in \mathcal S$ we find, using \fref{eq:varepsilon Ttrans} and the coercivity \fref{eq:coerciviteA}:
\be \label{eq:expo exit dotH1}
\begin{array}{r c l}
\parallel \varepsilon (s) \parallel_{\dot H^1}^2 & \leq &  C \int \varepsilon (s) H\varepsilon (s)\\
&\leq & C \int \varepsilon (S_{\text{trans}}) H\varepsilon (S_{\text{trans}}) +\int_{S_{\text{Trans}}}^s C\delta^4 e^{4e_0(s-S_{\text{trans}})}[1+(\tilde C')^4\alpha^4] ds \\
&\leq & C\parallel \varepsilon (S_{\text{trans}})\parallel_{\dot H^1}^2 + \delta^4  e^{4e_0(s-S_{\text{trans}})}(C+C\alpha^4(\tilde C')^2)\\
&\leq & C\delta^4+ \delta^4  e^{4e_0(s-S_{\text{trans}})}(C+C\alpha^4(\tilde C')^2)<\delta^4 \tilde C'  e^{4e_0(s-S_{\text{trans}})}
\end{array}
\ee
for $\tilde C'$ large enough and $\alpha$ small enough.\\

\noindent \emph{Estimate for $\parallel \varepsilon \parallel_{\dot H^2}$}. For $s\in \mathcal S$ we compute from \fref{eq:dotH2}, using \fref{eq:regime TtransTexit}, \fref{eq:controle s exit} and \fref{eq:controle lambda Ttrans}:
$$
\frac{d}{ds}[O(1)\int (H\varepsilon )^2 ]\leq Ca^4+Ca^2\parallel \varepsilon \parallel_{\dot H^2}^2\leq C\delta^4 e^{4e_0(s-S_{\text{trans}})}[1+(\tilde C')^4\alpha^4].
$$
where $O(1)=\frac{1}{\lambda^2}=1+O(\alpha)$ from \fref{eq:controle lambda Ttrans}. Reintegrating it in time between $S_{\text{trans}}$ and $s\in \mathcal S$ we find, using \fref{eq:varepsilon Ttrans} and the coercivity \fref{eq:coercivite}:
\be \label{eq:expo exit dotH2}
\begin{array}{r c l}
\parallel \varepsilon (s) \parallel_{\dot H^2}^2 & \leq & C\parallel \varepsilon (S_{\text{trans}})\parallel_{\dot H^2}^2 + \delta^4  e^{4e_0(s-S_{\text{trans}})}(C+C\alpha^2(\tilde C')^2)\\
&\leq & C\delta^4+ \delta^4  e^{4e_0(s-S_{\text{trans}})}(C+C\alpha^4(\tilde C')^2)<\delta^4 \tilde C'  e^{4e_0(s-S_{\text{trans}})}
\end{array}
\ee
for $\tilde C'$ large enough and $\alpha$ small enough.\\

\noindent \emph{Conclusion}. From \fref{eq:expo exit a}, \fref{eq:expo exit dotH1} and \fref{eq:expo exit dotH2} one obtains that $\mathcal S$ is open, hence $\mathcal S=[S_{\text{trans}},S_{\text{exit}})$ which ends the proof of \fref{eq:regime TtransTexit}. The law for $a$ \eqref{eq:regime TtransTexit} implies $T_{\text{exit}}<+\infty$, and the $\dot{H}^2$ bound \eqref{eq:regime TtransTexit} and the control of the scale now imply $T_{\text{exit}}<T_{u_0}$.\\

\noindent {\bf step 2} Exponential bounds in original time variable. Let now the constant $\tilde C'$ used in the first substep in \fref{eq:regime TtransTexit} be fixed. We claim that one has the following estimates\footnote{The constants involved in the $\lesssim $ may depend on $\tilde C'$ defined earlier on the first substep of Step 2 but this is not a problem as it is fixed from now on independently of the other constants.} in original time variables on $[T_{\text{trans}},T_{\text{exit}}]$:
\be \label{eq:original lambda z}
|z(t)-z(T_{\text{trans}})|+\left|\frac{\lambda (t)}{\lambda (T_{\text{trans}})}-1 \right|\lesssim \delta^2 e^{\frac{2 e_0}{\lambda^2(T_{\text{trans}})}(t-T_{\text{trans}})},
\ee
\be \label{eq:original a varepsilon}
\begin{array}{l l}
a(t)=\pm (\delta+\tilde a) e^{e_0\frac{(t-T_{\text{trans}})}{\lambda^2(T_{\text{trans}})}}, \ \ |\tilde a| \lesssim \delta^2 e^{\frac{e_0}{\lambda^2(T_{\text{trans}})}(t-T_{\text{trans}})}, \\
 \parallel \varepsilon \parallel_{\dot H^1}+\parallel \varepsilon \parallel_{\dot H^2}\lesssim \delta^2e^{\frac{2e_0}{\lambda^2(T_{\text{trans}})}(t-T_{\text{trans}})}.
\ea
\ee

\noindent  \emph{Bound for $\lambda$ and estimate on $s$}. The modulation equation \fref{eq:modulation lambda 2}, using \fref{eq:regime TtransTexit}, can be rewritten as:
$$
\left|\frac{d}{ds}\left[\text{log}(\lambda)+O(\delta^2e^{2e_0(s-S_{\text{trans}})}) \right]\right| \lesssim \delta^2 e^{2e_0(s-S_{\text{trans}})}
$$ 
After reintegration in time this becomes, using \fref{eq:controle lambda Ttrans} and \fref{eq:controle s exit}:
\be \label{eq:regime TtransTexit lambda}
\lambda (s)= \lambda (S_{\text{trans}})+O(\delta^2 e^{2e_0(s-S_{\text{trans}})}).
\ee
The definition of the renormalized time $s$ \fref{eq:def s} then implies:
$$
\frac{dt}{ds}=\lambda^2(T_{\text{trans}})+O(\delta^2e^{2e_0(s-S_{\text{trans}})}).
$$
Reintegrated in time this gives:
$$
t-T_{\text{trans}}=\lambda^2(T_{\text{trans}})(s-S_{\text{trans}})+O(\delta^2e^{2e_0(s-S_{\text{trans}})}).
$$
From \fref{eq:controle s exit} this implies:
$$
e^{s-S_{\text{trans}}}=e^{\frac{t-T_{\text{trans}}}{\lambda^2(T_{\text{trans}})}+O(\delta^2e^{2e_0(s-S_{\text{trans}})})}=e^{\frac{t-T_{\text{trans}}}{\lambda^2(T_{\text{trans}})}} (1+O(\alpha))
$$
and therefore
\be \label{eq:regime TtransTexit t}
t-T_{\text{trans}}=\lambda^2(T_{\text{trans}})(s-S_{\text{trans}})+O\left(\delta^2e^{\frac{2e_0}{\lambda^2(T_{\text{trans}})}(t-T_{\text{trans}})}\right) .
\ee

We inject the above identity in \fref{eq:regime TtransTexit lambda}, yielding the estimate for $\lambda $ in \fref{eq:original lambda z}.\\

\noindent \emph{Bound for $z$.} The modulation equation \fref{eq:modulation z 2} for $z$, using \fref{eq:regime TtransTexit}, can be rewritten as:
$$
\left|\frac{d}{ds}\left[z+O(\delta^2e^{2e_0(s-S_{\text{trans}})}) \right]\right| \lesssim \delta^2 e^{2e_0(s-S_{\text{trans}})}
$$ 
After reintegration in time this becomes:
\be \label{eq:regime TtransTexit z}
|z (s)-z(S_{\text{trans}})|\lesssim \delta^2 e^{2e_0(s-S_{\text{trans}})}.
\ee
We inject \fref{eq:regime TtransTexit t} in the above equation, giving the estimate for $z $ in \fref{eq:original lambda z}.

\noindent \emph{Bounds for $a$ and $\varepsilon$}. We inject \fref{eq:regime TtransTexit t} in \fref{eq:regime TtransTexit}, giving \fref{eq:original a varepsilon}.\\

\noindent{\bf step 3} Setting up the comparison with $Q^{\pm}$. From now on, without loss of generality, we treat the case of a "plus" sign in \fref{eq:original a varepsilon}, i.e. $a=\delta e^{e_0\frac{t-T_{\text{trans}}}{\lambda^2(T_{\text{trans}})}}$ at the leading order on $[T_{\text{trans}},T_{\text{exit}}]$, which will correspond to an exit close to a renormalized\footnote{i.e. an element of the orbit of $Q^+$ under the symmetries of the flow: scaling and space and scale translations.} version of $Q^+$. The case of a "minus" sign corresponds to an exit close to a renormalized\footnote{Idem.} version of $Q^-$ and can be treated with exactly the same techniques. From the definition \fref{eq:def Texit} of the exit time $T_{\text{exit}}$ and the law for $a$ \fref{eq:original a varepsilon} on $[T_{\text{trans}},T_{\text{exit}}]$ at this time there holds:
$$
\ba{l l}
\delta e^{\frac{e_0}{\lambda^2(T_{\text{trans}})}(T_{\text{exit}}-T_{\text{trans}})}+O\left(\delta^2 e^{\frac{2e_0}{\lambda^2(T_{\text{trans}})}(T_{\text{exit}}-T_{\text{trans}})}\right)=a(T_{\text{exit}})=\alpha, \\
\delta e^{\frac{e_0}{\lambda^2(T_{\text{trans}})}(T_{\text{exit}}-T_{\text{trans}})}=O(\alpha) ,
\ea
$$
from what one obtains the following formula for $T_{\text{exit}}$:
\be \la{eq:estimation Texit}
T_{\text{exit}}=T_{\text{trans}}+\frac{\lambda^2(T_{\text{trans}})}{e_0}\text{log} \left(\frac{\alpha (1+O(\alpha))}{\delta} \right) .
\ee
To ease the writing of the estimates, we first renormalize the function $u$ at time $T_{\text{trans}}$. For $t\in [0,\frac{T_{\text{exit}}-T_{\text{trans}}}{\lambda^2(T_{\text{trans}})}]$ we define the renormalized time:
\be \la{cla:def tau}
t'(t):=\lambda^2(T_{\text{trans}})t+T_{\text{trans}}\in [T_{\text{trans}},T_{\text{exit}}].
\ee
We define the renormalized versions of $u$ and of the adapted variables as:
\be \la{eq:def hat u}
\hat u(t,\cdot):=\left( \tau_{-z(T_{\text{trans}})}u\left( t',\cdot \right) \right)_{\frac{1}{\lambda (T_{\text{trans}})}} ,
\ee
\be \la{eq:def bar variables}
\ba{l l}
\bar \varepsilon (t):=\varepsilon (t'), \ \ \bar a (t):=a (t'), \ \ \bar z (t):=\frac{z(t')-z(T_{\text{trans}})}{\lambda (T_{\text{trans}}} ,\ \ \bar \lambda (t):=\frac{\lambda (t')}{\lambda (T_{\text{trans}})} .
\ea
\ee
We define the renormalized exit time $\hat T_{\text{exit}}$ by
\be \label{eq:estimation hatTexit}
\hat T_{\text{exit}}:=\frac{T_{\text{exit}}-T_{\text{trans}}}{\lambda^2(T_{\text{trans}})}= \frac{1}{e_0}\text{log}\left(  \frac{\alpha (1+O(\alpha))}{\delta} \right) \gg 1
\ee
from \fref{eq:estimation Texit}. From the invariances of the equation, $\hat u$ is also a solution of \fref{eq:NLH}, at least defined on $[0,\hat T_{\text{exit}}]$, and $u$ blows up with type I if and only if $\hat u$ blows up with type I. We will then show the result for $\hat u$. As $a,z,\lambda$ and $\varepsilon$ are the adapted variables for $u$ given by Definition \ref{def:trapped} and as $(\tau_z f)_{\lambda}=\tau_{\lambda z} (f_{\lambda})$:
\be \la{eq:first decomposition exit}
\ba{r c l}
\hat u (t) & = & \left( \tau_{-z(T_{\text{trans}})}u\left( t' \right) \right)_{\frac{1}{\lambda (T_{\text{trans}})}} = \left( \tau_{-z(T_{\text{trans}})}\tau_{z(t')}(Q+a(t')\mathcal Y+\varepsilon (t')) \right)_{\frac{1}{\lambda (T_{\text{trans}})}} \\
&=& (Q+a (t')\mathcal Y+ \varepsilon (t'))_{\frac{z(t')-z(T_{\text{trans}})}{\lambda (T_{\text{trans}})},\frac{\lambda (t')}{\lambda (T_{\text{trans}})}} = (Q+\bar a (t)\mathcal Y+\bar \varepsilon (t))_{\bar z (t),\bar \lambda (t)}
\ea
\ee
with $\bar \varepsilon$ satisfying the orthogonality conditions \fref{eq:orthogonalite}, from \fref{eq:def hat u}, \fref{eq:def bar variables}, \fref{cla:def tau}, \fref{eq:decomposition} and \fref{eq:orthogonalite}. Therefore, $(\bar \lambda,\bar z,\bar a,\bar \varepsilon)$ are the variables associated to the decomposition of $\hat u$ given by Definition \ref{def:trapped}. \fref{cla:def tau}, \fref{eq:def bar variables} and the bounds \fref{eq:original lambda z} and \fref{eq:original a varepsilon} imply:
\be \label{eq:bound decomposition exit}
|1-\bar \lambda |+|\bar z|+|\bar a-\delta e^{e_0 t}|+\parallel \bar \varepsilon \parallel_{\dot H^1}+\parallel \bar \varepsilon \parallel_{\dot H^2} \lesssim \delta^2 e^{2e_0t}
\ee
The change of variable we did thus simplified the estimates. As we aim at comparing $\hat u$ to $Q^{\pm}$, the scale and central points $\bar \lambda$ and $\bar z$ might be adapted for $\hat u$ but they are not for $Q^{\pm}$. We perform a second change of variables to treat these two profiles as perturbations of $Q$ under an affine adapted decomposition. To do so we define:
\be \la{eq:def hat variables}
\ba{l l}
\hat a=\frac{e^{-e_0t}}{\para \mathcal Y \para_{L^2}^2}\langle \hat u-Q,\mathcal Y\rangle-\delta , \ \  \hat z_i=\frac{d}{\int \chi_M|\nabla Q|^2}\langle \hat u-Q,\Psi_i \rangle \ \ \text{for} \ \ 1\leq i \leq d ,  \\
\hat b=\frac{1}{\int \chi_M\Lambda Q^2}\langle \hat u-Q,\Psi_0 \rangle , \ \ \hat v =\hat u -Q-(\hat a+\delta e^{e_0t})\mathcal Y-\hat b\Lambda Q-\hat z.\nabla Q .
\ea
\ee
From \fref{eq:orthogonalite Psii} these new variables produce the following decomposition for $\hat u$:
\be \la{exit:eq:decomposition hat u}
\hat u = Q + (\delta +\hat a) e^{e_0t}\mathcal Y+\hat v+\hat b \Lambda Q+\hat z.\nabla Q, \ \ \hat v \in \text{Span} (\mathcal Y,\Psi_0,\Psi_1,...,\Psi_d)^{\perp}
\ee
and from \fref{eq:first decomposition exit}, \fref{eq:bound decomposition exit}, \fref{eq:def hat variables} and \fref{eq:orthogonalite Psii} they enjoy the following bounds:
\be \label{eq:bound hatu}
|\hat b |+|\hat z|+|\hat a|e^{e_0 t}+\parallel \hat v \parallel_{\dot H^1}+\parallel \hat v \parallel_{\dot H^2} \lesssim \delta^2 e^{2e_0t}.
\ee
This, from Sobolev embedding, implies that:
$$
\para \hat u-Q\para_{L^{\frac{2d}{d-4}}}\lesssim \delta e^{e_0t}.
$$
By the parabolic regularization estimate \fref{para:eq:bd v W2infty} this implies that on $[1,\hat T_{\text{exit}}]$:
$$
\para \hat u-Q\para_{L^{\infty}}\lesssim \delta e^{e_0t}.
$$
In turn, using again \fref{exit:eq:decomposition hat u} and \fref{eq:bound hatu} this implies that in addition to \fref{eq:bound hatu} for all $t\in [1,\hat T_{\text{exit}}]$ one has an exponential $L^{\infty}$ bound:
\be \label{eq:bound hatu2}
\parallel \hat v \parallel_{L^{\infty}} \lesssim \delta e^{e_0t} .
\ee
We now aim at comparing $\hat u$, under the decomposition \fref{exit:eq:decomposition hat u}, the a priori bounds \fref{eq:bound hatu} and \fref{eq:bound hatu2}, to $Q^+$ in the time interval $[0,\hat T_{\text{exit}}]$. We need to provide a similar decomposition for $Q^+$. We recall that from \fref{eq:bound Qpm}, $Q^+$ satisfies on $(-\infty,0]$:
\be \label{eq:Q+ exit}
Q^+= Q+ (\epsilon+O(\epsilon^2e^{e_0t})) e^{te_0}\mathcal Y+w, \ \ \parallel w\parallel_{L^{\infty}}+\parallel w\parallel_{\dot H^1} \lesssim \epsilon^2e^{2e_0t}
\ee
for some $0<\epsilon \ll 1$ fixed independent of $\alpha$ and $\delta$. One can therefore assume: 
\be \la{cla:def epsilon}
\alpha \ll \epsilon
\ee
First, we perform a time translation so that at time $0$, the projection of $Q^+$ onto the unstable mode $\mathcal Y$ matches the one of $\hat u$, i.e. is $\delta$. To do so we define the time:
\be \la{eq:def t0 Q+}
t_0:=\frac{1}{e_0} \text{log}\left( \frac{\epsilon}{\delta}\right) \gg \hat T_{\text{exit}}
\ee
from \fref{eq:estimation hatTexit} and \fref{cla:def epsilon}. We let $\hat Q^+$ be the time translated version of $Q^+$ defined by:
\be \label{eq:def hatQ+}
\hat Q^+(t,x)=Q^+\left(t-t_0,x\right), \ \ (t,x)\in (-\infty,t_0]\times \mathbb R^d.
\ee
We then decompose $\hat Q^+$ in a similar way we decomposed $\hat u$, introducing the three associated parameters $\hat a'$, $\hat b'$ and $\hat z'$, and the profile $\hat v'$:
$$
\hat Q^+=  Q+ (\delta +\hat a') e^{e_0t}\mathcal Y+\hat v'+\hat b' \Lambda Q+\hat z'.\nabla Q, \ \ \hat v' \in \text{Span} (\mathcal Y,\Psi_0,\Psi_1,...,\Psi_d)^{\perp}
$$
with the following bounds on $[0,t_0]$ from \fref{eq:Q+ exit} and \fref{eq:def t0 Q+}:
\be \label{eq:bound hatQ+}
|\hat b' |+|\hat z'|+|\hat a'|e^{e_0 t}+\parallel \hat v' \parallel_{\dot H^1}+\parallel \hat v' \parallel_{L^{\infty}} \lesssim \delta^2 e^{2e_0t}.
\ee
Our aim is to compare $\hat u $ and $\hat Q^+$ and we claim that at the exit time there holds:
\be \label{eq:caracterisation Texit}
\parallel \hat u(\hat T_{\text{exit}})-\hat Q^+(\hat T_{\text{exit}})\parallel_{\dot H^1} \lesssim \delta .
\ee
Similarly, in the case $a(T_{\text{trans}})=-\delta$, the above bound holds replacing $Q^+$ with $Q^-$.\\

\noindent{\bf step 5} Proof of \eqref{eq:caracterisation Texit}. The evolution of the difference $\hat u-\hat Q^+$ on $[1,\hat T_{\text{exit}}]$ is:
\be \label{eq:evolution hatu-hatQ+}
(\hat v-\hat v')_t+H(\hat v-\hat v')=-(\hat a-\hat a')_te^{e_0t}\mathcal Y-(\hat b-\hat b')_t\Lambda Q-(\hat z-\hat z')_t.\nabla Q+NL-NL',
\ee
where:
$$
NL:=f(\hat u)-f(Q)-f'(Q)\hat u \ \ \text{and} \ \ NL':=f(\hat Q^+)-f(Q)-f'(Q)\hat Q^+
$$
We define the weighted distance between $\hat u$ and $\hat Q^+$ on $[1,\hat T_{\text{exit}}]$ as:
\be \label{eq:def norme exit}
D := \underset{1\leq t\leq \hat T_{\text{exit}}}{\text{sup}} \frac{e^{-e_0t}}{\delta }\left(\parallel \hat v-\hat v'\parallel_{\dot H^1} +\frac{|\hat a-\hat a'|}{\delta}e^{e_0t}+|\hat b-\hat b'| +|\hat z-\hat z'|\right)
\ee
From \fref{eq:bound hatQ+} and \fref{eq:bound hatu}, \fref{eq:bound hatu2}, \fref{eq:NL7}, H\"older inequality and interpolation one gets the following bounds for the difference of nonlinear terms on $[1,\hat T_{\text{exit}}]$:
\bea \label{eq:bound NL exit}
\non &&\parallel NL-NL' \parallel_{L^2} \lesssim  \para |\hat u-\hat Q^+|(|\hat u-Q|^{p-1}+|\hat Q^+-Q|^{p-1})  \para_{L^2}\\
\non & \leq & \para \hat u-\hat Q^+\para_{L^{\frac{2d}{d-2}}}(\para |\hat u-Q|^{p-1} \para_{L^{d}}+\para |\hat Q^+-Q|^{p-1}\para_{L^d}) \\
\non & \leq & \para \hat u-\hat Q^+\para_{\dot H^1}(\para \hat u-Q\para_{L^{d(p-1)}}^{p-1}+\para \hat Q^+-Q\para_{L^{d(p-1)}}^{p-1}) \\
\non & \leq & \delta e^{e_0t} D\Bigl(\para \hat u-Q\para_{L^{\frac{2d}{d-2}}}^{\frac{p-1}{2}}\para \hat u-Q\para_{L^{\infty}}^{\frac{p-1}{2}}+\para \hat Q^+-Q\para_{L^{\frac{2d}{d-2}}}^{\frac{p-1}{2}}\para \hat Q^+-Q\para_{L^{\infty}}^{\frac{p-1}{2}}\Bigr) \\
\non & \leq & \delta e^{e_0t} D\Bigl(\para \hat u-Q\para_{\dot H^1}^{\frac{p-1}{2}}\para \hat u-Q\para_{L^{\infty}}^{\frac{p-1}{2}}+\para \hat Q^+-Q\para_{\dot H^1}^{\frac{p-1}{2}}\para \hat Q^+-Q\para_{L^{\infty}}^{\frac{p-1}{2}}\Bigr) \\
&\lesssim &\delta^{p}e^{pe_0 t} D .
\eea

\noindent \emph{Energy estimate for the difference of errors}. From \fref{eq:evolution hatu-hatQ+}, the orthogonality conditions \fref{eq:orthogonalite v1-v2} and the bound \fref{eq:bound NL1-NL2} on the nonlinear term one gets the following energy estimate:
$$
\begin{array}{r c l}
\frac{d}{dt}\left[\int (\hat v-\hat v')H(\hat v-\hat v') \right] &=& 2 \int H(\hat v-\hat v')(NL-NL')-2\int H(\hat v-\hat v')^2 \\
&\lesssim&  \parallel NL-NL'\parallel_{L^2}^2 \lesssim  D^2\delta^{2p}e^{2e_0pt}.
\end{array}
$$
From the coercivity of the linearized operator \fref{eq:coercivite}, we reintegrate in time the above inequality to obtain that on $[1,\hat T_{\text{exit}}]$:
\be \la{eq:exit:bd v}
\ba{r c l}
\para \hat v-\hat v' \para_{\dot H^1}^2 & \lesssim & \int (\hat v-\hat v' )H(\hat v-\hat v' ) \\
&\leq & \int (\hat v(1)-\hat v'(1) )H(\hat v(1)-\hat v'(1) ) +\int_1^t D^2\delta^{2p}e^{2e_0pt}dt \\
&\lesssim & \para \hat v(1) \para_{\dot H^1}^2+ \para \hat v'(1) \para_{\dot H^1}^2+ D^2\delta^{2p}e^{2e_0pt} \lesssim  \delta^4+ D^2\delta^{2p}e^{2e_0pt}
\ea
\ee
where we used \fref{eq:bound hatu} and \fref{eq:bound hatQ+}. Using \fref{eq:estimation hatTexit} this means:
\be \label{eq:contraction v exit}
\begin{array}{r c l}
\underset{1\leq t\leq \hat T_{\text{exit}}}{\text{sup}}  \parallel \hat v-\hat v' \parallel_{\dot H^1}\frac{e^{-e_0t}}{\delta } & \lesssim  & \underset{1\leq t\leq \hat T_{\text{exit}}}{\text{sup}}  \delta^{p-1}e^{(p-1)e_0 t}D+\underset{1\leq t\leq \hat T_{\text{exit}}}{\text{sup}} \frac{e^{-e_0 t}\delta^2}{\delta} \\
& \lesssim & \alpha^{p-1}D+\delta\\
\end{array}
\ee

\noindent \emph{Modulation equations for the differences of parameters}. We take the scalar products of \fref{eq:evolution hatu-hatQ+} with $\Lambda Q$, $\mathcal Y$ and $\nabla Q$, using the bound \fref{eq:bound NL exit} for the nonlinear term and obtain for any $1\leq t\leq \hat T_{\text{exit}}$:
\bee
&&\left|\frac{d}{dt}\left[\hat b-\hat b'+\int (\hat v-\hat v')\Lambda Q \right] \right|+\left|\frac{d}{dt}\left[\hat z-\hat z'+\int (\hat v-\hat v') \nabla Q \right] \right|+|( \hat a-\hat a')_t|e^{e_0t} \\
&\lesssim & \delta^p e^{pe_0 t} D .
\eee
We integrate in time the two above equations on $[1,\hat T_{\text{exit}}]$, using \fref{eq:exit:bd v}, \fref{eq:bound hatu} and \fref{eq:bound hatQ+}, giving for $t\in [1,\hat T_{\text{exit}}]$:
$$
\ba{r c l}
&|\hat b(t)-\hat b'(t)|+|\hat z(t)-\hat z'(t)|+|\hat a(t)-\hat a'(t)|e^{e_0t} \\
\lesssim &|\hat b(1)-\hat b'(1)|+|\hat z(1)-\hat z'(1)|+|\hat a(1)-\hat a'(1)|e^{e_0t}+e^{e_0t}\int_1^t \delta^p e^{(p-1)e_0 t} Ddt\\
&+ |\int (\hat v(t)-\hat v'(t))\Lambda Q |+ |\int (\hat v(t)-\hat v'(t))\Lambda Q |+\int_1^t \delta^p e^{pe_0 t} Ddt \\
\lesssim & \delta^2+\delta^2e^{e_0t} +\para \hat v(t)-\hat v'(t) \para_{\dot H^1}+ \delta^p e^{pe_0 t} D\lesssim \delta^2e^{e_0t} + \delta^p e^{pe_0 t} D.
\ea
$$
From this, we deduce using \fref{eq:estimation hatTexit} that:
\be \label{eq:contraction exit1}
 \underset{1\leq t\leq \hat T_{\text{exit}}}{\text{sup}} (|\hat b-\hat b'|+|\hat z-\hat z'|+|\hat a-\hat a'|e^{e_0t})\frac{e^{-e_0t}}{\delta} \lesssim \delta+\alpha^{p-1}D.
\ee

\noindent \emph{End of the proof of \fref{eq:caracterisation Texit}}. From the definition \fref{eq:def norme exit} of the weighted norm of the difference and the estimates \fref{eq:contraction exit1} and \fref{eq:contraction 2} one obtains:
$$
D\lesssim \delta +\alpha^{p-1}D.
$$
We then conclude that $D\lesssim \delta$. From this fact and the definition \fref{eq:def norme exit} of $D$, the definition \fref{eq:estimation hatTexit} of $\hat T_{\text{exit}}$, at time $\hat T_{\text{exit}}$ there holds:
$$
\left(|\hat b-\hat b'|+|\hat z-\hat z'|+|\hat a-\hat a'|e^{e_0\hat T_{\text{exit}}}+\parallel \hat v-\hat v'\parallel_{\dot H^1}\right)\lesssim D\delta e^{e_0\hat T_{\text{exit}}}\lesssim \delta \alpha
$$
which implies the estimate \fref{eq:caracterisation Texit} we had to prove.

\noindent \emph{End of the proof}. From \fref{eq:def hatQ+} and \fref{eq:estimation hatTexit} one has:
$$
\hat Q^+(\hat T_{\text{exit}} )=Q^+\left(\text{log}\left( \left(\frac{\alpha+O(\alpha^2)}{\epsilon}\right)^{\frac{1}{e_0}}  \right) \right)
$$
which implies that there exists $C>0$ such that:
$$
\hat Q^+(\hat T_{\text{exit}} )\in \cup_{-C\leq \mu \leq C} Q^+\left(\text{log}\left( \left(\frac{\alpha+\mu \alpha}{\epsilon}\right)^{\frac{1}{e_0}}  \right) \right)=:\mathcal K.
$$
$\mathcal K$ is a compact set of $\dot H^1$ functions that will explode according to type I blow up forward in time, and $\mathcal K$ does not depend on $\delta$. We notice that as $\delta \rightarrow 0$, from \fref{eq:caracterisation Texit}:
$$
\underset{f\in \mathcal K}{\text{inf}} \parallel \hat u (\hat T_{\text{exit}}) -f\parallel_{\dot H^1} \lesssim \delta \rightarrow 0 \ \text{uniformly} \ \text{as} \ \delta \rightarrow 0.
$$ 
As the set of functions exploding with type I blow up forward in time is an open set of $\dot H^1$ from Proposition \ref{pr:type I} and Proposition \ref{pr:cauchy}, for $\delta$ small enough, $\hat u$ is blowing up with type I blow up forward in time. By the symmetries of the equation this means that $u$ is also blowing up with type I blow up forward in time.\\

The very same reasoning applies if $a(T_\text{trans})=-\delta$, and in that case at the exit time the solution is arbitrarily close in $\dot H^1$ to a compact set of renormalized versions of $Q^-$ going to $0$ in the energy topology because of dissipation. As the set of initial data such that the solution goes to $0$ in $\dot H^1$ as $t\rightarrow +\infty$ is also an open set of $\dot H^1$, one obtains that $u$ is global with $\parallel u \parallel_{\dot H^1}\rightarrow 0$ as $t\rightarrow +\infty$. This ends the proof of Theorem \ref{th:main}.

\end{proof}


\section{Stability of type I blow up} \la{sec:type I}

In this section we give some properties of solutions blowing up with type I blow up, and we prove the stability of this behavior. We follow the lines of \cite{Fe} where the authors prove it in the energy subcritical case $1<p<\frac{d+2}{d-2}$. Their proof adapts almost automatically but we write a proof here for the sake of completeness and for some estimates are more subtle in the energy critical case, mainly Proposition \ref{pr:energy to Linfty}. We recall that a solution $u$ of \fref{eq:NLH} blowing up at time $T$ is said to blow up with type I if it satisfies \fref{intro:def type I}.

\begin{proposition}[Stability of type I blow-up] \label{pr:type I}

Let $u_0\in W^{3,\infty}$ be an initial datum such that the solution $u$ of \fref{eq:NLH} starting from $u_0$ blows up with type I. Then there exists $\delta=\delta(u_0)>0$ such that for any $v_0\in W^{3,\infty}$ with 
\be \label{eq:proximite dotH1}
\parallel u_0-v_0\parallel_{W^{3,\infty}}\leq \delta 
\ee
the solution $v$ of \fref{eq:NLH} starting from $v_0$ blows up with type I.

\end{proposition}

\begin{remark}

The topology $W^{3,\infty}$ is convenient for our purpose but is not essential because of the parabolic regularizing effects, see Proposition \ref{pr:cauchy}.

\end{remark}

The section is organized as follows. In Subsection \ref{sub:typeI} we recall without proof some already known facts on type I blow-up, then we introduce the self-similar renormalization at a blow up point in Subsection \ref{sub:autosim}, before giving the proof of Proposition \ref{pr:type I} in Subsection \ref{sub:preuve}.


\subsection{Properties of type I blowing-up solution} \la{sub:typeI}

A point $x\in \mathbb R^d$ is said to be a blow up point for $u$ blowing up at time $T$ if there exists $(t_n,x)\rightarrow (T,x)$ such that:
$$
|u(t_n,x_n)|\rightarrow + \infty \ \ \text{as} \ n\rightarrow +\infty.
$$
A fundamental fact is the rigidity for solutions satisfying the type I blow up estimate \fref{intro:def type I} that are global backward in time, this is the result of Proposition \ref{lem:liouville autosimilaire}. Then in Lemma \ref{lem:description type I} we give a precise description of type I blow up, with an asymptotic at a blow up point and an ODE type caracterization.

\begin{proposition}[Liouville type theorem fot type I blow up \cite{MeZa2,Me9} ] \label{lem:liouville autosimilaire}

If $u$ be a solution of \fref{eq:NLH} on $(-\infty,0]\times \mathbb R^d$ such that $\para u\para_{L^{\infty}}\leq C(-t)^{\frac{1}{p-1}}$ for some constant $C>0$, then there exists $T\geq 0$ such that $u=\frac{\kappa}{(T-t)^{\frac{1}{p-1}}}$ where $\kappa$ is defined in \fref{intro:def kappa}.

\end{proposition}

\begin{lemma}[Description of type I blow up \cite{Gi3}, \cite{MeZa2}, \cite{Me9}] \label{lem:description type I}

Let $u$ solve \fref{eq:NLH} with $u_0\in W^{2,\infty}$ blowing up at $T>0$. The three following properties are equivalent:
\bea
\non &(i) & \text{The} \ \text{blow-up} \ \text{is} \ \text{of} \ \text{type} \ I. \\
\la{eq:car typeI ODE} & (ii) & \exists K>0, \ \ |\Delta u|\leq \frac{1}{2}|u|^p +K \ \text{on} \ \mathbb R^d \times [0,T) . \\
\la{eq:prop typeI} & (iii) & \parallel u \parallel_{L^{\infty}} (T-t)^{\frac{1}{p-1}} \rightarrow \kappa \ \text{as} \ t\rightarrow T .
\eea
Moreover, if $u$ blows up with type I at $x$ then:
\be \la{eq:prop typeI x}
 |u(t,x)|\sim \frac{\kappa}{(T-t)^{\frac{1}{p-1}}} \ \ \text{as} \ t\rightarrow T ,
\ee
and if $u_n(0)\rightarrow u(0)$ in $W^{2,\infty}$, for large $n$, $u_n$ blows up at time $T_n$ with $T_n\rightarrow T$.

\end{lemma}

The above results are stated in \cite{Fe}, \cite{Gi3}, \cite{MeZa2} and \cite{Me9} in the case $1<p<\frac{d+2}{d-2}$. They are however still valid in the energy critical case because the main argument is that the only bounded stationary solutions of \fref{eq:w} are $\kappa$, $-\kappa$ and $0$, which is still true for $p=\frac{d+2}{d-2}$. Indeed for a bounded stationary solution of \fref{eq:w}, the Pohozaev identity
$$
(d+2-p(d-2))\int_{\mathbb R^d} |\nabla w|^2e^{-\frac{|y|^2}{4}}dy+\frac{p-1}{2} \int_{\mathbb R^d}|y|^2|\nabla w |^2e^{-\frac{|y|^2}{4}}dy=0
$$
gives that $\nabla w=0$ if $1<p\leq \frac{d+2}{d-2}$. Hence $w$ is constant in space, meaning that it is one of the aforementioned solutions.


\subsection{Self-similar variables} \la{sub:autosim}

We follow the method introduced in \cite{Gi1}, \cite{Gi2}, \cite{Gi3} to study type I blow-up locally. The results the ideas of their proof are either contained in \cite{Gi2} or similar to the results there. A sharp blow-up criterion and other preliminary bounds are given by Lemma \ref{lem:gk} and a condition for local boundedness is given in Proposition \ref{pr:energy to Linfty}. For $u$ defined on $[0,T_{u_0})\times \mathbb R^d$, $a \in \mathbb R^d$ and $T>0$ we define the self-similar renormalization of $u$ at $(T,a)$:
\be \label{eq:def w}
w_{a,T}(y,t):=(T-t)^{\frac{1}{p-1}}u(t, a+\sqrt{T-t}y)
\ee
for $(t,y)\in [0,\text{min}(T_{u_0},T))\times \mathbb R^d$. Introducing the self-similar renormalized time:
\be \label{eq:def s autosim}
s:=-\text{log}(T-t)
\ee
one sees that if $u$ solves \fref{eq:NLH} then $w_{a,T}$ solves:
\be \label{eq:w}
\partial_s w_{a,T}-\Delta w_{a,T} -|w_{a,T}|^{p-1}w_{a,T} +\frac{1}{2}\Lambda w_{a,T}=0 .
\ee
Equation \fref{eq:w} admits a natural Lyapunov functional,
\be \label{eq:def E}
E(w)=\int_{\mathbb R^d} \left(\frac 1 2 |\nabla w(y)|^2+\frac{1}{2(p-1)}|w(y)|^2-\frac{1}{p+1}|w(y)|^{p+1} \right)\rho (y)dy ,
\ee
where $\rho (y):=\frac{1}{(4\pi)^{\frac d 2}}e^{-\frac{|y|^2}{4}}$ from the fact that for its solutions there holds:
\be \la{eq:decroissance energy}
\frac{d}{ds} E(w)=-\int_{\mathbb R^d}w_s^2\rho dy \leq 0 .
\ee
Another quantity that will prove to be helpful is the following:
\be \label{eq:def I}
I(w):=-2E(w)+\frac{p-1}{p+1} \left( \int_{\mathbb R^d} w^2\rho dy\right)^{\frac{p+1}{2}}.
\ee

\begin{lemma}[\cite{Gi1}, \cite{MeZa2}] \la{lem:gk}

Let $w$ be a global solution of \fref{eq:w} with $E(w(0))=E_0$, then\footnote{From the definition \fref{eq:def I} of $I$ and \fref{eq:blow up criterion} one has that for all $s\geq 0$, $E(w(s))\geq 0$. Hence the right hand side in \fref{eq:bound energy ws} is nonnegative.} for $s\geq 0$:
\be \la{eq:blow up criterion}
I(w(s))\leq 0 ,
\ee
\be \la{eq:bound energy ws}
\int_0^{+\infty} \int_{\mathbb R^d} w_s^2\rho dyds\leq E_0 .
\ee
If moreover $E_0:=E(w(0))\leq 1$, then\footnote{Idem for the right hand side of \fref{eq:bound energy w} and \fref{eq:bound energy nablaw}.} for any $s\geq 0$:
\be \la{eq:bound energy w}
\int_{\mathbb R^d} w^2\rho  dy \leq C E_0^{\frac{2}{p+1}},
\ee
\be \la{eq:bound energy nablaw}
\int_s^{s+1} \left( \int_{\mathbb R^d} ( |\nabla w|^2+w^2+|w|^{p+1})\rho dy \right)^2ds \leq C E_0^{\frac{p+3}{p+1}}.
\ee

\end{lemma}

\begin{remark}

If $I(w(s))>0$ holds for $w_{a,T}$ associated by \fref{eq:def w} to a solution $u$ of \fref{eq:NLH}, then $u$ blows up before $T$ from \fref{eq:def s autosim} since $w$ is not global from \fref{eq:blow up criterion}.

\end{remark}

\begin{proof} [Proof of Lemma \ref{lem:gk}]

\noindent \textbf{step 1} Proof of \fref{eq:blow up criterion}. We argue by contradiction and assume that $I(w(s_0))>0$ for some $s_0\geq 0$. The set $\mathcal S:=\{ s\geq s_0, \ I(s)\geq I(s_0) \}$ is closed by continuity. For any solution of \fref{eq:w} one has:
\be \la{eq:variation L2}
\frac{d}{ds} \left( \int_{\mathbb R^d} w^2\rho dy\right) = 2 \int_{\mathbb R^d} w w_s \rho dy = -4E(w)+\frac{2(p-1)}{p+1} \int_{\mathbb R^d} |w|^{p+1} \rho dy .
\ee
Therefore, for any $s\in \mathcal S$, from \fref{eq:def I} and Jensen inequality this gives:
\be \la{eq:croissance w2 bu}
\frac{d}{ds} \left( \int_{\mathbb R^d} w^2\rho dy\right) \geq -4E(w(s))+\frac{2(p-1)}{p+1} \left(\int_{\mathbb R^d} w^2 \rho dy\right)^{\frac{p+1}{2}} = I(w(s))>0 
\ee
as $I(w(s))\geq I(w(s_0))$ which with \fref{eq:decroissance energy} and \fref{eq:def I} imply $\frac{d}{ds} I(w(s))>0$. Hence $\mathcal S$ is open and therefore $\mathcal S=[s_0,+\infty)$. From \fref{eq:croissance w2 bu} and \fref{eq:decroissance energy} there exists $s_1$ such that $E(w(s))\leq \frac{p-1}{2(p+1)} \left(\int_{\mathbb R^d} w^2 \rho dy\right)^{\frac{p+1}{2}}$ for all $s\geq s_1$, implying from \fref{eq:croissance w2 bu}:
$$
\frac{d}{ds} \left( \int_{\mathbb R^d} w^2\rho dy\right) \geq 2\frac{p-1}{p+1} \left(\int_{\mathbb R^d} w^2 \rho dy\right)^{\frac{p+1}{2}} .
$$
This quantity must then tend to $+\infty$ in finite time, which is a contradiction.\\

\noindent \textbf{step 2} End of the proof. \fref{eq:bound energy ws} and \fref{eq:bound energy w} are consequences of \fref{eq:decroissance energy}, \fref{eq:def I} and \fref{eq:blow up criterion}. To prove \fref{eq:bound energy nablaw}, from \fref{eq:variation L2}, \fref{eq:decroissance energy}, \fref{eq:bound energy w} and H\"older one obtains:
$$
\int_s^{s+1}  \left( \int_{\mathbb R^d} |w|^{p+1} \rho dy \right)^2 ds \leq  \int_s^{s+1} \left(C E_0^2+C\int_{\mathbb R^d} w_s^2\rho dy\int_{\mathbb R^d} w^2\rho dy\right)ds \leq CE_0^{\frac{p+3}{p+1}}
$$
as $E_0\leq 1$. This identity, using \fref{eq:def E}, \fref{eq:decroissance energy} and as $E_0\leq 1$ implies \fref{eq:bound energy nablaw}.

\end{proof}

\begin{proposition}[Condition for local boundedness] \la{pr:energy to Linfty}

Let $R>0$, $0<T_-<T_+$ and $\delta>0$. There exists $\eta>0$ and $0<r\leq R$ such that for any $T\in [T_-,T_+]$ and $u$ solution of \fref{eq:NLH} on $[0,T)\times \mathbb R^d$ with $u_0\in W^{2,\infty}$ satisfying:
\be \la{eq:bd E waT}
\forall a\in B(0,R), \ \ E(w_{a,T}(0, \cdot))\leq \eta,
\ee
\be \la{eq:bd type I u}
\forall (t,x) \in [0,T)\times \mathbb R^d, \ \ |\Delta u(t,x)|\leq \frac 1 2 |u(t,x)|^p+\eta,
\ee
there holds
\be \la{eq:bd W2inftyloc u}
\forall t \in \left[\frac{T_-}{2},T\right), \ \ \para u(t) \para_{W^{2,\infty}(B(0,r))} \leq \delta.
\ee

\end{proposition}

The proof Proposition \ref{pr:energy to Linfty} is done at the end of this subsection. We need intermediate results: Proposition \ref{pr:energy to Linfty autosimilaire} gives local smallness in self-similar variables, Lemma \ref{lem:subtypeI Linfty} and its Corollary \ref{cor:subtypeI Linfty} give local boundedness in $L^{\infty}$ in original variables.

\begin{proposition} \la{pr:energy to Linfty autosimilaire}

For any $R,s_0,\delta>0$ there exists $\eta>0$ such that for any $w$ global solution of \fref{eq:w}, with $w(0)\in W^{2,\infty}(\mathbb R^d)$ satisfying
\be \la{eq:bound energy}
E(w(0))\leq \eta \ \ \text{and} \ \ \forall (s,y)\in [0,+\infty)\times \mathbb R^d, \  |\Delta w(s,y)|\leq \frac 1 2 |w(s,y)|^{p}+\eta,
\ee
then there holds:
\be \la{eq:bound energy to Linftyloc}
\forall (s,y) \in [s_0,+\infty)\times B(0,R), \ \ |w(s,y)|\leq \delta .
\ee

\end{proposition}

\begin{proof}[Proof of Proposition \ref{pr:energy to Linfty}]

It is a direct consequence of Lemma \ref{lem:energy to LinftydotH1} and Lemma \ref{lem:H1 to Linfty}.

\end{proof}

\begin{lemma} \la{lem:energy to LinftydotH1}

For any $R,s_0,\eta'>0$ there exists $\eta>0$ such that for $w$ a global solution of \fref{eq:w}, with $w(0)\in W^{2,\infty}(\mathbb R^d)$, satisfying \fref{eq:bound energy}, there holds
\be \la{eq:bound energy to LinftydotH1}
\forall s \in [s_0,+\infty), \ \ \int_{B(0,R)} (|w|^2+|\nabla w|^2)dy \leq \eta'.
\ee

\end{lemma}

\begin{lemma} \la{lem:H1 to Linfty}

For any $R,\delta>0$, $0<s_0<s_1$ there exists $\eta,\eta'>0$ and $0< r\leq R$ such that for $w$ a global solution of \fref{eq:w} with $w(0)\in W^{2,\infty}$, satisfying \fref{eq:bound energy} and \fref{eq:bound energy to LinftydotH1} there holds:
\be \la{eq:bound LinftyLinfty}
\forall (s,y)\in [ s_1,+\infty)\times B(0,r), \ \  |w(s,y)|\leq \delta .
\ee

\end{lemma}

We now prove the two above lemmas. In what follows we will often have to localize the function $w$. Let $\chi $ be a smooth cut-off function, $\chi =1$ on $B( 0,1)$ and $\chi=0$ outside $B(0,2)$. For $R>0$ we define $\chi_R(x)=\chi\left(\frac{x}{R}\right)$ and:
\be \la{eq:def v}
v:=\chi_R w 
\ee
(we will forget the dependence in $R$ in the notations to ease writing, and will write $\chi$ instead of $\chi_R$). From \fref{eq:w} the evolution of $v$ is then given by:
\be \la{eq:v}
v_s-\Delta v = \chi |w|^{p-1}w +\left( \left[\frac{1}{p-1}-\frac d 2 \right]\chi-\frac 1 2 \nabla \chi.y +\Delta \chi \right)w+\nabla . \left(\left[ \frac 1 2 \chi y -2\nabla \chi \right]w\right) 
\ee

\begin{proof}[Proof of Lemma \ref{lem:energy to LinftydotH1}]

We will prove that \fref{eq:bound energy to LinftydotH1} holds at time $s_0$, which will imply \fref{eq:bound energy to LinftydotH1} at any time $s\in [s_0,+\infty)$ because of time invariance.\\

\noindent \textbf{step 1} An estimate for $\Delta w$. First one notices that the results of Lemma \ref{lem:gk} apply. From \fref{eq:bound energy} and \fref{eq:w} there exists a constant $C>0$ such that:
$$
 |w|^{2p} \leq C(|w|^{p-1}w+\Delta w)^2+C \eta^2\leq C|w_s|^2+C|y|^2|\nabla w|^2+Cw^2+C\eta^2 .
$$
We integrate this in time, using \fref{eq:bound energy ws}, \fref{eq:bound energy w}, \fref{eq:bound energy nablaw} and \fref{eq:bound energy}, yielding for $s\geq 0$:
\be \la{eq:bound wL2pL2p}
\int_s^{s+1} \int_{B(0,2R)} |w|^{2p} dy ds \leq C \eta +C \eta^{\frac{p+3}{p+1}}+C\eta^{\frac{2}{p+1}}+C\eta^2\leq C\eta^{\frac{2}{p+1}}.
\ee
Injecting the above estimate in \fref{eq:bound energy}, using \fref{eq:bound energy w} and \fref{eq:bound energy nablaw} we obtain for $s\geq 0$:
\be \la{eq:bound DeltawL2L2}
\begin{array}{r c l}
& \int_s^{s+1} \para w \para_{H^2(B(0,2R))}^2ds \leq \int_s^{s+1} \int_{B(0,2R)}( |\Delta w|^2+|\nabla w|^2+w^2)dyds \\
\leq & \int_s^{s+1}\int_{B(0,2R)} C(|w|^{2p}+|\nabla w|^2+w^2)dyds +C\eta^2 \leq C\eta^{\frac{2}{p+1}}.
\end{array}
\ee

\noindent \textbf{step 2} Localization. We localize at scale $R$ and define $v$ by \fref{eq:def v}. From \fref{eq:def v}, \fref{eq:bound energy nablaw} and \fref{eq:bound energy w} one obtains that there exists $\tilde s_0 \in [\text{max}(0,s_0-1),s_0]$ such that:
\be \la{eq:bound tildes0 v H1}
\para v(\tilde s_0) \para_{H^1(\mathbb R^d)}^2\lesssim \int_{B(0,2R)} (w(\tilde s_0)^2+|\nabla w(\tilde s_0)|^2)dy\leq C \eta^{\frac{2}{p+1}}+C\eta^{\frac{p+3}{p+1}} \leq C \eta^{\frac{2}{p+1}} 
\ee
We apply Duhamel formula to \fref{eq:v} to find that $v(s_0)$ is given by:
\be \la{eq:Duhamel boot}
\ba{r c l}
v(s_0)& = &\int_{\tilde s_0}^{s_0} K_{s_0-s} *\left\{\chi |w|^{p-1}w +\left( \left[\frac{1}{p-1}-\frac d 2 \right]\chi-\frac 1 2 \nabla \chi.y +\Delta \chi \right)w\right\}ds \\
&&+\int_{\tilde s_0}^{s_0} \nabla .K_{s_0-s} * \left(\left[ \frac 1 2 \chi y -2\nabla \chi \right]w\right)ds +K_{{s_0-\tilde s_0}}*v(\tilde s_0)  .
\ea
\ee
We now estimate the $\dot H^1$ norm of each term in the previous identity, using \fref{eq:bound tildes0 v H1}, \fref{eq:bound energy nablaw}, \fref{eq:bd Kt}, Young and H\"older inequalities:
\be \la{eq:lin boot}
\para K_{{s_0-\tilde s_0}}*v(\tilde s_0) \para_{\dot H^1(\mathbb R^d)}\leq \para v(\tilde s_0) \para_{\dot H^1(\mathbb R^d)}\leq C\eta^{\frac{1}{p+1}},
\ee
\be \la{eq:extra lin boot}
\ba{r c l}
&\left\Vert \int_{\tilde s_0}^{s_0} K_{s_0-s} *\{( [\frac{1}{p-1}-\frac d 2 ]\chi-\frac{\nabla \chi.y}{2} +\Delta \chi )w\}+\nabla . K_{s_0-s} * ([ \frac{\chi y}{2} -2\nabla \chi ]w) \right\Vert_{\dot H^1}\\
\leq &  C \int_{\tilde s_0}^{s_0} \para w \para_{H^1(B(0,2R))}ds+C \int_{\tilde s_0}^{s_0} \frac{1}{|s_0-s|^{\frac 1 2}} \para w\para_{H^1(B(0,2R))}ds \\
\leq & C\eta^{\frac{p+3}{4(p+1)}}+C \left( \int_{\tilde s_0}^{s_0} \frac{ds}{|\tilde s_1-s|^{\frac 1 2 \times \frac 4 3}}\right)^{\frac{3}{4}} \left( \int_{\tilde s_0}^{s_0} \para w\para_{H^1(B(0,2R))}^4 ds\right)^{\frac 1 4} \leq C\eta^{\frac{p+3}{4(p+1)}} \\
\ea
\ee
For the non linear term in \fref{eq:Duhamel boot}, one first compute from \fref{eq:def v} that:
\be \la{eq:expression NL loc}
\nabla (\chi |w|^{p-1} w)=p\chi |w|^{p-1}\nabla w+\nabla \chi |w|^{p-1}w .
\ee
For the first term in the previous identity, using Sobolev embedding one obtains:
$$
\ba{r c l}
\para |w|^{p-1}\nabla w \para_{L^{\frac{2d}{d-2+(d-4)(p-1)}}(B(0,2R))} & \leq & C \para w \para^{p-1}_{L^{\frac{2d}{d-4}}(B(0,2R))} \para \nabla w \para_{L^{\frac{2d}{d-2}}(B(0,2R))}\\
& \leq & C \para w \para_{H^2(B(0,2R))}^p
\ea
$$
Therefore, from \fref{eq:bound DeltawL2L2} this force term satisfies:
$$
\int_{\tilde s_0}^{s_0} \para |w|^{p-1}\nabla w \para_{L^{\frac{2d}{d-2+(d-4)(p-1)}}(B(0,2R))}^{\frac{2}{p}}ds \leq \int_{\tilde s_0}^{s_0} \para w \para_{H^2(B(0,2R))}^2ds \leq C\eta^{\frac{2}{p+1}} . 
$$
We let $(q,r)$ be the Lebesgue conjugated exponents of $\frac{2}{p}$ and $\frac{2d}{(d-2)+(d-4)(p-1)}$:
$$
q=\frac{2}{2-p}>2, \ \ r=\frac{2d}{d+2-(d-4)(p-1)}>2 .
$$
They satisfy the Strichartz relation $ \frac{2}{q}+\frac{d}{r}=\frac d 2 $. Therefore, using \fref{strichartz inhomogene} one obtains:
$$
\ba{r c l}
&\left\Vert \int_{\tilde s_0}^{s_0} K_{s_0-s} *(p\chi |w(s)|^{p-1}\nabla w(s))ds \right\Vert_{L^2} \\
 \leq & C\left( \int_{\tilde s_0}^{s_0} \para |w|^{p-1}\nabla w \para_{L^{\frac{2d}{d-2+(d-4)(p-1)}}(B(0,2R))}^{\frac{2}{p}}ds\right)^{\frac p 2} \leq C \eta^{\frac{p}{(p+1)}} .
\ea
$$
For the second term in \fref{eq:expression NL loc} using \fref{eq:bound wL2pL2p}, \fref{eq:bd Kt} and H\"older one has:
$$
\left\Vert \int_{\tilde s_0}^{s_0} K_{s_0-s} *(\nabla \chi |w|^{p-1}w)ds \right\Vert_{L^2} \leq C \int_{\tilde s_0}^{s_0} \para w \para_{L^{2p}(B(0,2R))}^p \leq C \eta^{\frac{1}{p+1}}.
$$
The two above estimates and the identity \fref{eq:expression NL loc} imply the following bound:
$$
\left\Vert \int_{\tilde s_0}^{s_0} K_{s_0-s} *(\chi |w|^{p-1}w) ds \right\Vert_{\dot H^1} \leq C\eta^{\frac{1}{p+1}}
$$
We come back to \fref{eq:Duhamel boot} where we found estimates for each term in the right hand side in \fref{eq:lin boot}, \fref{eq:extra lin boot} and the above identity, yielding $\para v(s_0) \para_{\dot H^1}\leq C \eta^{\frac{1}{p+1}}$. From \fref{eq:def v}, as $v$ is compactly supported in $B(0,2R)$, the above estimate implies the desired estimate \fref{eq:bound energy to LinftydotH1} at time $s_0$.

\end{proof}

To prove Lemma \ref{lem:H1 to Linfty} we need the following parabolic regularization result.

\begin{lemma}[Parabolic regularization] \la{lem:parabolic}

Let $R,M>0$, $0< s_0\leq 1 $ and $w$ be a global solution of \fref{eq:w} satisfying:
\be \la{eq:petitesse w q}
\forall (s,y)\in [0,+\infty )\times \mathbb R^d, \ \ \para w(s,y) \para_{H^2(B(0,R))} \leq M .
\ee
Then there exists $0<r\leq R$, a constant $C=C(R,s_0)$ and $\alpha>1$ such that:
\be \la{eq:bootstrap w 1}
\forall (s,y)\in [s_0,+\infty)\times B(0,r), \ \ | w(s,y)| \leq C(M+M^{\alpha}) .
\ee

\end{lemma}

\begin{proof}[Proof of Lemma \ref{lem:parabolic}]

The proof is very similar to the proof of Lemma \ref{para:lem:para}, therefore we do not give it here.

\end{proof}

\begin{proof}[Proof of Lemma \ref{lem:H1 to Linfty}]

Without loss of generality we take $\eta'=\eta$, $s_0=0$, localize at scale $\frac R 2$ by defining $v$ by \fref{eq:def v}. The assumption \fref{eq:bound energy to LinftydotH1} implies that for $s\geq 0$:
\be \la{eq:bound LinftyH1v}
\int_{\mathbb R^d} (|v(s)|^2+|\nabla v(s)|^2)dy \leq C \eta .
\ee
We claim that for all $s\geq \frac{s_1}{2}$,
$$
\para v \para_{H^2} \leq C \eta .
$$
This will give the desired result \fref{eq:bound LinftyLinfty} by applying Lemma \ref{lem:parabolic} from \fref{eq:def v}. We now prove the above bound. By time invariance, we just have to prove it at time $\frac{s_1}{2}$.\\

\noindent \textbf{step 1} First estimate on $v_s$. Since $w$ is a global solution starting in $W^{2,\infty}(\mathbb R^d)$ with $E(w(0))\leq \eta$, from \fref{eq:bound energy ws} one obtains:
\be \la{eq:bound vs 1}
\int_0^{+\infty} \int_{\mathbb R^d} |v_s|^2dyds \leq C\eta . 
\ee

\noindent \textbf{step 2} Second estimate on $v_s$. Let $u=v_s$. From \fref{eq:w} and \fref{eq:def v} the evolution of $u$ is given by:
\be \la{eq:u}
u_s-\Delta u = p|w|^{p-1}u +\left( \left[\frac{1}{p-1}-\frac d 2 \right]\chi-\frac 1 2 \nabla \chi.y +\Delta \chi \right)w_s+\nabla . \left(\left[ \frac 1 2 \chi y -2\nabla \chi \right]w_s\right) .
\ee
We first state a non linear estimate. Using Sobolev embedding, H\"older inequality and \fref{eq:bound energy to LinftydotH1}, one obtains:
$$
\int_{\mathbb R^d} |u|^2|w|^{p-1}dy  \leq \para u \para_{L^{\frac{2d}{d-2}}(\mathbb R^d)}^2 \para w \para_{L^{\frac{2d}{d-2}}(B(0,R))}^{p-1} \leq C\eta^{\frac{p-1}{2}} \int_{\mathbb R^d}|\nabla u|^2 dy.
$$
We now perform an energy estimate. We multiply \fref{eq:u} by $u$ and integrate in space using Young inequality for any $\kappa>0$ and the above inequality:
$$
\ba{r c l}
\frac 1 2 \frac{d}{ds} \left[\int_{\mathbb R^d} |u|^2dy \right] & = & -\int_{\mathbb R^d} |\nabla u|^2dy+\int_{\mathbb R^d} \left( \left[\frac{1}{p-1}-\frac d 2 \right]\chi-\frac 1 2 \nabla \chi.y +\Delta \chi \right)w_sudy \\
&&+\int  \left(\left[ \frac 1 2 \chi y -2\nabla \chi \right]w_s\right).\nabla u dy +\int_{\mathbb R^d}u^2|w|^{2(p-1)}dy \\
&\leq & -\int_{\mathbb R^d} |\nabla u|^2dy+C\int_{B(0,R)}(w_s^2+u^2)dy+ \frac{C}{\kappa} \int_{B(0,R)} w_s^2dy \\
&&+C\kappa \int_{\mathbb R^d} |\nabla u|^2dy+C\eta^{\frac{p-1}{2}}\int_{\mathbb R^d}|\nabla u|^2dy \\
&\leq & -\int_{\mathbb R^d} |\nabla u|^2dy+C(\kappa) \int_{B(0,R)} w_s^2 dy
\ea
$$
if $\kappa$ and $\eta$ have been chosen small enough. Now because of the integrability \fref{eq:bound vs 1} there exists at least one $\tilde s \in [\text{max}(0,\frac{s_1}{2}-1),\frac{s_1}{2}]$ such that:
$$
\int_{\mathbb R^d} |v_s(\tilde s)|^2  dy\leq C(s_1) \eta.
$$
One then obtains from the two previous inequalities and \fref{eq:bound energy ws}:
\be \la{eq:bound vs 2}
\int_{\mathbb R^d} |v_s(s )|^2 dy \leq \int_{\mathbb R^d} |v_s(\tilde s)|^2 dy +C \int_{\tilde s}^{\frac{s_1}{2}} \int_{B(0,R)} w_s^2 dy ds \leq C \eta .
\ee

\noindent \textbf{step 3} Estimate on $\Delta v$. Applying Sobolev embedding and H\"older inequality, using the fact that $\left( \frac{\frac{2d}{d-4}}{2} \right)'=\frac{d}{4}=\frac{\frac{2d}{d-2}}{2(p-1)}$ one gets that for any $s\geq 0$:
\bea 
\non &&\int_{\mathbb R^d} v^2 |w|^{2(p-1)} dy \leq \para v^2 \para_{L^{\frac{\frac{2d}{d-4}}{2}}(\mathbb R^d)}  \para |w|^{2(p-1)} \para_{L^{\frac{\frac{2d}{d-2}}{2(p-1)}}(B(0,R))}   \\
\non &=& \para v \para_{L^{\frac{2d}{d-4}}(\mathbb R^d)}^2 \para w \para_{L^{\frac{2d}{d-2}}(B(0,R ))}^{2(p-1)} \leq  C \para v \para_{\dot H^2(\mathbb R^d)}^2 \para w \para_{H^1(B(0,R))}^{2(p-1)} \\
\la{eq:bound NL w} &\leq & C \eta^{p-1}  \int_{\mathbb R^d} |\Delta v|^2 dy
\eea
where injected the estimate \fref{eq:bound energy to LinftydotH1}. We inject the above estimate in \fref{eq:v}, using \fref{eq:def v}, yielding for all $s\geq 0$:
$$
\ba{r c l}
\int_{\mathbb R^d} |\Delta v|^2dy & \leq & C\left(\int_{\mathbb R^d} (|v_s|^2+ |w|^2+ |\nabla w|^2+v^2 |w|^{2(p-1)})dy \right) \\
 & \leq & C \int_{\mathbb R^d} |v_s|^2dy+C\eta +C\eta^{p-1}\int_{\mathbb R^d}|\Delta v|^2dy
\ea
$$
where we used \fref{eq:petitesse w q}. Injecting \fref{eq:bound vs 2}, for $\eta$ small enough:
\be \la{eq:bound Delta v}
\int_{\mathbb R^d} \left|\Delta v\left(\frac{s_1}{2}\right)\right|^2dy \leq C\int_{\mathbb R^d} \left|v_s\left(\frac{s_1}{2}\right)\right|^2 dy+C\eta \leq C\eta .
\ee

\noindent \textbf{step 4} Conclusion. From \fref{eq:bound LinftyH1v} and \fref{eq:bound Delta v} we infer $ \para v(\frac{s_1}{2})\para_{\dot H^2}\leq C\eta$ which is exactly the bound we had to prove.

\end{proof}

We now go from boundedness in $L^{\infty}$ in self-similar provided by Proposition \ref{pr:energy to Linfty autosimilaire} to boundedness in $L^{\infty}$ in original variables.

\begin{lemma}[\cite{Gi3}] \la{lem:subtypeI Linfty}

Let $0\leq a \leq \frac{1}{p-1}$ and $R,\epsilon_0>0$. Let $0<\epsilon \leq \epsilon_0$ and $u$ be a solution of \fref{eq:NLH} on $[-1,0)\times \mathbb R^d$ satisfying
\be \la{eq:bound hp subtypeI}
\forall (t,x)\in [-1,0) \times B(0,R), \ \ |u(t,x)|\leq \frac{\epsilon}{{|t|^{\frac{1}{p-1}-a}}} .
\ee
For $\epsilon_0$ small enough the following holds for all $(t,x)\in [-1,0)\times B\left(0,\frac R 2\right)$.
\bea
\la{eq:bound subtypeI Linfty}& \text{If} \ \ \frac{1}{p-1}-a< \frac 1 2, \ & \ \ |u(t,x)|\leq C(a)\epsilon . \\
\la{eq:bound subtypeI ln} & \text{If} \ \ \frac{1}{p-1}-a= \frac 1 2, \ & \ \ |u(t,x)|\leq C\epsilon (1+|\text{ln}(t)|) \\
\la{eq:bound subtypeI power} & \text{If} \ \ \frac{1}{p-1}-a> \frac 1 2, \ &  \ \ |u(t,x)|\leq \frac{C(a)\epsilon}{|t|^{\frac{1}{p-1}-a-\frac 1 2}}
\eea

\end{lemma}

Applying several times Lemma \ref{lem:subtypeI Linfty}, via scale change and time invariance, one obtains the following corollary.

\begin{corollary} \la{cor:subtypeI Linfty}

Let $R>0$ and $0<T_-<T_+$. There exists $\epsilon_0>0$, $0<r\leq R$ and $C>0$ such that the following holds. For any $0<\epsilon<\epsilon_0$, $T\in [T_-,T_+]$ and $u$ solution of \fref{eq:NLH} on $[0,T)\times \mathbb R^d$ satisfying
\be \la{eq:bound hp subtypeI 2}
\forall (t,x)\in [0, T)\times B(0,R), \ \ |u(t,x)|\leq \frac{\epsilon}{{(T-t)^{\frac{1}{p-1}}}}
\ee
one has:
\be \la{eq:bound subtypeI Linfty 2}
\forall (t,x)\in [0,T)\times B(0,r), \ \ |u(t,x)|\leq C\epsilon
\ee

\end{corollary}

To prove Lemma \ref{lem:subtypeI Linfty} we need two technical Lemmas taken from \cite{Gi3} whose proof can be found there.

\begin{lemma}[\cite{Gi3}]

Define for $0<\alpha<1$ and $0<\theta<h<1$ the integral $I(h)=\int_h^1 (s-h)^{-\alpha}s^{\theta}ds $. It satisfies
\bea
\la{eq:bd I >1} & \ \text{If} \ \alpha+\theta>1, \ & \ \ I(h)\leq \left( \frac{1}{1-\alpha}+\frac{1}{\alpha+\theta -1}\right) h^{1-\alpha-\theta}. \\
\la{eq:bd I =1} &\text{If} \ \alpha+\theta =1, \ & \ \ I(h)\leq  \frac{1}{1-\alpha}+|\text{log}(h)|. \\
\la{eq:bd I <1} &\text{If} \ \alpha+\theta <1, \ & \ \ I(h)\leq  \frac{1}{1-\alpha-\theta}.
\eea

\end{lemma}

\begin{lemma}[\cite{Gi3}]

If $y$, $r$ and $q$ are continuous functions defined on $[t_0,t_1]$ with 
$$
y(t)\leq y_0+ \int_{t_0}^t y(s)r(s) ds+\int_{t_0}^{t} q(s) ds
$$
for $t_0\leq t \leq t_1$, then for all $t_0\leq t\leq t_1$:
\be \la{eq:bd gronwall}
y(t)\leq e^{\int_{t_0}^t r(\tau)d\tau} \left[ y_0+\int_{t_0}^t q(\tau) e^{-\int_{t_0}^{\tau}r(\sigma)d\sigma}d\tau \right] .
\ee

\end{lemma}

\begin{proof}[Proof of Lemma \ref{lem:subtypeI Linfty}]

We first localize the problem, with $\chi$ a smooth cut-off function, with $\chi=1$ on $B\left(0,\frac R 2\right)$, $\chi=0$ outside $B(0,R)$ and $|\chi|\leq 1$. We define
\be \la{eq:def v subtypeI}
v:=\chi u
\ee
whose evolution, from \fref{eq:NLH}, is given by:
\be \la{eq:id vs}
v_t=\Delta v+|u|^{p-1}v+\Delta \chi u-2\nabla .(\nabla \chi u).
\ee
We apply Duhamel formula to \fref{eq:id vs} to find that for $t\in [-1,0)$:
\be \la{eq:id v}
v(t)=K_{t+1}*v(-1)+\int_{-1}^t K_{t-s}*(|u|^{p-1}v+\Delta \chi u-2\nabla .(\nabla \chi u)) ds .
\ee
From \fref{eq:bound hp subtypeI} and \fref{eq:def v subtypeI} one has for free evolution term:
\be \la{eq:bd free}
\para K_{t+1}*v(-1) \para_{L^{\infty}}\leq \epsilon .
\ee
We now find an upper bound for the other terms in the previous equation.\\

\noindent \textbf{step 1} Case (i). For the linear terms, as $\frac{1}{p-1}-a+\frac 1 2<1$, from \fref{eq:bd I <1} one has:
\be \la{eq:bd bord i}
\ba{r c l}
\para \int_{-1}^t K_{t-s}*(\Delta \chi u-2\nabla .(\nabla \chi u)) ds \para_{L^{\infty}} & \leq & C \int_{-1}^t \frac{1}{(t-s)^{\frac 1 2}} \para u \para_{L^{\infty}(B(0,R))} \\
&\leq & C\epsilon \int_{-1}^t \frac{1}{(t-s)^{\frac 1 2}}\frac{1}{|s|^{\frac{1}{p-1}-a}} \leq C(a) \epsilon .
\ea
\ee
For the nonlinear term, as $\frac{1}{p-1}-a<\frac 1 2<\frac{1}{2(p-1)}=\frac{d-2}{8}$ because $d\geq 7$ we compute using \fref{eq:bound hp subtypeI}:
\be \la{eq:bd NL i}
\ba{r c l}
\para \int_{-1}^t K_{t-s}*( \chi |u|^{p-1}v)ds \para_{L^{\infty}} & \leq & \int_{-1}^t \para u \para_{L^{\infty}(B(0,R))}^{p-1}\para v \para_{L^{\infty}} ds \\
&\leq & \epsilon^{p-1} \int_{-1}^t \frac{1}{|s|^{\frac 1 2}}\para v \para_{L^{\infty}} ds .
\ea
\ee
Gathering \fref{eq:bd free}, \fref{eq:bd bord i} and \fref{eq:bd NL i}, from \fref{eq:id v} one has:
$$
\para v(t) \para_{L^{\infty}} \leq C(a) \epsilon+\epsilon^{p-1}\int_{-1}^{t} \frac{1}{|s|^{\frac 1 2}}\para v \para_{L^{\infty}}.
$$
Applying \fref{eq:bd gronwall} one obtains:
$$
\para v(t) \para_{L^{\infty}} \leq C(a) \epsilon e^{ \int_{-1}^t|s|^{-\frac 1 2}ds} \leq C(a) \epsilon
$$
which from \fref{eq:def v subtypeI} implies the bound \fref{eq:bound subtypeI Linfty} we had to prove.\\

\noindent \textbf{step 2} Case (ii). For the linear terms, as $\frac{1}{p-1}-a=\frac 1 2$, from \fref{eq:bd I =1} one has:
\be \la{eq:bd bord ii}
\ba{r c l}
\para \int_{-1}^t K_{t-s}*(\Delta \chi u-2\nabla .(\nabla \chi u)) ds \para_{L^{\infty}} & \leq & C \int_{-1}^t \frac{1}{(t-s)^{\frac 1 2}} \para u \para_{L^{\infty}(B(0,R))} \\
&\leq & C\epsilon \int_{-1}^t \frac{1}{(t-s)^{\frac 1 2}}\frac{1}{|s|^{\frac 1 2}} \leq C \epsilon(1+| \text{log} (t)|) .
\ea
\ee
For the nonlinear term, as $\frac{1}{p-1}-a<\frac 1 2<\frac{1}{2(p-1)}=\frac{d-2}{8}$ as $d\geq 7$, using \fref{eq:bound hp subtypeI}:
\be \la{eq:bd NL ii}
\ba{r c l}
\para \int_{-1}^t K_{t-s}*( \chi |u|^{p-1}v)ds \para_{L^{\infty}} & \leq &  \int_{-1}^t \para u \para_{L^{\infty}(B(0,R))}^{p-1}\para v \para_{L^{\infty}} \\
&\leq & \epsilon^{p-1} \int_{-1}^t \frac{1}{|s|^{\frac 1 2}}\para v \para_{L^{\infty}} .
\ea
\ee
Gathering \fref{eq:bd free}, \fref{eq:bd bord ii} and \fref{eq:bd NL ii}, from \fref{eq:id v} one has:
$$
\para v(t) \para_{L^{\infty}} \leq C\epsilon+C\epsilon |\text{log}(t)|+ \epsilon^{p-1} \int_{-1}^{t} \frac{1}{|s|^{\frac 1 2}}\para v \para_{L^{\infty}}.
$$
Applying \fref{eq:bd gronwall} one obtains:
$$
\ba{r c l}
\para v(t) \para_{L^{\infty}} \leq C\epsilon e^{\int_{-1}^t |s|^{-\frac 1 2}ds} \left[1 +\int_{-1}^t \frac{ds}{|s|}e^{-\int_{-1}^s |\tau|^{-\frac 1 2}d\tau} \right] \leq C\epsilon (1 +|\text{log}(t)|)
\ea
$$
which from \fref{eq:def v subtypeI} implies \fref{eq:bound subtypeI ln}.\\

\noindent \textbf{step 3} Case (iii). For the linear terms, as $\frac{1}{p-1}-a>\frac 1 2$, from \fref{eq:bd I =1} one has:
\be \la{eq:bd bord iii}
\ba{r c l}
\para \int_{-1}^t K_{t-s}*(\Delta \chi u-2\nabla .(\nabla \chi u)) ds \para_{L^{\infty}} & \leq & C \int_{-1}^t \frac{1}{(t-s)^{\frac 1 2}} \para u \para_{L^{\infty}(B(0,R))} \\
&\leq & C\epsilon \int_{-1}^t \frac{1}{(t-s)^{\frac 1 2}}\frac{1}{|s|^{\frac{1}{p-1}-a}} \leq \frac{C(a)\epsilon}{|t|^{\frac{1}{p-1}-a-\frac 1 2}}.
\ea
\ee
For the nonlinear term, as $\frac{1}{p-1}-a\leq \frac{1}{p-1}$ we compute using \fref{eq:bound hp subtypeI}:
\be \la{eq:bd NL iii}
\ba{r c l}
\para \int_{-1}^t K_{t-s}*( \chi |u|^{p-1}v)ds \para_{L^{\infty}} & \leq &  \int_{-1}^t \para u \para_{L^{\infty}(B(0,R))}^{p-1}\para v \para_{L^{\infty}}ds \\
&\leq & \epsilon^{p-1} \int_{-1}^t \frac{1}{|s|}\para v \para_{L^{\infty}} .
\ea
\ee
Gathering \fref{eq:bd free}, \fref{eq:bd bord iii} and \fref{eq:bd NL iii}, from \fref{eq:id v} one has:
$$
\para v(t) \para_{L^{\infty}} \leq \frac{C(a)\epsilon}{|t|^{\frac{1}{p-1}-a-\frac 1 2}}+ \epsilon^{p-1} \int_{-1}^{t} \frac{1}{|s|}\para v \para_{L^{\infty}}ds.
$$
Applying \fref{eq:bd gronwall} one obtains if $\epsilon^{p-1}<\frac{1}{p-1}-a-\frac 1 2$:
$$
\begin{array}{r c l}
\para v(t) \para_{L^{\infty}} & \leq & C(a)\epsilon e^{\epsilon^{p-1} \int_{-1}^t\frac{ds}{|s|}} \left[1+\int_{-1}^t \frac{1}{|s|^{\frac{1}{p-1}-a+\frac 1 2}}e^{-\epsilon^{p-1}\int_{-1}^s \frac{d\tau}{\tau}}ds \right]  \leq  \frac{C(a)\epsilon}{|t|^{\frac{1}{p-1}-a-\frac 1 2}} 
\end{array}
$$
implying \fref{eq:bound subtypeI power} from \fref{eq:def v subtypeI} .

\end{proof}

We can now end the proof of Proposition \ref{pr:energy to Linfty}.

\begin{proof}[Proof of Proposition \ref{pr:energy to Linfty}]

For any $a \in B(0,R)$, from \fref{eq:def w}, \fref{eq:bd E waT} and \fref{eq:bd type I u} $w_{a,T}$ satisfies $E(w_{a,T}(0,\cdot))\leq \eta$ and:
$$
|\Delta w_{a,T}|\leq \frac 1 2 |w_{a,T}|^p+\eta T_+^{\frac{p}{p-1}}.
$$
Applying Proposition \ref{pr:energy to Linfty autosimilaire} to $w_{a,T}$ one obtains that for any $\eta'>0$ if $\eta$ is small enough:
$$
\forall s\geq s\left(\frac{T_-}{4} \right), \ \  |w_{a,T}(s,0)|\leq \eta' .
$$
In original variables this means:
$$
\forall (t,x)\in B(0,R)\times [\frac{T_-}{4},T), \ \ |u(t,x)|\leq \frac{\eta'}{(T-t)^{\frac{1}{p-1}}}.
$$
Applying Corollary \ref{cor:subtypeI Linfty} for $\eta'$ small enough there exists $r>0$ such that
$$
\forall (t,x)\in B(0,R)\times [\frac{T_-}{4},T), \ \ |u(t,x)|\leq C \eta' .
$$
Then, a standard parabolic estimate, similar to these in Lemma \ref{para:lem:para} propagates this bound for higher derivatives, yielding the result \fref{eq:bd W2inftyloc u}.

\end{proof}


\subsection{Proof of Proposition \ref{pr:type I}} \la{sub:preuve}

We now prove Proposition \ref{pr:type I} by contradiction, following \cite{Fe}. Assume the result is false. From Lemma \ref{lem:description type I} and from the Cauchy theory in $W^{2,\infty}$ the negation of the result of Proposition \ref{pr:type I} means the following. There exists $u_0 \in W^{3,\infty}$ such that the solution of \fref{eq:NLH} starting from $u_0$ blows up at time $1$ with:
\be \la{eq:bd Linfty u}
\parallel u(t) \parallel \sim \kappa (1-t)^{-\frac{1}{p-1}} \ \text{as} \ t\rightarrow 1,
\ee
and satisfies:
\be \la{eq:bd ODE u}
|\Delta u|\leq \frac{1}{2}|u|^p +K \ \text{on} \ \mathbb R^d \times [0,1) .
\ee
There exists a sequence $u_n$ of solutions of \fref{eq:NLH} blowing up at time $T_n$ with:
\be \la{eq:lim Tn2}
T_n \rightarrow 1 \ \ \text{and} \ \ u_n\rightarrow u \  \text{in} \ \mathcal C_{\text{loc}} ([0,1),W^{3,\infty}( \mathbb R^d))
\ee
for any $0\leq T<1$ and there exists two sequences $0\leq t_n<T_n$ and $x_n$ such that:
\be \la{eq:bd ODE un 1}
|\Delta u_n |\leq \frac{1}{2}|u_n|^p +2K \ \text{on} \ \mathbb R^d \times [0,t_n) ,
\ee
\be \la{eq:bd ODE un 3}
 |\Delta u_n (t_n , x_n) |= \frac{1}{2}|u_n(t_n,x_n)|^p +2K .
\ee
The strategy is the following. First we centralize the problem, showing in Lemma \ref{lem:renormalisation} that one can assume without loss of generality $x_n=0$. One also obtains that $u$ and $ u_n$ become singular near $0$ as $t\rightarrow 1$ and $n\rightarrow +\infty$. In view of Lemma  \ref{lem:description type I}, the ODE type bound \fref{eq:bd ODE un 1} means that $u_n$ behaves approximately as a type I blowing up solution until $t_n$. This intuition is made rigorous by the second Lemma \ref{lem:typeI renorm}, stating that if we renormalize $u_n$ before $t_n$, one converges to the constant in space blow up profile associated to type I blow up and that the $L^{\infty}$ norm grows as during type I blow up. We end the proof by showing that the inequality \fref{eq:bd ODE un 3} then passes to the limit, contradicting \fref{eq:bd ODE u}.

\begin{lemma} \la{lem:renormalisation}

Let $u$, $u_n$ be solutions of \fref{eq:NLH}, $t_n$ and $x_n$ satisfy \fref{eq:bd Linfty u}, \fref{eq:bd ODE u}, \fref{eq:lim Tn}, \fref{eq:bd ODE un 1} and \fref{eq:bd ODE un 2}. Then: 
\be \la{eq:lim tn}
t_n\rightarrow 1
\ee
and there exist $\hat u$ and $\hat u_n$ solutions of \fref{eq:NLH} satisfying \fref{eq:bd Linfty u}, \fref{eq:bd ODE u}, \fref{eq:bd ODE un 1} and \fref{eq:bd ODE un 2} with $\hat x_n=0$, $\hat u (t_n,0)\rightarrow +\infty$. In addition, $\hat u$ blows up with type I at $(1,0)$ and $\hat u_n$ blows up at time $T_n$. One has the following asymptotics:
\be \la{eq:lim Tn}
\para \hat u _n(0)\para_{W^{2,\infty}}\lesssim 1, \ \  T_n \rightarrow 1 \ \ \text{and} \ \ u_n\rightarrow u \  \text{in} \ \mathcal C_{\text{loc}}^{1,2} ([0,1)\times \mathbb R^d)
\ee

\end{lemma}

\begin{proof}[Proof of Lemma \ref{lem:renormalisation}]

\noindent \textbf{step 1} Proof of \fref{eq:lim tn}. At time $t_n$, $u$ satisfies the inequality \fref{eq:bd ODE u} whereas $u_n$ doesn't from \fref{eq:bd ODE un 3}. As $u_n$ converges to $u$ in $C^{1,2}_{\text{loc}}([0,1)\times \mathbb R^d)$ from \fref{eq:lim Tn} this forces $t_n$ to tend to $1$.\\

\noindent \textbf{step 2} Centering and limit objects. Define $\hat u_n(t,x)=u_n(t,x+x_n)$. Then $\hat u_n$ is a solution satisfying \fref{eq:bd ODE un 1}, \fref{eq:bd ODE un 3} with $\hat x_n=0$, and blowing up at time $T_n\rightarrow 1$ from \fref{eq:lim Tn}. From the Cauchy theory, see Proposition \ref{pr:cauchy}, $(t,x)\mapsto u(t,x_n+x)$ is uniformly bounded in $C^{\frac 3 2,3}_{\text{loc}}([0,1),\mathbb R^d)$, hence as $n\rightarrow +\infty$ using Arzela Ascoli theorem it converges to a function $\hat u$ that also solves \fref{eq:NLH}, satisfies \fref{eq:bd ODE u} and 
\be \la{eq:bd Linfty hatu}
\parallel \hat u(t) \parallel \lesssim \kappa (1-t)^{-\frac{1}{p-1}} .
\ee
As $u_n$ converges to $u$ in $C_{\text{loc}} ([0,1),W^{3,\infty}( \mathbb R^d))$ from \fref{eq:lim Tn}, $\hat u_n$ converges to $\hat u$ in $C_{\text{loc}}^{1,2} ([0,1)\times \mathbb R^d))$, establishing \fref{eq:lim Tn}.\\

\noindent \textbf{step 3} Conditions for boundedness. We claim two facts. 1) If $\hat u$ does not blow up at $(1,0)$ then there exists $r,C>0$ such that for all $(t,y)\in [0,t_n]\times B(0,r)$, $|\hat u_n(t,y)|\leq C$. 2) If there exists $C>0$ such that for all $0\leq t \leq t_n$, $|\hat u_n(t,0)|\leq C$ then $\hat u$ does not blow up at $(0,1)$.

\noindent \emph{Proof of the first fact}. We reason by contradiction. If $\hat u$ does not blow up at $(1,0)$ there exists $r,C>0$ such that for all $(t,y)\in [0,1)\times B(0,r)$, $|\hat u(t,y)|\leq C$. Assume that there exists $(\tilde x_n,\tilde t_n)$ such that $\tilde x_n\in B(0,r)$ and $|\hat u_n (\tilde x_n,\tilde t_n)|\rightarrow +\infty$. As $\hat u_n$ solves \fref{eq:NLH}, from \fref{eq:bd ODE un 3} one then has that:
$$
\forall t\in [0,\tilde t_n], \ \ \partial_t |\hat u_n (t,\tilde x_n)|\leq \frac 3 2 |\hat u_n(t,\tilde x_n)|^p+2K, \ \ |\hat u_n (\tilde x_n,\tilde t_n)|\rightarrow +\infty
$$
This then implies that for any $M>0$ there exists $s>0$ such that for $n$ large enough, $|\hat u_n(\tilde x_n,t)|\geq M$ on $[\text{max}(0,\tilde t_n-s),\tilde t_n]$. But this contradicts the convergence in $C_{\text{loc}}([0,1)\times B(0,r))$ established in Step 2 to the bounded function $\hat u$.

\noindent \emph{Proof of the second fact}. We also prove it by contradiction. Assume that $\hat u$ blows up at $(0,1)$ and $|\hat u_n(t_n,0)|\leq C$. Then we claim that 
$$
\forall t\in [0,t_n), \ \ | \hat u_n(t,0)|\leq \text{max}((4K)^{\frac 1 p},C)
$$
Indeed, as $\hat u_n$ is a solution of \fref{eq:NLH} satisfying \fref{eq:bd ODE un 1} one has that:
$$
\forall t\in [0,t_n], \ \ \partial_t |\hat u_n(t,0)| \geq \frac 1 2 |\tilde \hat u_n (t,0)|^p-2K.
$$
So if the bound we claim is violated at some time $0\leq t_0\leq \tau_n'$, then $| \hat u_n (t,0)|$ is non decreasing on $[t_0,\tau_n']$, strictly greater than $C$, which at time $t_n$ is a contradiction. But now as this bound is independent of $n$, valid on $[0,t_n)$ with $t_n\rightarrow 1$, and as $\hat u_n (t,0)\rightarrow \hat u(t,0)$ on $[0,1)$ one obtains at the limit that $\hat u (t,0)$ is bounded on $[0,1)$. From \fref{eq:prop typeI x} this contradicts the blow up of $\hat u$ at $(1,0)$. \\

\noindent \textbf{step 4} End of the proof. It remains to prove the singular behavior near $0$: that $\hat u$ blows up at $(1,0)$ and that $|\hat u_n(t_n,0|\rightarrow +\infty$. We reason by contradiction. From Step 3 we assume that there exists $C,r>0$ such that $|\hat u|+|\hat u_n |\leq C$ on $[0,1)\times B(0,r)$. A standard parabolic estimate, similar to these in Lemma \ref{para:lem:para} then implies that 
\be \la{eq:bd W3infty hatu hatun}
\para \hat u(t) \para_{W^{3,\infty}(B(0,r'))}+ \para \hat u_n(t) \para_{W^{3,\infty}(B(0,r'))}+\leq C'
\ee
for all $t\in [\frac 1 2, 1)$ for some $0<r'\leq r$ and. Let $\chi$ be a cut-off function, $\chi=1$ on $B(0,\frac{r'}{2})$, $\chi=0$ outside $B(0,r')$. The evolution of $\tilde u_n=\chi \hat u_n$ is given by:
$$
\tilde u_{n,\tau}-\Delta \tilde u_n = \chi |\hat u_n|^{p-1}\hat u_n +\Delta \chi \hat u_n-2\nabla . \left( \nabla \chi \hat u_n\right) =F_n
$$
with $\para F_n\para_{W^{1,\infty}}\leq C$ from \fref{eq:bd W2loc vn}. Fix $0<s\ll 1$. One has:
$$
\begin{array}{r c l}
\Delta \hat u_n(t_n,0) & = & K_{s}*(\Delta \tilde u_n(t_n-s))(0)+\sum_1^d \int_0^s \left[\partial_{x_i}K_{s-s'}*\partial_{x_i}F(t_n-s+s')\right](0) \\
&=& \Delta \hat u (t_n-s,0)+o_{n\rightarrow +\infty}(1)+o_{s\rightarrow 0}(1)
\end{array}
$$
from \fref{eq:lim Tn}, the estimate on $F_n$ and \fref{eq:bd W3infty hatu hatun}. Similarly, 
$$
\hat u_n(t_n,0) = \hat u (t_n,0)+o_{n\rightarrow +\infty}(1)+o_{s\rightarrow 0}(1).
$$
The equality \fref{eq:bd ODE un 3} and the two above identities imply the following asymptotics: $\text{lim} \ \text{inf} |\Delta \hat u(t_n)|-\frac{|\hat u(t_n,0)|^p}{2}\geq 2K$, which is in contradiction with \fref{eq:bd ODE u}. Hence $\hat u$ blows up at $(1,0)$ with type I blow up from \fref{eq:bd Linfty hatu} and $|\hat u(t_n,0)|\rightarrow +\infty$.

\end{proof}

We return to the study of $u$ and $u_n$ introduced at the begining of this subsection to prove Proposition \ref{pr:type I} by contradiction. From Lemma \ref{lem:renormalisation}, keeping the the notation $u$ and $u_n$ for $\hat u$ and $\hat u_n$ introduced there, one can assume without loss of generality that in addition to \fref{eq:bd Linfty u}, \fref{eq:bd ODE u} and \fref{eq:bd ODE un 1}, $u$ and $u_n$ satisfy \fref{eq:lim tn}, \fref{eq:lim Tn} and:
\be \la{eq:bd ODE un 2}
|\Delta u_n (t_n , 0) |= \frac{1}{2}|u_n(t_n,0)|^p +2K ,
\ee
\be \la{eq:lim un(tn,0)}
u_n(t_n,0)\rightarrow +\infty ,
\ee
\be \la{eq:lim u}
|u(t,0)|\sim \frac{\kappa}{(1-t)^{\frac{1}{p-1}}} .
\ee
To renormalize appropriately $u_n$ near $(1,0)$ we do the following. Define:
\be \la{eq:def Mn}
M_n(t):= \left(\frac{\kappa}{\para u_n(t) \para_{L^{\infty}}} \right)^{p-1} .
\ee
For $(\tilde t_n)_{n\in \mathbb N}$ a sequence of times, $0\leq \tilde t_n<T_n$, the renormalization near $(\tilde t_n,0)$ is:
\be \la{eq:def vn}
v_n(\tau,y):=M_n^{\frac{1}{p-1}}(\tilde t_n)u_n\left(M_n^{\frac 1 2}(\tilde t_n)y,\tilde t_n+\tau M_n(\tilde t_n)\right)
\ee
for $(\tau,y) \in[-\frac{\tilde t_n}{M_n(\tilde t_n)},\frac{T_n-\tilde t_n}{M_n(\tilde t_n)}]\times \mathbb R^d$. One has the following asymptotics.

\begin{lemma} \la{lem:typeI renorm}

Assume $0\leq \tilde t_n\leq t_n$ and $\tilde t_n\rightarrow 1$. Then 
\be \la{eq:lim Mn}
\para u_n(\tilde t_n)\para_{L^{\infty}}\sim \frac{\kappa}{(T_n-\tilde t_n)^{\frac{1}{p-1}}}, \ \ \text{i.e.} \ M_n(\tilde t_n)\sim (T_n-\tilde t_n).
\ee
Moreover, up to a subsequence\footnote{With the convention that if the limit in the denominator is $0$ the limit function is $0$.}:
\be \la{eq:cv vn}
v_n \rightarrow \frac{\kappa}{\left[\left( \text{lim}  \ \frac{u_n(\tilde t_n,0)}{\para u_n(\tilde t_n)\para_{L^{\infty}}}\right)^{1-p}-t\right]^{\frac{1}{p-1}}}  \ \ \text{in} \ C^{1,2}_{\text{loc}}(]-\infty,1)\times \mathbb R^d).
\ee

\end{lemma}

\begin{proof}[Proof of Lemma \ref{lem:typeI renorm}]

\noindent \textbf{step 1} Upper bound for $M_n(\tilde t_n)$. We claim that one always has $\para u_n(\tilde t_n) \para_{L^{\infty}}\geq \frac{\kappa}{(T_n-\tilde t_n)^{\frac{1}{p-1}}}$, i.e. 
\be \la{eq:lim Mn leq}
M_n(\tilde t_n)\leq (T_n-\tilde t_n).
\ee
Indeed if it is false then there exists $\delta>0$ such that $\para u_n(\tilde t_n)\para_{L^{\infty}}< \frac{\kappa}{(T_n+\delta-\tilde t_n)^{\frac{1}{p-1}}}$. Therefore, from a parabolic comparison argument this inequality propagates for the solutions, yielding that $-\frac{\kappa}{(T_n+\delta-t)^{\frac{1}{p-1}}} \leq u_n\leq \frac{\kappa}{(T_n+\delta-t)^{\frac{1}{p-1}}}$ for all times $t\geq \tilde t_n$. This implies that $u_n$ stays bounded up to $T_n$, which is a contradiction. \\

\noindent \textbf{step 2} Proof of \fref{eq:cv vn}. Let $(x_n)_{n\in \mathbb N}\in (\mathbb R^d)^{\mathbb N}$ and define:
\be \la{eq:def vn2}
\tilde v_n(\tau,y):=M_n^{\frac{1}{p-1}}(\tilde t_n)u_n\left(x_n+M_n^{\frac 1 2}(\tilde t_n)y,\tilde t_n+\tau M_n(\tilde t_n)\right)
\ee
From \fref{eq:def vn}, $ \tilde v_n$ is defined on $[-\frac{\tilde t_n}{M_n(\tilde t_n)},\frac{T_n-\tilde t_n}{M_n(\tilde t_n)}]\times \mathbb R^d$. The lower bound, $-\frac{\tilde t_n}{M_n(\tilde t_n)}$, then goes to $-\infty $ from \fref{eq:lim Mn leq}. $ \tilde v_n$ is a solution of \fref{eq:NLH} satisfying:
\be \la{eq:bd Linfty vn}
\para  \tilde v_n(0)\para_{L^{\infty}}\leq \kappa, 
\ee
\be \la{eq:bd ODE vn 1}
\forall (\tau,y) \in[-\frac{\tilde t_n}{M_n(\tilde t_n)},0]\times \mathbb R^d,†\ \ \  |\Delta \tilde v_n |\leq \frac{1}{2}| \tilde v_n|^p +2KM_n^{\frac{p}{p-1}}(\tilde t_n) ,
\ee
from \fref{eq:bd ODE un 1} and \fref{eq:def vn}.

\noindent \emph{Precompactness of the renormalized functions}. We claim that $\tilde v_n$ is uniformly bounded in $C^{\frac 3 2,3}_{\text{loc}}(]-\infty,1)\times \mathbb R^d)$. We now prove this result. First, we claim that 
\be \la{eq:bd Linfty vn 2}
| \tilde v_n|\leq \text{max}((4K)^{\frac 1 p}M_n^{\frac{1}{p-1}}(\tilde t_n),\kappa)
\ee
Indeed, as $\tilde v_n$ is a solution of \fref{eq:NLH} satisfying \fref{eq:bd ODE vn 1} one has that:
$$
\partial_t |\tilde v_n| \geq \frac 1 2 |\tilde v_n|^p-2KM_n^{\frac{p}{p-1}}(\tilde t_n).
$$
So if the bound we claim is violated, then $\para \tilde v_n \para_{L^{\infty}}$ is strictly increasing, greater than $\kappa$, which at time $0$ is a contradiction to \fref{eq:bd Linfty vn}. Moreover, as $\para \tilde v_n(0)\para_{L^{\infty}}\leq \kappa$, from a comparison argument, for $0\leq t<1$, on has that $\para \tilde v_n(0)\para_{L^{\infty}}\leq \kappa(1-t)^{-\frac{1}{p-1}}$. This and the above bound implies that for any $T<1$, $\tilde v_n$ is uniformly bounded, independently of $n$, in $L^{\infty}((-\frac{\tilde t_n}{M_n(\tilde t_n)},T]\times \mathbb R^d)$. Applying Lemma \ref{para:lem:para}, it is uniformly bounded in $C^{\frac 3 2,3}((-\frac{\tilde t_n}{M_n}+1,T)\times \mathbb R^d)$, yielding the desired result.

\noindent \emph{Rigidity at the limit}. From Step 2 and Arzela Ascoli theorem, up to a subsequence, $ v_n$ converges in $C^{1,2}_{\text{loc}}((-\infty,0]\times \mathbb R^d)$ to a function $v$. The equation \fref{eq:NLH} passes to the limit and $v$ also solves \fref{eq:NLH}. \fref{eq:bd Linfty vn 2} and \fref{eq:lim Mn leq} imply that $|v|\leq \kappa$. \fref{eq:NLH}, \fref{eq:lim Mn leq} and \fref{eq:bd ODE vn 1} imply that:
$$
\partial_t |v|\geq \frac{1}{2}|v|^p.
$$
Reintegrating this differential inequality one obtains that $|v|\leq \frac{C}{|c-\tau|^{\frac{1}{p-1}}}$ for some $C,c>0$. Applying the Liouville Lemma \ref{lem:liouville autosimilaire}, one has that $v$ is constant in space. Up to a subsequence, $v(0,x_n)=\kappa \ \text{lim} \ \frac{u_n(\tilde t_n,x_n)}{\para u_n(\tilde t_n)\para_{L^{\infty}}}$. Taking $x_n=0$, $\tilde v_n=v_n$ defined by \fref{eq:def vn} and $v$ is then given by \fref{eq:cv vn}, ending the proof of this identity.\\

\noindent \textbf{step 3} Lower bound on $M_n$. We claim that $\text{lim} \ \text{inf} \ \frac{M_n}{T_n-\tilde t_n}\geq 1$. We prove it by contradiction. From  \fref{eq:def Mn}, and up to a subsequence, assume that there exists $0<\delta \ll 1$ and $x_n\in \mathbb R^d$ such that $u_n(\tilde t_n,x_n)>\frac{(1+\delta)\kappa}{(T_n-\tilde t_n)^{\frac{1}{p-1}}}$ and $\frac{u_n(\tilde t_n,x_n)}{\para u_n(\tilde t_n)\para_{L^{\infty}}}\rightarrow 1$. Therefore the renormalized function $\tilde v_n$ defined by \fref{eq:def vn2} blows up at $\frac{T_n-\tilde t_n}{M_n(\tilde t_n)}\geq (1+\delta)^{p-1}$. From Step 2 $v(0,\cdot)$ is uniformly bounded and converges to $\kappa$. Hence, defining the self similar renormalization near $((1+\delta)^{p-1},0)$, 
$$
w_{0,(1+\delta)^{p-1}}^{(n)}(t,y)=((1+\delta)^{p-1}-t)^{\frac{1}{p-1}}\tilde v_n(t,\sqrt{(1+\delta)^{p-1}-t}y),
$$
one has that $I(w_{0,(1+\delta)^{p-1}}(0,\cdot))\rightarrow I((1+\delta)^{p-1}\kappa)>0$ where $I$ is defined by \fref{eq:def I}. From \fref{eq:blow up criterion} for $n$ large enough $\tilde v_n$ should have blown up before $(1+\delta)^{p-1}$ which yields the desired contradiction.

\end{proof}

To end the proof of Proposition \ref{pr:type I}, we now distinguish two cases for which one has to find a contradiction (which cover all possible cases up to subsequence):
\be \la{eq:case1}
\text{Case 1:} \ \ \ \text{lim}\ \frac{|u_n(x_n,t_n)|}{\parallel u_n(t_n) \parallel_{L^{\infty}}}>0, 
\ee
\be \la{eq:case2}
\text{Case 2:} \ \ \  \text{lim} \  \frac{|u_n(x_n,t_n)|}{\parallel u_n(t_n) \parallel_{L^{\infty}}}= 0
\ee

\begin{proof}[Proof of Proposition \ref{pr:type I} in Case 1]

In this case we can renormalize at time $t_n$. Let $\tilde t_n=t_n$ and define $v_n$ and $M_n(\tilde t_n)$ by \fref{eq:def vn} and \fref{eq:def Mn}. \fref{eq:cv vn} and \fref{eq:case1} imply that $\Delta v_n(0,0)\rightarrow 0$ and $ v_n(0,0)\rightarrow v(0,0)>0$. From \fref{eq:bd ODE un 2} $v_n$ satisfies at the origin:
$$
|\Delta v_n (0,0)|= \frac{1}{2}|v_n(0,0)|^p +2KM_n^{\frac{p}{p-1}}(t_n).
$$
As $M_n(t_n)\rightarrow 0$ from \fref{eq:lim Mn}, at the limit we get $0=\frac 1 2 v(0,0)>0$ which is a contradiction. This ends the proof of Proposition \ref{pr:type I} in Case 1.

\end{proof}

\begin{proof}[Proof of Proposition \ref{pr:type I} in Case 2]

\noindent \textbf{step 1} Suitable renormalization before $t_n$. We claim that for any $0<\kappa_0 \ll 1$ one can find a sequence of times $\tilde t_n$ such that $0\leq \tilde t_n\leq t_n$, $\tilde t_n\rightarrow 1$ and such that $v_n$ defined by \fref{eq:def vn} satisfy up to a subsequence:
\be \la{eq:cv vn 2}
v_n \rightarrow \frac{\kappa}{\left[ \left(\frac{\kappa}{\kappa_0}\right)^{p-1}-1-t\right]^{\frac{1}{p-1}}}  \ \text{in} \ C^{1,2}_{\text{loc}}(]-\infty,1)\times \mathbb R^d).
\ee
We now prove this fact. On one hand, $\frac{|u(t,0)|}{\parallel u (t)\para_{L^{\infty}}}\rightarrow 1$ as $t\rightarrow 1$ (from \fref{eq:lim u} and \fref{eq:prop typeI} as $u$ blow up with type I at $0$) and for any $0\leq T<$1 $u_n$ converges to $u$ in $\mathcal C([0,T],L^{\infty}(\mathbb R^d))$ from \fref{eq:lim Tn}. As $t_n \rightarrow 1$, using a diagonal argument, up to a subsequence there exists a sequence of times $0\leq t'_n\leq t_n$ such that $\frac{|u_n(t'_n,0)|}{\parallel u (t'_n)\para_{L^{\infty}}}\rightarrow 1$. On the other hand, from the assumption \fref{eq:case2} and \fref{eq:lim tn}, $\text{lim} \  \frac{|u_n(t_n,0)|}{\parallel u_n(t_n) \parallel_{L^{\infty}}}= 0$ and $t_n\rightarrow 1$. From a continuity argument, for $\kappa_0$ small enough, there exists a sequence $t'_n\leq \tilde t_n\leq t_n$ such that $\text{lim} \  \frac{|u_n(\tilde t_n,0)|}{\parallel u_n(\tilde t_n) \parallel_{L^{\infty}}}= \frac{1}{\left[\left(\frac{\kappa}{\kappa_0}\right)^{p-1}-1 \right]^{\frac{1}{p-1}}}$. From Lemma \fref{lem:typeI renorm} one obtains the desired result \fref{eq:cv vn 2}.\\

\noindent \textbf{step 2} Boundedness via smallness of the energy. Take $\tilde t_n$ and $v_n$ as in Step 1. From \fref{eq:def vn} and \fref{eq:lim Mn} $v_n$  blows up at time $\tau_n=\frac{T_n-\tilde t_n}{M_n(\tilde t_n)}\rightarrow 1$. Up to time $\tau_n'=\frac{T_n-t_n}{M_n(\tilde t_n)}$, $0\leq \tau_n'\leq 0$, $v_n$ satisfies:
\be \la{eq:bd Delta vn 3}
|\Delta v_n |\leq \frac{1}{2}|v_n|^p +2KM_n^{\frac{p}{p-1}}(\tilde t_n) 
\ee
and we recall that $M_n(\tilde t_n)\rightarrow 0$ from \fref{eq:lim Mn}. Let $R>0$ and $a\in B(0,R)$. Define
$$
w_{a,\tau_n}^{(n)}(y,t):= (\tau_n-t)^{\frac{1}{p-1}}v_n(t,a+\sqrt{\tau_n-t}y).
$$
Then as $v_n(-1)\rightarrow \kappa_0$ from \fref{eq:cv vn 2}, one has that for $n$ large enough
$$
E[w_{a,\tau_n}^{(n)}(-1,\cdot)]= O(\kappa_0^2)
$$
where the energy is defined by \fref{eq:def E}. One can then apply the result \fref{eq:bd W2inftyloc u} of Proposition \ref{pr:energy to Linfty}: there exists $r>0$ such that for $\kappa_0$ small enough and $n$ large enough one has:
\be \la{eq:bd W2loc vn}
\forall t\in [0,\tau_n'], \ \ \para v_n(t)\para_{W^{2,\infty}(B(0,r))}\leq C .
\ee

\noindent \textbf{step 3} End of the proof. Let $\chi$ be a cut-off function, $\chi =1$ on $B(0,\frac{R}{16})$ and $\chi=0$ outside $B(0,\frac R 8)$. The evolution of $\tilde v_n=\chi v_n$ is given by:
$$
\tilde v_{n,\tau}-\Delta \tilde v_n = \chi |v_n|^{p-1}v_n +\Delta \chi v_n-2\nabla . \left( \nabla \chi v_n\right) =F_n
$$
with $\para F_n\para_{W^{1,\infty}}\leq C$ from \fref{eq:bd W2loc vn}. Fix $0<s\ll 1$. One has:
$$
\begin{array}{r c l}
\Delta v_n(\tau'_n,0) & = & K_{s}*(\Delta \tilde v_n(\tau_n'-s))(0)+\sum_1^d \int_0^s \left[\partial_{x_i}K_{s-s'}*\partial_{x_i}F(\tau_n'-s+s')\right](0) \\
&=& o_{n\rightarrow +\infty}(1)+o_{s\rightarrow 0}(1)
\end{array}
$$
from \fref{eq:cv vn 2} and the estimate on $F_n$. Hence $\Delta v_n(\tau'_n,0) \rightarrow 0$ as $n\rightarrow +\infty$. On the other hand, $\text{lim} v_n(\tau_n',0)=v(\tau_n',0)>0 $ from \fref{eq:cv vn 2} and the fact that $0\leq \tau_n'\leq 1$. We recall that at time $\tau_n'$ $v_n$ satisfies:
$$
|\Delta v_n (\tau_n',0)|= \frac{1}{2}|v_n(\tau_n',0)|^p +2KM_n^{\frac{p}{p-1}}(\tilde t_n).
$$
As $M_n^{\frac{p}{p-1}}(\tilde t_n)\rightarrow 0$ from \fref{eq:lim Mn} at the limit one has $0=\frac{1}{2}|v(\tau_n',0)|^p>0 $ which is a contradiction. This ends the proof of Proposition \ref{pr:type I} in Case 2.

\end{proof}


\appendix

\section{Kernel of the linearized operator $-\Delta-pQ^{p-1}$}  \label{sec:H}

In this section we prove study the linearized operator $-\Delta-pQ^{p-1}$. We characterize its kernel on each spherical harmonics in the following lemma. This will be useful in the next section to derive suitable coercivity properties for this operator. The numbers
$$
k(0):=1, \ k(1):=d, \ k(n):=\frac{2n+p-2}{n} \begin{pmatrix} n+p-3\\ n-1 \end{pmatrix} \ \text{for} \ n\geq 2
$$
are the number of spherical harmonics of degree $n$. We recall that spherical harmonics are the eigenfunctions of the Laplace-Beltrami operator on the sphere $\mathbb{S}^{d-1}$. The spectrum of this self-adjoint operator with compact resolvent is $\left\{ n(d+n-2), \ n\in \mathbb N\right\}$. For each $n\in \mathbb{N}$ the eigenvalue $n(d+2-n)$ has geometric multiplicity $k(n)$. We then denote the associated orthonormal family of eigenfunctions by $(Y^{(n,k)})_{n\in \mathbb N, \ 1\leq k \leq k(n)}$:
$$
L^2(\mathbb{S}^{d-1})= \underset{n=0}{\overset{+\infty}{\oplus}}^{\perp} \text{Span}(Y^{(n,k)}, \ 1\leq k \leq k(n),
$$
\be \label{intro:eq:def Ynk}
\Delta_{\mathbb S^{d-1}}Y^{(n,k)}=n(d+n-2)Y^{(n,k)}, \ \int_{S^{d-1}(1)} Y^{(n,k)}Y^{(n',k')}=\delta_{(n,k),(n',k')},
\ee
The potential part in the Schr\"odinger operator $H$ being radially symmetric, one can associate to $H$ the following family of linear second order operators on radial functions:
\be \label{intro:eq:def Hn}
(H^{(n)})_{n\in \mathbb N}:=\left(-\partial_{rr}-\frac{d-1}{r}\partial_r+\frac{n(d+n-2)}{r^2}-pQ^{p-1}\right)_{n\in \mathbb N}
\ee
so that one has the following identity for smooth enough radial functions $f$:
$$
H \left(f(|x|)Y^{(n,k)}\left( \frac{x}{|x|}\right)\right)=(H^{(n)}(f))(|x|)Y^{(n,k)}\left( \frac{x}{|x|}\right).
$$

\begin{lemma}[Zeros of $-\Delta-pQ^{p-1}$ on spherical harmonics] \label{lem:zeros}

Let $n\in \mathbb N$ and $f\in \mathcal C^2((0,+\infty),\mathbb R)$ satisfy $H^{(n)} f=0$. Then $f\in \text{Span}(T^{(n)},\Gamma^{(n)})$ where $T^{(n)}$ and $\Gamma^{(n)}$ are smooth and satisfy:
\begin{itemize}
\item[(i)] \emph{Radial case:} $T^{(0)}=\Lambda Q$ and $\Gamma^{(0)}(r)\sim r^{-d+2}$ as $r\rightarrow 0$.
\item[(ii)] \emph{First spherical harmonics:} $T^{(1)}=-\partial_r Q$ and $\Gamma^{(1)}(r)\sim r^{-d+1}$ as $r\rightarrow 0$.
\item[(iii)] \emph{Higher spherical harmonics:} for $n\geq 2$, $T^{(n)}>0$, $T^{(n)}\sim r^n$ as $r\rightarrow +\infty$ and $\Gamma^{(n)}(r)\sim r^{-d+2-n}$ as $r\rightarrow 0$.
\end{itemize}

\end{lemma}

\begin{remark}

This lemma states that on the spherical harmonics of degree $0$ and $1$ there exists a zero of $-\Delta-pQ^{p-1}$ that is not singular at the origin and decays at infinity, whereas for higher degrees all zeros are singular at the origin or do not decay at infinity.

\end{remark}

\begin{proof}[Proof of Lemma \ref{lem:zeros}] 

Let $n\in \mathbb N$ and $f$ satisfy $H^{(n)}f=0$. First we rewrite the equation as an almost constant coefficient ODE. Setting $w(t)=f(e^t)$, $f$ solves $H^{(n)} f=0$ if and only if $w$ solves:
\be \label{eq:ODEn}
w''+(d-2)w'-\left[e^{2t}V(e^{t})+n(d+n-2)\right]w=0.
\ee
First, as $|V(r)|\lesssim (1+|r|^4)^{-1}$, one gets that $|e^{2t}V(e^{t}|\lesssim e^{-2|t|}$. This implies that asymptotically, as $t\rightarrow \pm \infty$, \fref{eq:ODEn} is almost the constant coefficients ODE
\be \label{eq:ODE asymptotique}
w''+(d-2)w'-n(d+n-2)w.
\ee
One also has the bound 
\be \label{eq:bound e2tV}
\forall t\in \mathbb R, |e^{2t}V(e^t)|=r^2|V(r)|\leq (\sqrt{d(d-2)})^2V(\sqrt{d(d-2)})=\frac{d(d+2)}{4}
\ee
which is obtained by maximizing this function.\\

\noindent{\bf step 1} Existence of a solution behaving like $e^{nt}$ as $t\rightarrow -\infty$. We claim that for any $n\in \mathbb N$, there exists a solution $a^{(n)}$ of \fref{eq:ODEn} such that $a^{(n)}(t)= e^{n t}+v(t)$ with $|v(t)|+|v'(t)|\lesssim e^{(n+1)t}$ as $t\rightarrow -\infty$. We now prove this fact. For $n=0$, the function $(\Lambda Q (0))^{-1}\Lambda Q (e^{t})$ satisfies the desired property. We now assume $n\geq 1$. We use a standard fixed point argument to construct this solution as a perturbation of $t\mapsto e^{n t}$ which solves the asymptotic ODE \fref{eq:ODE asymptotique}. Using Duhamel formula, $a^{(n)}$ is a solution of \fref{eq:ODEn} if and only if $v$ is a solution of:
$$
v(t)=\frac{1}{2n+d-2}\int_{-\infty}^t \left(e^{n(t-t')}-e^{-(d+n-2)(t-t')} \right)e^{2t'}V(t')\left[e^{nt'}+v(t')\right]dt'.
$$
For $t_0\in \mathbb R$ we define the following functional space:
$$
X_{t_0}:=\left\{v\in \mathcal C((-\infty,t_0],\mathbb R), \ \ \underset{t\leq t_0}{\text{sup}} \ |v(t)|e^{-(n+1)t}<+\infty \right\}
$$
on which we define the following canonical weighted $L^{\infty}$ norm:
$$
\parallel v \parallel_{X_{t_0}} := \underset{t\leq t_0}{\text{sup}} \ |v(t)|e^{-(n+1)t}.
$$
$(X_{t_0},\parallel \cdot \parallel_{X_{t_0}})$ is a Banach space. We define the following function $\Phi$ on $X_{t_0}$:
$$
(\Phi(v))(t):= \frac{1}{2n+d-2}\int_{-\infty}^t \left(e^{n(t-t')}-e^{-(d+n-2)(t-t')} \right)e^{2t'}V(t')(e^{nt'}+v(t'))dt'.
$$
As $| e^{2t'}V(t')|\lesssim e^{2t'}$, for $t_0 \ll 0$ small enough, estimating by brute force, for $v\in X_{t_0}$, $t_0\ll 0$ and $t\leq t_0$:
$$
\begin{array}{r c l}
|(\Phi(v))(t)| &\lesssim & e^{nt}\int_{-\infty}^{t} e^{(2-n)t'}|V(t')|(e^{nt'}+|v(t')|)dt' \\
&&+e^{-(d+n-2)t}\int_{-\infty}^{t} e^{(d+n)t'}|V(t')|(e^{nt'}+|v(t')|)dt' \\
&\lesssim & e^{nt}\int_{-\infty}^{t} e^{(2-n)t'}(e^{nt'}+e^{(n+1)t'}\parallel v \parallel_{X_{t_0}})dt' \\
&&+e^{-(d+n-2)t}\int_{-\infty}^{t} e^{(d+n)t'}(e^{nt'}+e^{(n+1)t'}\parallel v \parallel_{X_{t_0}})dt' \\
&\lesssim & e^{(n+2)t}+e^{(n+3)t}\parallel v \parallel_{X_{t_0}}
\end{array}
$$
so that $\parallel \Phi (v)\parallel_{X_{t_0}} \lesssim e^{t_0}+e^{2t_0}\parallel v \parallel_{X_{t_0}}$. This implies that for $t_0$ small enough, $\Phi$ maps $B_{X_{t_0}}(0,1)$, the unit ball of $X_{t_0}$, into itself. Now, as $\Phi$ is an affine function one computes similarly for $t_0\ll 1$, $v_1,v_2\in X_{t_0}$ and $t\leq t_0$:
$$
\begin{array}{r c l}
&|(\Phi (v_1)-\Phi (v_2))(t)| \\
\lesssim & e^{nt}\int_{-\infty}^{t} e^{(2-n)t'}|V(t')||v_1- v_2|dt' +e^{-(d+n-2)t}\int_{-\infty}^{t} e^{(d+n)t'}|V(t')||v_1-v_2|)dt' \\
\lesssim & e^{((2+n)t}\parallel v_1-v_2\parallel_{X_{t_0}},
\end{array}
$$
implying that 
$\parallel \Phi (v_1)-\Phi (v_2)\parallel_{X_{t_0}} \lesssim e^{2t_0}\parallel v_1-v_2 \parallel_{X_{t_0}}$,
meaning that $\Phi$ is a contraction on $B_{X_{t_0}}(0,1)$. By Banach fixed point theorem, one gets that there exists a unique fixed point $v_0$ of $\Phi$. The function $a^{(n)}(t)=e^{n t}+v_0(t)$ is then a solution of \fref{eq:ODEn} on $(-\infty,t_0)$. There exists a unique global solution of \fref{eq:ODEn} that coincides with it on $(-\infty,t_0]$ that we still denote by $a^{(n)}$. As $v$ is a fixed point of $\Phi $ using verbatim the same type of computations we just did one sees that $|v'(t)|\lesssim e^{(n+1)t}$ as $t\rightarrow -\infty$. Hence $a^{(n)}$ has the properties we claimed in this step.\\

\noindent{\bf step 2} Existence of a solution behaving like $e^{-(d+n-2)t}$ as $t\rightarrow +\infty$. We claim that for any $n\in \mathbb N$, there exists a solution $b^{(n)}$ of \fref{eq:ODEn} such that $b^{(n)}(t)= e^{-(d+n-2) t}+v(t)$ with $|v(t)|+|v'(t)|\lesssim e^{-(d+n-1)t}$ as $t\rightarrow +\infty$. We now prove this fact. We reverse time and let $\tilde t=-t$, $\tilde w(\tilde t)= w( t )$. Then, $w$ solves \fref{eq:ODEn} if and only if $\tilde w$ solves
$$
\tilde w''-(d-2)\tilde w'-\left[e^{-2\tilde t}V(e^{\tilde t})+n(d+n-2)\right]\tilde w=0
$$
and $w=e^{-(d+n-2)\cdot}+v$ with $|v(t)|+|v'(t)|\lesssim e^{-(d+n-1)t}$ as $t\rightarrow +\infty$ if and only if $\tilde w=e^{(d+n-2)}+\tilde v$ with $|\tilde v(\tilde t)|+|\tilde v'(\tilde t)|\lesssim e^{(d+n-1)\tilde t}$ as $\tilde t\rightarrow -\infty$. One notices that the eigenvalues of the constant coefficients part of the ODE, \fref{eq:ODE asymptotique}, are $-n$ and $d+n-2$, and that $|e^{- 2 \tilde t}V(e^{\tilde t})|\lesssim e^{2\tilde t}$ as $\tilde t \rightarrow -\infty$. Therefore, we are again in the context of a constant coefficient second order ODE with a strictly positive and a strictly negative eigenvalue, plus a exponentially small perturbative linear term. One obtains the result claimed in this step by verbatim the same techniques we just employed in Step 1.\\

\noindent{\bf step 3} Existence of a solution behaving like $e^{nt}$ as $t\rightarrow +\infty$. We claim that for any $n\geq 1$, there exists a solution $c^{(n)}$ of \fref{eq:ODEn} such that $c^{(n)}(t)= e^{n t}+v(t)$ with $|v(t)|\leq \frac{e^{nt}}{2}$ as $t\rightarrow +\infty$. We now prove this fact. Let $t_0\in \mathbb R$ and $w$ be the solution of \fref{eq:ODEn} with initial condition $w(t_0)=e^{nt_0}$ and $w'(t_0)=ne^{nt_0}$. Then we claim that there exists $0\ll t_0$ large enough, such that $|w-e^{nt}|\leq \frac{e^{nt}}{2}$ for all $t_0\leq t$. To prove it, we use a connectedness argument. We define $\mathcal T\subset [t_0,+\infty)$ as the set of times $t\geq t_0$ such that this inequality holds on $[t_0,t]$. $\mathcal T$ is non empty as it contains $t_0$. It is closed by continuity. Then using Duhamel formula, one has for any $t\in \mathcal T$:
$$
\begin{array}{r c l}
|w(t)-e^{nt}| &\lesssim & \int_{t_0}^t \left| \left(e^{n(t-t')}-e^{-(d+n-2)(t-t')} \right)e^{2t'}V(t')(w(t')) \right|dt' \\
&\lesssim & e^{nt} e^{-2t_0}.
\end{array}
$$
This implies that for $t_0$ large enough, $\mathcal T$ is open. By connectedness, one has $\mathcal T=[t_0,+\infty)$, which means that $|w(t)-e^{nt}|\leq \frac{e^{nt}}{2}$ for all times $t\geq t_0$. We extend $w$ backward in times to obtain a global solution of \fref{eq:ODEn}, it then has the properties we claimed in this third step.\\

\noindent{\bf step 4} Existence of a solution behaving like $e^{-(d+n-2)t}$ as $t\rightarrow -\infty$. We claim that for any $n\in \mathbb N$, there exists a solution $d^{(n)}$ of \fref{eq:ODEn} such that $d^{(n)}(t)= e^{-(d+n-2) t}+v(t)$ with $|v(t)|\leq \frac{e^{-(d+n-2)t}}{2}$ as $t\rightarrow +\infty$. This can be proved using verbatim the same techniques we already employed: first by reversing time as in Step 2, then by performing a bootstrap argument as in Step 3.\\

\noindent{\bf step 5} $a^{(n)}\neq b^{(n)}$ for $n\geq 2$. Let $n\geq 2$, $a^{(n)}$ and $b^{(n)}$ be the two solutions we defined in Step 1 and Step 2 respectively. Then we claim that $a^{(n)}\neq b^{(n)}$. To show this, we see \fref{eq:ODEn} as a planar dynamical system. We associate to $w\in \mathcal C^{2}(]-\infty,+\infty),\mathbb R)$ the vector $W:=\begin{pmatrix} w \\ w' \end{pmatrix}$. Then $w$ solves \fref{eq:ODEn} if an only if $W$ solves:
$$
W'=\begin{pmatrix} 0 & 1 \\ e^{2t}V(e^t)+n(d+n-2) & -(d-2)  \end{pmatrix} W.
$$
We denote by $A^{(n)}$ and $B^{(n)}$ the vectors associated to $a^{(n)}$ and $b^{(n)}$. From the results of Step 1 and Step 2, one has that:
$$
A^{(n)}= e^{nt} \begin{pmatrix} 1 \\ n \end{pmatrix}+O(e^{(n+1)t}) \ \ \text{as} \ \ t\rightarrow -\infty,
$$
$$
B^{(n)}= e^{-(d+n-2)t} \begin{pmatrix} 1 \\ -(d+n-2) \end{pmatrix}+O(e^{-(d+n-1)t}) \ \ \text{as} \ \ t\rightarrow +\infty.
$$
We claim that solutions starting in $Z:=\{(x,y)\in \mathbb R^2, \ x\geq 0 \ \text{and} \ y\geq -\frac{d-2}{2}x \}$ cannot escape this zone. Once we have proven this, the result we claim follows as from their asymptotic behaviors, $A^{(n)}$ is in $Z$ for small enough times and $B^{(n)}$ is not in $Z$ for large times. To prove this, one computes the direction of the flow at the boundary $\partial Z:=Z_1\cup Z_2\cup \{ (0,0)\}$, where $Z_1:=\{(0,y), \ y>0\}$ and $Z_2:=\{(x,-\frac{d-2}{2}x), \ x>0\}$. On $Z_1$, one takes $(1,0)$ as the unit vector orthogonal to $Z_1$ pointing inside $Z$. At a time $t\in \mathbb R$, one computes the entering flux at the point $(0,1)$:
$$
\begin{pmatrix} 1 \\ 0\end{pmatrix}.\left( \begin{pmatrix} 0 & 1 \\ e^{2t}V(e^t)+n(d+n-2) & -(d-2)  \end{pmatrix}\begin{pmatrix} 0 \\ 1\end{pmatrix} \right)=1.
$$
By linearity, the flux entering $Z$ through $Z_1$ is also always strictly positive at each point of $Z_1$ for each time $t\in \mathbb R$. On $Z_2$ one takes $(\frac{d-2}{2},1)$ as an orthogonal vector pointing inside $Z$, and one computes the entering flux at the point $(\frac{d-2}{2})$ using \fref{eq:bound e2tV}:
$$
\begin{array}{r c l}
& \begin{pmatrix} \frac{d-2}{2} \\ 1 \end{pmatrix}.\left( \begin{pmatrix} 0 & 1 \\ e^{2t}V(e^t)+n(d+n-2) & -(d-2)  \end{pmatrix}\begin{pmatrix} 1 \\ -\frac{d-2}{2}\end{pmatrix} \right)\\
=& \frac{(d-2)^2}{4}+e^{2t}V(e^t)+n(d+n-2) \\
\geq &  \frac{(d-2)^2}{4}-\frac{d(d+2)}{4}+2d \\
=& \frac{d}{2}+1.
\end{array}
$$
By linearity, the flux entering $Z$ through $Z_2$ is also always strictly positive at each point of $Z_2$ for each time $t\in \mathbb R$. Consequently we have proven that $Z$ is forward in time stable by the flow. This ends the proof of this step.\\

\noindent{\bf step 6} Conclusion. We collect the results proved for the equivalent ODE \fref{eq:ODEn} in the five previous steps and go back to original variables. We set for $n\in \mathbb N$, $T^{(n)}(r)=a^{(n)}(\text{log}(r))$, $\Gamma^{(n)}(r)=b^{(n)}(\text{log}(r))$ and $\tilde \Gamma^{(n)}(r)=d^{(n)}(\text{log}(r))$, and for $n\geq 1$, $\tilde T^{(n)}(r)=c^{(n)}(\text{log}(r))$. From the previous steps there hold
$$
\forall n \geq 1, \ T^{(n)}\underset{r\rightarrow 0}{\sim} r^n , \ \Gamma^{(n)}\underset{r\rightarrow 0}{\sim} r^{-(d+n-2)}, \ \tilde T^{(n)}\underset{r\rightarrow +\infty}{\sim} r^n , \ \tilde \Gamma^{(n)}\underset{r\rightarrow +\infty}{\sim} r^{-(d+n-2)}.
$$

\noindent \emph{Case $n=0$}. From a direct computation, one has $H^{(0)}\Lambda Q=0$. This comes from the invariance by scale change of the equation. Together with the asymptotic of the other solution $\Gamma^{(0)}$ at the origin, this proves the lemma for the remaining case $n=0$.

\noindent \emph{Case $n=1$}. From a direct computation, one has $H^{(1)}\partial_r Q=0$. This comes from the invariance by translation of the equation. We claim that there exists $a_1,a_2\in \mathbb R$ such that $T^{(1)}=a_1\partial_r Q=a_2\tilde \Gamma^{(1)}$. Indeed, as the equation is a second order linear ODE, there exists $b_1,b_2\in \mathbb R$ such that $\partial_r Q=b_1T^{(1)}+b_2\Gamma^{(1)}$, and as $\Gamma^{(1)}$ is singular at the origin, one has that $b_2=0$. We apply the same reasoning at $+\infty$ to prove that $\partial_r Q$ is collinear with $\tilde \Gamma^{(1)}$. This, together with the asymptotic of the other solution $\Gamma^{(1)}$ at the origin proves the lemma for $n=1$.

\noindent \emph{Case $n\geq 2$}. We proved in Step 5 that $T^{(n)}$ and $\tilde \Gamma^{(n)}$ are not collinear. Hence there exists $c^{(n)},c^{(n)'}$ with $c^{(n)}\neq 0$ such that $T^{(n)}=c^{(n)}\tilde T^n+c^{(n)'} \tilde \Gamma^{(n)}\sim c^{(n)}r^n$ as $r\rightarrow +\infty$. This, together with the asymptotic of the other solution $\Gamma^{(n)}$ at the origin proves the lemma for $n\geq 2$.

\end{proof}


\section{Proof of the coercivity lemma \ref{lem:coercivite}}
\label{sec:coercivite}

This Appendix is devoted to the proof of Lemma \ref{lem:coercivite} which adapts to the non radial setting the related proof in \cite{RaphRod}.\\

Thanks to Lemma \ref{lem:zeros} and Proposition \ref{pr:H}, we can state and prove the following coercivity property for the linearized operator $-\Delta-pQ^{p-1}$ under suitable orthogonality conditions. We keep the notations for the spherical harmonics introduced in Appendix \ref{sec:H}.\\

\begin{proof}[Proof of Proposition \ref{pr:H}] For each $n\geq 1$ we define the following first order operator on radial functions:
\be \label{eq:def An}
A^{(n)}:=-\partial_r+W^{(n)}
\ee
where the potential is $W^{(n)}:=\partial_y (\text{log}T^{(n)})$, and where $T^{(n)}$ is defined in Lemma \ref{lem:zeros}, with the convention $T^{(1)}=-\partial_r Q$. As $H^{(n)}$ for $n\geq 1$ has a positive zero eigenfunction, this implies the following factorization property for smooth enough functions:
$$
\int u^{(n,k)}H^{(n)}u^{(n,k)}r^{d-1}dr=\int |A^{(n)}u^{(n,k)}|^2r^{d-1}dr .
$$
In turn, this gives the following formula for all functions $u\in \dot H^1\cap \dot H^2$:
\bea
 \label{eq:formule H}
\int uHu dx& = & \sum_{n\in \mathbb N, \ 1\leq k \leq k(n)}\int_0^{+\infty} u^{(n,k)}H^{(n)}u^{(n,k)}r^{d-1}dr\\
\nonumber &=&\int_0^{+\infty}u^{(0,1)}H^{(0)}u^{(0,1)}r^{d-1}dr+ \sum_{1\leq n, \ 1\leq k \leq k(n)}\int_0^{+\infty}  |A^{(n)}u^{(n,k)}|^2r^{d-1}dr.
\eea
The second term in the right hand side is always non negative, and the first term is non negative if $u\perp \mathcal Y$ from the first part of the proof. This ends the proof of Proposition \ref{pr:H}.
\end{proof}

\begin{proof}[Proof of Lemma \ref{lem:coercivite}]

We recall that the first order operarors factorizing $H$ on each spherical harmonics are defined by \fref{eq:def An}. If $u\in \dot H^1(\mathbb R^d)$, or $u\in \dot H^1\cap \dot H^2(\mathbb R^d)$ or $u\in \dot H^1\cap \dot H^3(\mathbb R^d)$, with decomposition into spherical harmonics 
$$
u(x)=\sum_{n \in \mathbb N, \ 1\leq k \leq k(n)}u^{(n,k)}(|x|)Y^{(n,k)}\left( \frac{x}{|x|}\right)
$$
one deduces respectively:
\bee
\int_{\mathbb R^d} |\nabla u|^2-pQ^{p-1}u^2 & = & \sum_{n \geq 1, \ 1\leq k \leq k(n)} \int_0^{+\infty} |A^{(n)}u^{(n,k)}|^2r^{d-1}dr \\
&& + \int_0^{+\infty} (|\partial_r u^{(0,1)}|^2-pQ^{p-1} |u^{(0,1)}|^2 )r^{d-1}dr
\eee
$$
\int_{\mathbb R^d} |Hu|^2  =  \sum_{n \in \mathbb N, \ 1\leq k \leq k(n)} \int_0^{+\infty} |H^{(n)}u^{(n,k)}|^2r^{d-1}dr 
$$
\bea
\label{eq:expression H3}
\int_{\mathbb R^d} |\nabla Hu|^2-pQ^{p-1}|Hu|^2 & = & \sum_{n \geq 1, \ 1\leq k \leq k(n)} \int_0^{+\infty} |A^{(n)}H^{(n)}u^{(n,k)}|^2r^{d-1}dr \\
\nonumber && \int_0^{+\infty} (|\partial_r H^{(0)}u^{(0,1)}|^2-pQ^{p-1} |H^{(0)}u^{(0,1)}|^2 )r^{d-1}dr .
\eea
We first show the estimate \fref{eq:coercivite H3} for which the proof is a bit more delicate than the proof of \fref{eq:coerciviteA} and \fref{eq:coercivite}.\\ 

\noindent{\bf step 1} Subcoercivity. We claim that for any $d\geq 7$ there exists a constant $C=C(d)>0$ such that for any $u\in \dot H^1\cap \dot H^3 (\mathbb R^d)$ there holds:
\be \label{eq:souscoercivite}
\begin{array}{r c l}
\int_{\mathbb R^d} |\nabla Hu|^2-pQ^{p-1}|Hu|^2 & \geq & \frac 1 C\left(\int_{\mathbb R^d} |\nabla^3 u|^2+\int_{\mathbb R^d} \frac{|\nabla^2 u|^2}{|x|^2}+\int_{\mathbb R^d} \frac{|\nabla u|^2}{|x|^4}+\int_{\mathbb R^d} \frac{ u^2}{|x|^6}\right)\\
&&-C\left( \int_{\mathbb R^d} \frac{|\nabla^2 u|^2}{1+|x|^4}+\int_{\mathbb R^d} \frac{|\nabla u|^2}{1+|x|^6}+\int_{\mathbb R^d} \frac{ u^2}{1+|x|^8}\right).
\end{array}
\ee
Indeed, first recall the standard Hardy inequality for $f\in \dot H^s$ for $0\leq s<\frac d 2$:
\be \label{eq:hardy}
\int_{\mathbb R^d} \frac{|f|^2}{|x|^{2s}}\lesssim \parallel f \parallel_{\dot H^s}^2 .
\ee
In particular, as we are in dimension $d\geq 7$, we can apply this inequality for $s=1,2,3$. As $H=-\Delta-pQ^{p-1}$, with a potential decays faster than the Hardy potential, i.e. for all $j\in \mathbb N$, $|\partial^j_r Q^{p-1}|\lesssim (1+|x|)^{-4-j}$, using the above Hardy inequality plus Young and Cauchy-Schwarz inequalities:
\bee
&& \int_{\mathbb R^d} |\nabla Hu|^2-pQ^{p-1}|Hu|^2 \\
& = & \int_{\mathbb R^d} |\nabla \Delta u|^2+2\nabla \Delta u.\nabla (pQ^{p-1}u)+p^2|\nabla (Q^{p-1}u)|^2\\
&&-\int_{\mathbb R^d} pQ^{p-1}|\Delta u|^2-2p^2Q^{2(p-1)}u\Delta u-p^3Q^{3(p-1)}u^2  \\
&\geq & \frac{1}{2} \int_{\mathbb R^d} |\nabla \Delta u|^2 -\int_{\mathbb R^d} p^2|\nabla (Q^{p-1}u)|^2 + pQ^{p-1}|\Delta u|^2 \\
&&-\int_{\mathbb R^d}2p^2Q^{2(p-1)}u\Delta u +p^3Q^{3(p-1)}u^2 \\
&\geq & \frac{1}{2} \int_{\mathbb R^d} |\nabla \Delta u|^2 -C\left( \int_{\mathbb R^d} \frac{|u|^2}{1+|x|^{10}}+\int_{\mathbb R^d} \frac{|\nabla u|^2}{1+|x|^8}+\int_{\mathbb R^d} \frac{|\nabla^2 u|^2}{1+|x|^4} \right) \\
&\geq & c\left( \int_{\mathbb R^d} |\nabla \Delta u|^2+\frac{|u|^2}{|x|^6}+\frac{|\nabla u|^2}{|x|^4}+ \frac{|\nabla^2 u|^2}{|x|^2}\right)  \\
&&-C\left( \int_{\mathbb R^d} \frac{|u|^2}{1+|x|^{10}}+ \frac{|\nabla u|^2}{1+|x|^8}+\frac{|\nabla^2 u|^2}{1+|x|^4} \right) 
\eee
for some constants $C,c>0$. This gives the estimate \fref{eq:souscoercivite} we claimed in this step.\\

\noindent{\bf step 2} Orthogonality for $Hu$. We claim that for any $d\geq 7$, if $u\in \dot H^3(\mathbb R^d)$ is such that $u\perp \mathcal Y$ then
\be \label{eq:orthogonalite Hu}
Hu\in \text{Span} (\partial_{x_1}Q,...,\partial_{x_n}Q,\Lambda Q,\mathcal Y)^{\perp}.
\ee
We now prove this. The linear form 
$$
\Phi:u\mapsto (\langle H u,\partial_{x_1} Q\rangle,...,\langle H u,\partial_{x_d} Q\rangle,\langle Hu,\Lambda Q\rangle,\langle Hu,\mathcal Y \rangle )
$$
is well defined on $\dot H^3(\mathbb R^d)$. To see this we estimate each term via Cauchy-Schwarz inequality and Hardy inequalities \fref{eq:hardy}, using the asymptotic of the solitary wave \fref{eq:def Q} and the fact that $\mathcal Y$ decays exponentially fast:
\bee
&& \sum_1^d |\langle H u,\partial_{x_i} Q\rangle|+|\langle Hu,\Lambda Q\rangle|+|\langle H u,\mathcal Y \rangle | \lesssim  \int_{\mathbb R^d} \frac{|\Delta u|}{1+|x|^{d-2}} + \int_{\mathbb R^d} \frac{|u|}{1+|x|^{d+2}} \\
&\lesssim & \left(\int_{\mathbb R^d} \frac{|\Delta u|^2}{1+|x|^2} \right)^{\frac 1 2}\left(\int_{\mathbb R^d} \frac{1}{1+|x|^{2d-6}} \right)^{\frac 1 2}+\left(\int_{\mathbb R^d} \frac{u^2}{1+|x|^6} \right)^{\frac 1 2}\left(\int_{\mathbb R^d} \frac{1}{1+|x|^{2d-2}} \right)^{\frac 1 2}\\
& \lesssim & \parallel u\parallel_{\dot H^3}.
\eee
This also gives the continuity of $\Phi$ for the natural topology on $\dot H^3(\mathbb R^d)$. For $u$ smooth and compactly supported satisfying the orthogonality condition $u\perp \mathcal Y$ one can perform the following integrations by parts:
$$
\begin{array}{r c l}
\Phi(u) &= &(\langle u, H\partial_{x_1} Q\rangle,...,\langle u,H\partial_{x_d} Q\rangle,\langle u,H\Lambda Q\rangle,\langle u,H\mathcal Y \rangle ) \\
&=& (\langle u, 0\rangle,...,\langle u,0 \rangle,\langle u,0 \rangle,\langle u,-e_0\mathcal Y \rangle )\\
&=&(0,...,0) .
\end{array}
$$
By density of such functions, one then gets the desired orthogonality \fref{eq:orthogonalite Hu} for all functions $u\in\dot H^3(\mathbb R^d)$ satisfying $u\perp \mathcal Y$.\\

\noindent{\bf step 3} Proof of the coercivity estimate. First, from the orthogonality condition \fref{eq:orthogonalite Hu}, the fact that on radial functions $H$ is self adjoint and admits only $\mathcal Y$ as an eigenfunction associated to a negative eigenvalue, and the formula \fref{eq:expression H3}, one obtains the nonegativity of the quantity:
$$
\int_{\mathbb R^d} |\nabla Hu|^2-pQ^{p-1}|Hu|^2 \geq 0.
$$
We now argue by contradiction and assume that the estimate \fref{eq:coercivite H3} does not hold for functions on $u\in \dot H^3(\mathbb R^d)$ satisfying the orthogonality conditions \fref{eq:orthogonalite}. Up to renormalization, this amounts to say that there exists a sequence $(u_n)_{n\in \mathbb N}\in [\dot H^3(\mathbb R^n]^{\mathbb N}$ such that for each $n$, $u_n$ satisfies the orthogonality conditions \fref{eq:orthogonalite}, $H(u_n)$ satisfies \fref{eq:orthogonalite Hu}, 
\be \label{eq:bound un}
\int_{\mathbb R^d} |\nabla^3 u_n|^2+\int_{\mathbb R^d} \frac{|\nabla^2 u_n|^2}{|x|^2}+\int_{\mathbb R^d} \frac{|\nabla u_n|^2}{|x|^4}+\int_{\mathbb R^d} \frac{ u_n^2}{|x|^6}=1
\ee
and:
\be \label{eq:Hun}
\int_{\mathbb R^d} |\nabla Hu_n|^2-pQ^{p-1}|Hu_n|^2 \rightarrow 0 \ \text{as} \ n\rightarrow +\infty.
\ee
From the subcoercivity formula \fref{eq:souscoercivite} from Step 1, the convergence to zero of its left hand side \fref{eq:Hun} and the order 1 size for the first terms of the right hand side \fref{eq:bound un} one deduces that there exists $c>0$ such that for all $n\in \mathbb N$:
\be \label{eq:localisation masse}
\int_{\mathbb R^d} \frac{|\nabla^2 u_n|^2}{1+|x|^4}+\int_{\mathbb R^d} \frac{|\nabla u_n|^2}{1+|x|^6}+\int_{\mathbb R^d} \frac{ u_n^2}{1+|x|^8}>c .
\ee
From weak compactness of $\dot H^3(\mathbb R^d)$ and the compactness of the embedding for localized Sobolev spaces, there exists $u_{\infty}\in \dot H^3$ such that $u_n$ converges toward $u_{\infty}$ weakly in $\dot H^3(\mathbb R^d)$ and strongly in $H^2_{\text{loc}}(\mathbb R^d)$. The above lower bound, together with \fref{eq:bound un} means that the mass of $(u_n)_{n\in \mathbb N}$ cannot go to infinity. Combined with the strong local convergence, this implies:
$$
\int_{\mathbb R^d} \frac{|\nabla^2 u_n|^2}{1+|x|^4}+\frac{|\nabla u_n|^2}{1+|x|^6}+\frac{ u_n^2}{1+|x|^8} \rightarrow \int_{\mathbb R^d} \frac{|\nabla^2 u_{\infty}|^2}{1+|x|^4}+ \frac{|\nabla u_{\infty}|^2}{1+|x|^6}+\frac{ u_{\infty}^2}{1+|x|^8} .
$$
Combined with \fref{eq:localisation masse} one gets the non nullity of the limit: $u_{\infty}\neq 0$. From the weak convergence, $u_{\infty}$ satisfies also the orthogonality conditions \fref{eq:orthogonalite} as $\Psi_0$ is exponentially decaying and $\Psi_1,...,\Psi_d$ are compactly supported. $Hu_{\infty}$ then satisfies the orthogonality condition \fref{eq:orthogonalite Hu}, from the result of Step 2. From \fref{eq:Hun}, Fatou lemma and the formula \fref{eq:expression H3} , one gets:
\bea \label{eq:expression H3uinfty}
\non &&\int_{\mathbb R^d} |\nabla Hu_{\infty}|^2-pQ^{p-1}|Hu_{\infty}|^2 \\
\non &= & \underset{{n \geq 1, \ 1\leq k \leq k(n)}}{\sum} \int_0^{+\infty} |A^{(n)}H^{(n)}u_{\infty}^{(n,k)}|^2r^{d-1}dr \\
&&+ \int_0^{+\infty} (|\partial_r H^{(0)}u_{\infty}^{(0,1)}|^2-pQ^{p-1} |H^{(0)}u_{\infty}^{(0,1)}|^2 )r^{d-1}dr =0.
\eea
We now decompose $u_{\infty}$ on spherical harmonics:
$$
u_{\infty}(x)=\sum_{n \in \mathbb N, \ 1\leq k \leq k(n)}u_{\infty}^{(n,k)}(|x|)Y^{(n,k)}\left( \frac{x}{|x|}\right)
$$
and prove the nullity of each component $u_{\infty}^{(n,k)}$, which will be a contradiction to the fact that we just obtained the non nullity of $u_{\infty}$.\\

\noindent \emph{Case $n\geq 2$}. Let $n\geq 2$ and $1\leq k \leq k(n)$, then we claim that $u_{\infty}^{(n,k)}=0$. From \fref{eq:expression H3uinfty} and the fact that all the terms in the right hand side are nonnegative thanks to the spectral Proposition \ref{pr:H} for $H$ one gets:
$$
\int_0^{+\infty} |A^{(n)}H^{(n)}u_{\infty}^{(n,k)}|^2r^{d-1}dr=0
$$
which implies $A^{(n)}H^{(n)}u_{\infty}^{(n,k)}=0$. From the definition \fref{eq:def An} of $A^{(n)}$ this implies that there exists $c^{(n,k)}\in \mathbb R$ such that:
$$
H^{(n)}u_{\infty}^{(n,k)}=c^{(n,k)} T^{(n)}
$$
where $T^{(n)}$ is defined in Lemma \fref{lem:zeros}. However still from this Lemma one has that $T^{(n)}(r)\sim r^n$ as $r\rightarrow +\infty$. Hence $\int_0^{+\infty} \frac{|T^{(n)}|^2}{1+r^4}r^{d-1}dr=+\infty$. From $u_{\infty}\in \dot H^3(\mathbb R^d)$ and Sobolev embedding, one gets $\int_0^{+\infty} \frac{|H^{(n)}u_{\infty}^{(n,k)}|^2}{1+r^4}r^{d-1}dr<+\infty$. Hence $c^{(n,k)}=0$, which means that $H^{(n)}u_{\infty}=0$. In turn, again from Lemma \ref{lem:zeros}, this means that there exists two constants $c_1^{(n,k)}$ and $c_2^{(n,k)}$ such that:
$$
u_{\infty}^{(n,k)}=c^{(n,k)}_1 T^{(n)}+c^{(n,k)}_2 \Gamma^{(n)}.
$$
As $u_{\infty},T^{(n)}\in L^2_{\text{loc}}(\mathbb R^d)$ and as $\Gamma^{(n)}$ is singular at the origin, with in particular $\Gamma^{(n)}\notin L^2_{\text{loc}}(\mathbb R^d)$ one gets that $c^{(n,k)}_2=0$. As $\int_{\mathbb R^d} \frac{|u_{\infty}^{(n,k)}|^2}{|x|^6}<+\infty$ from Fatou Lemma, and as $\int_{\mathbb R^d} \frac{|T^{(n)}|^2}{|x|^6}=+\infty$ because $T^{(n)}\sim |x|^n$ at infinity from Lemma \ref{lem:zeros}, one obtains that $c^{(n,k)}_1=0$ too. Therefore $u_{\infty}^{(n,k)}=0$ which is the fact we claimed.\\ 

\noindent \emph{Case $n=1$}. Let $1\leq k \leq d$, we claim that $u_{\infty}^{(1,k)}=0$. Similarly, from \fref{eq:expression H3uinfty} and the fact that all the terms in the right hand side are nonnegative thanks to the spectral Proposition \ref{pr:H} for $H$ one gets:
$$
 \int_0^{+\infty} |A^{(1)}H^{(1)}u_{\infty}^{(1,k)}|^2r^{d-1}dr=0
$$
which implies $A^{(1)}H^{(n)}u_{\infty}^{(1,k)}=0$. From its definition \fref{eq:def An} and Lemma \fref{lem:zeros} this implies that there exists $c^{(1,k)}\in \mathbb R$ such that:
$$
H^{(1)}u_{\infty}^{(1,k)}=c^{(1,k)} \partial_r Q.
$$
From the orthogonality conditions \fref{eq:orthogonalite Hu} for $Hu$ one gets $c^{(1,k)}=0$, meaning that $H^{(1)}u_{\infty}^{(1,k)}=0$. From Lemma \ref{lem:zeros}, there exists two constants $c_1^{(1,k)}$ and $c_2^{(1,k)}$ such that:
$$
u_{\infty}^{(n,k)}=c^{(n,k)}_1 \partial_r Q+c^{(n,k)}_2 \Gamma^{(1)}.
$$
As $u_{\infty},\partial_r Q \in L^2_{\text{loc}}(\mathbb R^d)$ and $\Gamma^{(n)}\notin L^2_{\text{loc}}(\mathbb R^d)$ for it is singular at the origin from Lemma \ref{lem:zeros}, the second integration constant is nul: $c^{(1,k)}_2=0$. As $u_{\infty}$ satisfies from \fref{eq:orthogonalite} $\int u_{\infty}\Psi_k=0$, which reads $\int_0^{+\infty} u^{(1,k)}_{\infty}\chi_M\partial_r Q r^{d-1}dr=0$ in spherical harmonics, one gets that the first integration constant is nul: $c^{(1,k)}_1=0$. Hence $u_{\infty}^{(1,k)}=0$ which is the fact we claimed.\\ 

\noindent \emph{Case $n=0$}. From the two previous points one has that $u_{\infty}$ is a radial function and so $Hu_{\infty}$ is also radial. As $Hu_{\infty}$ enjoys the orthogonality conditions $\langle Hu_{\infty},\Lambda Q \rangle=0=\langle Hu_{\infty},\mathcal Y \rangle$ from \fref{eq:orthogonalite Hu}, Proposition \ref{pr:H} for $H$ implies that $Hu_{\infty}=0$. From Lemma \ref{lem:zeros}, this means $u_{\infty}\in \text{Span}(\Lambda Q,\Gamma^{(0)})$. As $u_{\infty}$ and $\Lambda Q$ are square integrable at the origin, whereas $\Gamma^{(0)}$ is not (it is singular from Lemma \ref{lem:zeros}), one gets that $u_{\infty}\in \text{Span}(\Lambda Q)$. The orthogonality conditions $\langle u_{\infty}, \Psi_0\rangle =0$ then imply $u_{\infty}=0$. \\

Consequently, we have proven that $u_{\infty}$ is $0$, which is the desired contradiction. Hence the coercivity property \fref{eq:coercivite H3} is true. The proofs of the first two coercivity properties \fref{eq:coerciviteA} and \fref{eq:coercivite} follow the same line and is left to the reader.
\end{proof}


\section{Adapted decomposition close to the manifold of ground states} 

\label{sec:decomposition}

In this section we give the proof of the decomposition Lemma \ref{lem:decomposition}. Such result is standard in modulation theory, and we give the proof here for the sake of completeness.

\begin{proof}[Proof of Lemma \ref{lem:decomposition}]

We give a classical proof relying on the implicit function theorem.\\

\noindent{\bf step 1} Stationary decomposition. We define the following function:
$$
\begin{array}{l l l l l}
\Phi :& \dot H^1(\mathbb R^d) \times \mathbb R^{d+1}\times (0,+\infty) & \rightarrow & \mathbb R^{d+2} \\
&(u,z,a,\lambda)&\mapsto &(\langle v,\Psi_1 \rangle,...,\langle v,\Psi_d Q \rangle,\langle v,\mathcal Y \rangle,\langle v,\Psi_0 \rangle)
\end{array}
$$
where:
$$
v=(\tau_{-z}u)_{\frac{1}{\lambda}}+(\tau_{-z}Q)_{\frac{1}{\lambda}}-Q-a\mathcal Y .
$$
The function $\Phi$ is well defined, as the orthogonality conditions are taken against functions that are either exponentially decaying or compactly supported, and as $\dot H^1$ is continuously embedded in $L^{\frac{2d}{d-2}}$ from Sobolev inequalities. It is $\mathcal C^{\infty}$. One computes the Jacobian matrix with respect to the last arguments at the point $(0,0,0,1)$:
$$
\begin{array}{r c l}
J\Phi (0,0,0,1)& = & (\frac{\partial \Phi}{\partial z_1},...,\frac{\partial \Phi}{\partial z_d},\frac{\partial \Phi}{\partial a},\frac{\partial \Phi}{\partial \lambda})(0,0,0,1) \\
&=& \begin{pmatrix} \int \frac 1 d \chi_M |\nabla Q|^2 & & & &(0) \\
 & . & & & \\
& & \int \frac 1 d \chi_M |\nabla Q|^2 & & \\
 & & & \int \mathcal Y^2 & \\
(0) & & & & \int \chi_M \Lambda Q^2 
\end{pmatrix}
\end{array}
$$
The implicit function theorem then gives that there exists $\delta >0$ and unique smooth functions $z$, $a$ and $\lambda$ on $B_{\dot H^1}(0,\delta)$ such that for any $u \in B_{\dot H^1}(0,\delta)$ one has $\Phi (u,z,a,\lambda)=0$, meaning that $v\in \text{Span}(\Psi_1,...,\Psi_d,\mathcal Y,\Psi_0)^{\perp}$.\\

\noindent{\bf step 2} Dynamical decomposition. Now, by invariance of the $\dot H^1$ norm by scaling and translation, for any $z\in \mathbb R^d$ and $\lambda >0$ there exists such a decomposition for all balls $B_{\dot H^1}(Q_{z,\lambda},\delta)$, and they coincide when the balls overlap. Therefore, there exist smooth functions that we still denote by $\lambda$, $a$, and $z$, defined on $(Q_{z,\lambda})_{\lambda>0,z\in \mathbb R^d}+B_{\dot H^1}(0,\delta)$, such that $v\in \text{Span}(\Psi_1,...,\Psi_d,\mathcal Y,\Psi_0)^{\perp}$.\\

To finish, assume $u_0\in \dot H^1(\mathbb R^d)$ with $\parallel u_0-Q\parallel_{\dot H^1}< \delta$ and that the solution of \fref{eq:NLH} given by Proposition \ref{pr:cauchy} is defined on some time interval $[0,T)$ and satisfies:
$$
\underset{0\leq t< T}{\text{sup}}\underset{\lambda>0, z\in \mathbb R^d}{\text{inf}}\parallel u-Q_{z,\lambda}\parallel_{\dot H^1}< \delta .
$$
Then, from Proposition \ref{pr:cauchy}, $u\in \mathcal C^1((0,T),\dot H^1)$, hence, the functions $\lambda (u(t))$, $z(u(t))$ and $a(u(t))$ give the desired decomposition and are $C^1$ on $(0,T)$.

\end{proof}


\section{Nonlinear inequalities} \la{sec:NL}

In this section, we state certain estimates on the nonlinear term $f(Q+u)-f'(Q)u-f(Q)$ and on its derivatives.

\begin{lemma}[Pointwise estimates on the purely nonlinear term] \label{lem:NL}

For any $x,y,y_1,y_2\in \mathbb R$ there holds the following estimates:
\be \label{eq:NL1}
\left||1+x|^{p-1}(1+x)-px-1\right|\lesssim |x|^p.
\ee
\be \label{eq:NL2}
\left||1+x|^{p-1}(1+x)-px-1\right|\lesssim |x|^2.
\ee
\be \label{eq:NL3}
\left||1+x+y|^{p-1}(1+x+y)-p(x+y)-1\right|\lesssim |x|^2+|y|^p.
\ee
\be \label{eq:NL4}
\left||1+x+y|^{p-1}-1\right|\lesssim |x|+|y|^{p-1}.
\ee
\be \label{eq:NL5}
\left||1+x+y|^{p-1}-1\right|\lesssim |x|+|y|.
\ee
\be \label{eq:NL6}
\left||1+x+y|^{p-1}-1-(p-1)(x+y)\right|\lesssim |x|^{1+\frac{2}{d-2}}+|y|^2.
\ee
If $x> 0$:
\be \label{eq:NL7}
\left| |x+y_1|^{p-1}(x+y_1)-|x+y_2|^{p-1}(x+y_2)-px^{p-1}(y_1-y_2)\right|\lesssim |y_1-y_2|(|y_1|^{p-1}+|y_2|^{p-1})
\ee
\be \label{eq:NL8}
\left| |x+y|^{p-1}-|x|^{p-1}-(p-1)|x|^{p-2}y\right||x|\lesssim |y|^p
\ee
\be \label{eq:NL9}
\left| |x+y|^{p-1}-|x|^{p-1} \right|\lesssim |y|^{p-1}
\ee
\be \label{eq:NL10}
\left| |1+x|^{p+1}-1-(p+1)x \right|\lesssim |x|^{p+1}+|x|^2
\ee

\end{lemma}

\begin{proof}[Proof of Lemma \ref{lem:NL}]

\noindent - \emph{Proof of \fref{eq:NL1}}. As $g:x\mapsto |1+x|^{p-1}(1+x)$ is $\mathcal C^2$ on $[-\frac 1 2,\frac 1 2]$ at $x=0$, with $g(0)=1$ and $g'(0)=p$ there exists $C>0$ such that:
\be \label{eq:NL1 1}
\forall x\in [-\frac 1 2,\frac 1 2], \ \ \left||1+x|^{p-1}(1+x)-px-1\right|\leq C|x|^2\leq C|x|^p.
\ee
as $1<p<2$. Now as
\be \label{eq:NL asymptotique}
\left||1+x|^{p-1}(1+x)-px-1\right|\sim |x|^p \ \text{as} \ |x|\rightarrow +\infty,
\ee
there exists $C'>0$ such that:
\be \label{eq:NL1 2}
\forall |x|\geq \frac 1 2, \ \ \left||1+x|^{p-1}(1+x)-px-1\right|\leq  C'|x|^p.
\ee
The two estimates \fref{eq:NL1 1} and \fref{eq:NL1 2} then imply \fref{eq:NL1}.\\

\noindent - \emph{Proof of \fref{eq:NL2}}. From \fref{eq:NL asymptotique}, as $1<p<2$, there exists $C'>0$ such that:
$$
\forall |x|\geq \frac 1 2, \ \ \left||1+x|^{p-1}(1+x)-px-1\right|\leq  C'|x|^2.
$$
This, combined with \fref{eq:NL1 1}, implies \fref{eq:NL2}.\\

\noindent - \emph{Proof of \fref{eq:NL3}}. As $g:x\mapsto |1+x|^{p-1}(1+x)$ is $\mathcal C^2$ on $[-\frac 1 2,\frac 1 2]$ at $x=0$, with $g(0)=1$ and $g'(0)=p$ there exists $C>0$ such that:
\be \label{eq:NL3 1}
\forall x,y\in [-\frac 1 2,\frac 1 2], \ \ \left||1+x|^{p-1}(1+x)-px-1\right|\leq C(|x|^2+|y|^2)\leq C(|x|^2+|x|^p).
\ee
as $1<p<2$. Now, for $|x|,|y|\geq\frac 1 4$ one has:
\be \label{eq:NL3 2}
\left||1+x+y|^{p-1}(1+x+y)-p(x+y)-1\right|\leq |1+x+y|^p+1+|x|+|y| \lesssim |x|^p+|y|^p \leq |x|^2+|y|^p
\ee
as $1<p<2$. \fref{eq:NL3 1} and \fref{eq:NL3 2} then imply \fref{eq:NL3}.

\noindent - \emph{Proof of \fref{eq:NL4} and \fref{eq:NL5}}. They can be proved using the same arguments used in the proof of \fref{eq:NL3}.

\noindent - \emph{Proof of \fref{eq:NL6}}. It can be proved using the same arguments of \fref{eq:NL3}, using the fact that:
$$
1<1+\frac{2}{d-2}<p<2.
$$

\noindent - \emph{Proof of \fref{eq:NL7}}. First, by dividing everything by $x^p$ \fref{eq:NL7} is equivalent to:
\be \label{eq:NL7 expression}
\left| |1+y_1|^{p-1}(1+y_1)-|1+y_2|^{p-1}(1+y_2)-p(y_1-y_2)\right|\lesssim |y_1-y_2|(|y_1|^{p-1}+|y_2|^{p-1}).
\ee
If $|y_i|\geq \frac 1 2$ for $i=1,2$ and $|y_1|\geq 2|y_2|$ or $|y_1|\geq 2|y_2|$ (as the estimate is symmetric we assume $|y_1|\geq 2|y_2|$) then one has:
\bee
&&\left| |1+y_1|^{p-1}(1+y_1)-|1+y_2|^{p-1}(1+y_2)-p(y_1-y_2)\right| \\
&=& \left| |1+y_1|^{p-1}(1+y_1)-|1+y_2|^{p-1}(1+y_2)-p((y_1+1)-(y_2+1))\right| \\
&\leq & |1+y_1|^p+|1+y_2|^p+p|1+y_1|+p|1+y_2| \leq  2|y_1|^p+2|y_2|^p+p|y_1|+p|y_2| \\
&\leq &(4+2p)|y_1|^p \leq (8+4p)|y_1-y_2|(|y_1|^{p-1}+|y_2|^{p-1}).
\eee
If $|y_i|\geq \frac 1 2$ for $i=1,2$ and $\frac 1 2 |y_1|\leq |y_2|\leq 2|y_1|$, then using \fref{eq:NL1}:
\bee
&&\left| |1+y_1|^{p-1}(1+y_1)-|1+y_2|^{p-1}(1+y_2)-p(y_1-y_2)\right| \\
&=&\Bigl||1+y_2+(y_1-y_2)|^{p-1}(1+y_2+(y_1-y_2))-|1+y_2|^{p-1}(1+y_2)\\
&&-p|1+y_2|^{p-1}(y_1-y_2)  +p|1+y_2|^{p-1}(y_1-y_2)-p(y_1-y_2)\Bigr| \\
&\leq &\Bigl||1+y_2+(y_1-y_2)|^{p-1}(1+y_2+(y_1-y_2))-|1+y_2|^{p-1}(1+y_2)\\
&&-p|1+y_2|^{p-1}(y_1-y_2)\Bigr| +p|1+y_2|^{p-1}|y_1-y_2|+p|y_1-y_2| \\
&\lesssim &|y_1-y_2|^p+|1+y_2|^{p-1}|y_1-y_2|+|y_1-y_2|\\
&\lesssim &|y_1-y_2|(1+|y_1-y_2|^{p-1}+|1+y_2|^{p-1})\\
&\lesssim &|y_1-y_2|(1+|y_1|^{p-1}+|y_2|^{p-1})\lesssim  |y_1-y_2|(|y_1|^{p-1}+|y_2|^{p-1}).
\eee
If $|y_i|\leq \frac 1 2$ for $i=1,2$, then as $g:x\mapsto |1+x|^{p-1}(1+x)$ is $\mathcal C^2$ on $[-\frac 1 2,\frac 1 2]$ there exists a constant $C>0$ such that:
$$
\begin{array}{r c l}
&\left| |1+y_1|^{p-1}(1+y_1)-|1+y_2|^{p-1}(1+y_2)-p(y_1-y_2)\right| \\
\leq & C(y_1-y_2)^2\leq C|y_1-y_2|(|y_1|+|y_2|) \leq C|y_1-y_2|(|y_1|^{p-1}+|y_2|^{p-1})
\end{array}
$$
as $0<p-1<1$. The three above estimates in established in a partition of $\mathbb R^2$ in three zones, imply \fref{eq:NL7 expression}.

\noindent - \emph{Proof of \fref{eq:NL8}}. This can be proved using the same reasoning we did to prove \fref{eq:NL1}.

\noindent - \emph{Proof of \fref{eq:NL9}}. It is a direct consequence of \fref{eq:NL4}.

\noindent - \emph{Proof of \fref{eq:NL10}}. The function $g(x)=\left| |1+x|^{p+1}-1-(p+1)x \right|$ is smooth near $0$ and one has $g(0)=g'(0)=0$. Hence there exists $C>0$ such that $|g(x)|\leq C|x|^2$ for $|x|\leq \frac 1 2$. As $|g(x)|\sim |x|^{p+1}$ as $x\rightarrow \pm \infty$, there exists a constant $C'$ such that $|g(x)|\leq C' |x|^{p+1}$ for $|x|\geq \frac 1 2$. Therefore, for all $x\in \mathbb R$, $|g(x)|\leq C|x|^2+C'|x|^{p+1}$ and \fref{eq:NL10} is proven.

\end{proof}


\section{Parabolic estimates} \la{sec:para}

This last section is devoted to the results coming from the parabolic properties of the equation. First we state some estimates on the heat kernel $K_t$ defined in \fref{eq:def Kt} in Lemma \ref{para:lem:Kt}. We then recall the maximum principle in Lemma \ref{para:lem:Kt}. Some parabolic estimates that are used several times in the paper are stated in Lemma \ref{para:lem:para} and eventually we state some Strichartz-type inequalities in Lemma \ref{lem:strichartz}. We prove here most of the results for the sake of completeness, and refer to \cite{Gi5} and \cite{La} for more informations about elliptic and parabolic equations, and to \cite{Ba}, Theorem 8.18 for more on Strichartz estimates. \\

\begin{lemma}[Estimates for the heat kernel] \la{para:lem:Kt}

For any $d\in \mathbb N$ and $t>0$ one has:
\be \la{eq:bd Kt}
\forall j\in \mathbb N, \ \forall q\in [1,+\infty], \ \ \para \nabla^j K_t \para_{L^q} \leq \frac{C(d,j)}{t^{\frac{d}{2q'}+\frac j 2}}
\ee
where $q'$ is the Lebesgue conjugated exponent of $q$,
\be \la{eq:bd holder Kt}
\forall y \in \mathbb R^d, \ \ \frac{1}{|y|^{\frac 1 4}}\int |\nabla K_t(x)-\nabla K_t(x+y)|\leq \frac{C(d)}{t^{\frac{5}{8}}},
\ee
and for $t'>t$:
\be \la{eq:bd holder Kt 2}
\frac{1}{|t'-t|^{\frac 1 4}} \int |\nabla K_t(x)-\nabla K_{t'}(x)|\leq \frac{C(d)}{t^{\frac{3}{4}}}.
\ee

\end{lemma}

\begin{proof}[Proof of Lemma \ref{para:lem:Kt}]

\fref{eq:bd Kt} is a standard computation that we do not write here. To prove \fref{eq:bd holder Kt} we change variables and let $\tilde y = \frac{y}{2\sqrt t}$ and $\tilde x=\frac{x}{2\sqrt t}$:
$$
\ba{r c l}
&\frac{1}{|y|^{\frac 1 4}}\int |\nabla K_t(x)-\nabla K_t(x+y)| \\
=& \frac{1}{|y|^{\frac 1 4}} \frac{1}{Ct^{\frac d 2 +1}} \int |xe^{-\frac{|x|^2}{4t}}-(x+y)e^{\frac{-|x+y|^2}{4t}}|dx \\
=& \frac{1}{|\tilde y|^{\frac 1 2}} \frac{1}{Ct^{\frac 5 8}} \int |\tilde xe^{-|\tilde x|^2}-(\tilde x+\tilde y)e^{-|\tilde x+\tilde y|^2}|d\tilde x. \\
\ea
$$
Now, if $\tilde y\geq 1$, then:
$$
\frac{1}{|\tilde y|^{\frac 1 2}}  \int |\tilde xe^{-|\tilde x|^2}-(\tilde x+\tilde y)e^{-|\tilde x+\tilde y|^2}|d\tilde x\leq 2\int |xe^{-|x|^2}|\leq C
$$
and if $y\leq 1$:
$$
\ba{r c l}
&\frac{1}{|\tilde y|^{\frac 1 2}}  \int |\tilde xe^{-|\tilde x|^2}-(\tilde x+\tilde y)e^{-|\tilde x+\tilde y|^2}|d\tilde x \\
\leq & |\tilde y|^{-\frac 1 2} \int |\tilde y|\underset{|x'-\tilde x|\leq 1}{\text{sup}}|\nabla (x'e^{-|x'|^2})|d\tilde x \\
\leq & |\tilde y|^{\frac 1 2} \int 3(|\tilde x|+1)^2e^{-(|\tilde x|-1)^2}d\tilde x \leq C \\
\ea 
$$
The three previous equations then imply \fref{eq:bd holder Kt}. To prove \fref{eq:bd holder Kt 2} we change variables and let $\tilde x = \frac{x}{2\sqrt t}$ and $\tilde t=\frac{t'}{t}$:
$$
\ba{r c l}
&\frac{1}{|t'-t|^{\frac 1 4}} \int |\nabla K_t(x)-\nabla K_{t'}(x)|dx \\
\leq & \frac{1}{t^{\frac 3 4}} \frac{1}{(\tilde t-1)^{\frac 1 4}} \int_{\mathbb R^d} |\tilde xe^{-|\tilde x|^2}-\frac{\tilde x e^{-\frac{|\tilde x|^2}{\tilde t}}}{\tilde t^{\frac d 2+1}}|d\tilde x
\ea
$$
Now, $ \frac{1}{(\tilde t-1)^{\frac 1 4}} \int_{\mathbb R^d} |\tilde xe^{-|\tilde x|^2}|-\frac{\tilde x e^{-\frac{|\tilde x|^2}{\tilde t}}}{\tilde t^{\frac d 2+1}}|d\tilde x\rightarrow 0$ as $\tilde t\rightarrow +\infty$ and $\tilde t\rightarrow 1$ (using Lebesgue differentiation theorem), hence this quantity is bounded, and the above estimate implies \fref{eq:bd holder Kt 2}.

\end{proof}

We now state a comparison principle adapted to our problem.

\begin{lemma}[Comparison principle] \la{para:lem:max}

Let $u$ be a solution of \fref{eq:NLH} given by Proposition \ref{pr:cauchy} on $[0,T)$ with $u_0\in W^{2,\infty}\cap \dot H^1$ and $u_0\geq 0$. Then:
\begin{itemize}
\item[(i)] $u(t)\geq 0$ for all $t\in [0,T)$.
\item[(ii)] If $\partial_t u(0)\geq 0$ (resp. $\partial_t u(0)\leq 0$) then $\partial_t u(t)\geq 0$ (resp. $\partial_t u(t)\leq 0$) for all $t\in [0,T)$.
\item[(iii)] If $ u(0)\leq Q$ (resp. $ u(0)\geq Q$) then $u(t)\leq Q$ (resp. $u(t)\geq Q$) for all $t\in [0,T)$.
\end{itemize}

\end{lemma}

\begin{proof}

We do not do a detailed proof here, just sketch the main arguments. 

\noindent To prove (i), one notices that the semi-group $(K_t*\cdot)_{t\geq 0}$ and the nonlinearity $f(\cdot)$ preserve the cones of positive and negative functions and hence so does the solution mapping of \fref{eq:NLH}. Now, to prove (ii), one notices that $\partial_t u$ solves $\partial_t (\partial_t u)=\Delta \partial_t u+pu^{p-1}\partial_t u$ which is again a parabolic equation with a force term $pu^{p-1}\cdot $ that preserves the cones of positive and negative functions as $u$ is positive from (i). The same argument as in (i) then yields that $\partial_t u $ stays positive (resp. negative) if it is so initially.

To prove (iii), notice that the difference $u(t)-Q$ solves $\partial_t (u-Q)=\Delta (u-Q)+f(Q+(u-Q))-f(Q)$ as $Q$ is a solution of \fref{eq:NLH}. The force term, again, preserves the cones of positive and negative functions: if $u-Q\geq 0$ (resp. $u-Q\leq 0$ then $f(Q+(u-Q))-f(Q)\geq 0$ (resp. $f(Q+(u-Q))-f(Q)\leq 0$) as $f$ is increasing on $\mathbb R$. For the same arguments $u(t)-Q$ then has to stay positive (resp. negative) if it is so initially.

\end{proof}

We now state some estimates that we use several time in the paper. The main purpose is to propagate some pointwise in time space-averaged exponential bounds at a regularity level to higher regularity levels.

\begin{lemma}[Parabolic estimates] \la{para:lem:para}

Let $\mu\geq 0$ and $I=(t_0,t_1)$, with $-\infty \leq t_0\leq t_0+1<t_1<+\infty$. There exists $C>0$ such that for $0<\delta \lesssim \text{min}(1,e^{-\mu t})$, for any $u$ solution of \fref{eq:NLH} on $(t_0,t_1)$  of the form:
\be \la{para:eq:id u}
u(t)=Q+v
\ee
satisfying for some $q\geq \frac{2d}{d-4}$ for any $t\in I$
\be \la{para:eq:bd v Lq}
\para v \para_{L^q} \leq \delta e^{\mu t}
\ee
there holds for any\footnote{With the convention that $-\infty +\tilde t=-\infty$.} $t\in (t_0+\tilde t,t_1)$:
\be \la{para:eq:bd v W2infty}
\para v \para_{W^{2,\infty}(\mathbb R^d)}+\para \partial_t v \para_{L^{\infty}(\mathbb R^d)} \leq C \delta e^{\mu t},
\ee
and if $\mu>0$ or $t_0\neq -\infty$ for $C'(t_1)>0$:
\be \la{para:eq:bd v holder}
\para \nabla^2 v \para_{C^{0,\frac 1 4}((t_0+\tilde t,t_1)\times \mathbb R^d)}+\para \partial_t v \para_{C^{0,\frac 1 4}((t_0+\tilde t,t_1)\times \mathbb R^d)} \leq C'.
\ee
where $C^{0,\frac 1 4}$ denotes the H\"older $\frac 1 4$-norm.

\end{lemma}

The first step to prove Lemma \ref{para:lem:para} is to obtain the $L^{\infty}$ bound, which is the purpose of the following lemma.

\begin{lemma}[Parabolic bootstrap] \la{para:lem:expo}

There exists $\nu>0$ such that the following holds. Let $I=(t_0,t_1)$, with $-\infty\leq t_0\leq t_0+1<t_1<+\infty$ and $\mu\geq 0$. For any $0<\tilde t<1$ there exists $\delta^*=\delta (t_1,\tilde t) >0$ and $C=C(\tilde t)>0$ such that for any $0<\delta<\delta^*$ and $u$ solution of \fref{eq:NLH} of the form:
\be \la{para:eq:id u 2}
u(t)=Q+v
\ee
satisfying for some $q\geq \frac{2d}{d-4}$ for any $t\in I$:
\be \la{para:eq:bd v}
\para v \para_{L^q} \leq \delta e^{\mu t}
\ee
there holds for any\footnote{With the convention that $-\infty +\tilde t=-\infty$.} $t\in (t_0+\tilde t,t_1)$:
\begin{itemize}
\item[(i)] If $q \leq pd$,
\be \la{para:eq:bootstrap v}
\para v \para_{L^{(1+\nu)q}} \leq C \delta e^{\mu t}.
\ee
\item[(i)] If $q> pd$,
\be \la{para:eq:bootstrap v Linfty}
\para v \para_{L^{\infty}} \leq \delta e^{\mu t}.
\ee

\end{itemize}

\end{lemma}

\begin{proof}[Proof of Lemma \ref{para:lem:expo}]

We take $\nu=\frac{1}{d^2}$ and $\delta^*=\text{min}(1,e^{\mu t_1})$. From \fref{para:eq:id u 2} and \fref{eq:NLH} the evolution of $v$ is given by:
\be \la{para:id vt}
v_t=\Delta v -Vv+NL, \ \ \ NL:=f(Q+v)-f(Q)-vf'(Q)
\ee
which implies using Duhamel formula that for any $t\in (t_0+\tilde t,t_1)$
\be \la{para:id v}
v(t)=K_{\tilde t} *v(t-\tilde t)+\int_0^{\tilde t} K_{\tilde t-s} *[-Vv(t-\tilde t+s)+NL(t-\tilde t+s)]ds .
\ee
We now estimate each term in the above identity. \\

\noindent \textbf{step 1} Case $\frac{2d}{d-4}\leq q \leq pd$. Using Young inequality for convolution, from \fref{para:eq:bd v} and \fref{eq:bd Kt} for the free evolution term:
$$
\para K_{\tilde t} *v(t-\tilde t) \para_{L^{(1+\nu)q}}\leq C(\tilde t) \delta e^{\mu (t-\tilde t)} \leq C(\tilde t) \delta e^{\mu t}.
$$
For the linear force term associated to the potential, using \fref{para:eq:bd v}, Young inequality and \fref{eq:bd Kt}:
$$
\ba{r c l}
&\para \int_0^{\tilde t} K_{\tilde t-s} *(-Vv(t-\tilde t+s))ds \para_{L^{(1+\nu)q}} \\
\leq &  \int_0^{\tilde t} \para K_{\tilde t-s}\para_{L^{\frac{1}{1+\frac{1}{(1+\nu)q}-\frac 1 q}}} \para v(t-\tilde t+s))ds \para_{L^q} \para V\para_{L^{\infty}}ds \\
\leq & C \int_0^{\tilde t} \frac{1}{|\tilde t-s|^{\frac{d}{2q}\frac{\nu}{1+\nu}}} \delta e^{\mu(t-\tilde t+s)}ds\leq C\delta e^{\mu t}
\ea
$$
because $\frac{d}{2q}\frac{\nu}{1+\nu}<\frac 1 4$ as $\nu=\frac{1}{d^2}$ and $q\geq \frac{2d}{d-4}$, and $0<\tilde t\leq 1$. For the nonlinear term, using  \fref{para:eq:bd v}, Young inequality, \fref{eq:bd Kt} and \fref{eq:NL1}:
$$
\ba{r c l}
&\para \int_0^{\tilde t} K_{\tilde t-s} *(NL(t-\tilde t+s))ds \para_{L^{(1+\nu)q}} \\
\leq & C \int_0^{\tilde t} \para K_{\tilde t-s}\para_{L^{\frac{1}{1+\frac{1}{(1+\nu)q}-\frac p q}}} \para NL(t-\tilde t+s))ds \para_{L^{\frac q p}} ds \\
\leq & C \int_0^{\tilde t} \para K_{\tilde t-s}\para_{L^{\frac{1}{1+\frac{1}{(1+\nu)q}-\frac p q}}} \para v(t-\tilde t+s))ds \para_{L^q}^p ds \\
\leq & C \int_0^{\tilde t} \frac{1}{|\tilde t-s|^{\frac{d}{2q}\frac{p-1+p \nu}{1+\nu}}} \delta^p e^{p\mu(t-\tilde t+s)}ds \\
\leq & C \int_0^{\tilde t} \frac{1}{|\tilde t-s|^{\frac{d}{2q}\frac{p-1+p \nu}{1+\nu}}} \delta^p e^{p\mu(t-\tilde t+s)}ds\leq C\delta e^{\mu t}
\ea
$$
from the fact that $\frac{d}{2q}\frac{p-1+p\nu}{1+\nu}\leq \kappa (d)<1$ as $\nu=\frac{1}{d^2}$ and $q\geq \frac{2d}{d-4}$, and $0<\delta\leq 1$. The three previous estimates then imply \fref{para:eq:bootstrap v}.\\

\noindent{step 2} Case $q>pd $. Using H\"older inequality, \fref{para:eq:bd v} and \fref{eq:bd Kt} for the first term:
$$
\para K_{\tilde t} *v(t-\tilde t) \para_{L^{\infty}}\leq C(\tilde t) \delta e^{\mu (t-\tilde t)} \leq C(\tilde t) \delta e^{\mu t}.
$$
For the second, using \fref{para:eq:bd v}, H\"older inequality and \fref{eq:bd Kt}:
$$
\ba{r c l}
\para \int_0^{\tilde t} K_{\tilde t-s} *(-Vv(t-\tilde t+s))ds \para_{L^{\infty}} & \leq & C \int_0^{\tilde t} \para K_{\tilde t-s}\para_{L^{q'}} \para v(t-\tilde t+s))ds \para_{L^q}ds \\
&\leq & C \int_0^{\tilde t} \frac{1}{|\tilde t-s|^{\frac{d}{2q}}} \delta e^{\mu(t-\tilde t+s)}ds\leq C \delta e^{\mu t}
\ea
$$
from the fact that $\frac{d}{2q}<\frac 1 2$. For the nonlinear term, using again \fref{para:eq:bd v}, H\"older inequality, \fref{eq:bd Kt} and \fref{eq:NL1}:
$$
\ba{r c l}
\para \int_0^{\tilde t} K_{\tilde t-s} *(NL(t-\tilde t+s))ds \para_{L^{\infty}} & \leq & C \int_0^{\tilde t} \para K_{\tilde t-s}\para_{L^{\left( \frac{q}{p}\right)'}} \para v(t-\tilde t+s))ds \para_{L^q}^p ds \\
&\leq & C \int_0^{\tilde t} \frac{1}{|\tilde t-s|^{\frac{dp}{2q}}} \delta^p e^{p\mu(t-\tilde t+s)}ds \\
&\leq & C \int_0^{\tilde t} \frac{1}{|\tilde t-s|^{\frac{dp}{2q}}} \delta^p e^{p\mu(t-\tilde t+s)}ds\leq C(\tilde t)\delta e^{\mu t}
\ea
$$
from the fact that $\frac{dp}{2q}<\frac 1 2$. The three above estimates give \fref{para:eq:bootstrap v Linfty}.\\

\end{proof}

We can now end the proof of Lemma \ref{para:lem:para}.

\begin{proof}[Proof of Lemma \ref{para:lem:para}]

First, iterating several times Lemma \ref{para:lem:expo} one obtains that for any $t\in(t_0+\frac{\tilde t}{4},t_1)$:
\be \la{para:bd v Linfty}
\para v \para_{L^{\infty}}\leq C\delta e^{\mu t}.
\ee
We define for $t\in (t_0,t_1)$:
\be \la{para:id F}
F(t)=-Vv(t)+NL(t), \ \ \ NL:=f(Q+v)-f(Q)-vf'(Q)
\ee
so that $v$ solves $v_t=\Delta v+F(t)$ and the Duhamel formula writes:
\be \la{para:id v 2}
v(t+t')=K_{t'}*v(t)+\int_{0}^{t'} K_{t'-s}*F(t+s)ds.
\ee

\noindent \textbf{step 1} Proof of the $W^{1,\infty}$ bound. From \fref{para:bd v Linfty}, \fref{para:id F}, \fref{eq:NL1} one has that for $t\in ]t_0+\frac{\tilde t}{4},t_1)$:
$$
\para F \para_{L^{\infty}} \leq C(\delta e^{\mu t}+\delta^p e^{\mu t})\leq C\delta e^{\mu t}.
$$
Using this, \fref{para:bd v Linfty}, \fref{para:id v 2}, \fref{eq:NL1}, \fref{eq:bd Kt} and H\"older inequality one has that for $t\in (t_0+\frac{\tilde t}{2},t_1)$:
$$
\ba{r c l}
\para \nabla v(t)\para_{L^{\infty}} & \leq & \para \nabla K_{\frac{\tilde t}{4}} *v(t-\frac{\tilde t}{4})\para_{L^{\infty}}+\int_0^{\frac{\tilde t}{4}} \para \nabla K_{\frac{\tilde t}{4}-s}\para_{L^1} \para F(t-\frac{\tilde t}{4}+s)\para_{L^{\infty}}ds \\
&\leq & C(\tilde t)e^{\mu t}+\int_0^{\frac{\tilde t}{4}} C\frac{C(\delta e^{\mu t}+\delta^pe^{p\mu t})}{|\frac{\tilde t}{4}-s|^\frac 1 2} \leq C(\tilde t)\delta e^{\mu t}.
\ea
$$
which with \fref{para:bd v Linfty} means that on $t\in (t_0+\frac{\tilde t}{2},t_1)$
\be \la{para:bd v W1infty}
\para v \para_{W^{1,\infty}}\leq C\delta e^{\mu t}.
\ee

\noindent \textbf{step 2} Proof of the $W^{2,\infty}$ bound. From \fref{para:id F} one has:
$$
\nabla (NL) = p\left( |Q+v|^{p-1}-Q^{p-1} \right)\nabla v+p\left(|Q+v|^{p-1}-Q^{p-1}-(p-1)Q^{p-2}v \right)\nabla Q \\
$$
This, \fref{para:bd v W1infty} and \fref{para:id F} then imply that for $t\in (t_0+\frac{\tilde t}{2},t_1)$:
\be \la{para:bd F Linfty}
\para \nabla F \para_{L^{\infty}} \leq C(\delta e^{\mu t}+\delta^p e^{\mu t})\leq C\delta e^{\mu t}.
\ee
One then computes from \fref{para:id v 2}, \fref{para:bd v W1infty}, H\"older inequality, \fref{eq:bd Kt} and the above estimate that for $t\in (t_0+\frac{3\tilde t}{4},t_1)$ and $1\leq i,j\leq d$:
$$
\ba{r c l}
\para \partial_{x_ix_j} v(t)\para_{L^{\infty}} &= &\para \partial_{x_i} K_{\frac{\tilde t}{4}} *\partial_{x_j}v(t-\frac{\tilde t}{4})\para_{L^{\infty}}+\int_0^{\frac{\tilde t}{4}} \para \partial_{x_i} K_{\frac{\tilde t}{4}-s}\para_{L^1} \para \partial_{x_j}F(t-\frac{\tilde t}{4}+s)\para_{L^{\infty}}ds \\
&\leq & C(\tilde t)\delta e^{\mu t}+\int_0^{\frac{\tilde t}{4}} \frac{C(\delta e^{\mu (t-\frac{\tilde t}{4}+s)}+\delta^pe^{p\mu(t-\frac{\tilde t}{4}+s)})}{|\frac{\tilde t}{4}-s|^{\frac 1 2}}\leq C(\tilde t)e^{\mu t}
\ea
$$
which with \fref{para:bd v W1infty} means that on $t\in (t_0+\frac{3\tilde t}{4},t_1)$
\be \la{para:bd v W2infty 2}
\para v \para_{W^{2,\infty}}\leq C\delta e^{\mu t}.
\ee

\noindent \textbf{step 3} Proof of the H\"older bound. For $t$ fixed in $(t_0+\tilde t,t_1)$ and $1\leq i,j\leq d$, using \fref{para:bd F Linfty} and one computes that:
$$
\ba{r c l}
& \parallel \partial_{x_ix_j} v \parallel_{C^{0,\frac 1 4}(\{t\}\times \mathbb R)} \\
 \leq & \underset{y,x\in \mathbb R^d}{\text{sup}}|y|^{-\frac 1 4} \left| \int_{0}^{\frac{\tilde t}{4}}\int_{\mathbb R^d} (\nabla K_{\frac{\tilde t}{4}-s}(x+y+z)-\nabla K_{\frac{\tilde t}{4}-s}(x+z)) \nabla F(t-\frac{\tilde t}{4}+s)dz\right| \\
 &+\para \partial_{x_i} K_{\frac{\tilde t}{4}}*\partial_{x_j} v(t-\frac{\tilde t}{4})\para_{C^{0,\frac{1}{4}(\mathbb R^d)}}\\
\leq & \int_0^{\frac{\tilde t}{4}} \para \nabla F(t-\frac{\tilde t}{4}+s) \para_{L^{\infty}} \underset{y,x\in \mathbb R^d}{\text{sup}}|y|^{-\frac 1 4} \int_{\mathbb R^d} |\nabla K_{\frac{\tilde t}{4}-s}(y+z)-\nabla K_{\frac{\tilde t}{4}-s}(z)| dzds \\
&+C(\tilde t)\delta e^{\mu t}\\
\leq & \int_{0}^{\frac{\tilde t}{4}} \frac{C \delta e^{\mu t'}}{|t-t'|^{\frac 3 4}}dt'+C(\tilde t)\delta e^{\mu t}\leq C(\tilde t)\delta e^{\mu t} .
\ea
$$
Now fix $x\in \mathbb R^d$ and $t_0+\tilde t<t<t'<t_1$. We treat the case $\mu>0$. In the case $\mu=0$ and $t_0\neq +\infty$ being similar. Using \fref{para:bd F Linfty}, H\"older inequality, \fref{eq:bd Kt}, \fref{eq:bd holder Kt 2} and \fref{para:id v 2}:
$$
\ba{r c l}
&\frac{|\partial_{x_i x_j}v(t',x)-\partial_{x_i x_j}v(t,x)|}{|t'-t|^{\frac 1 4}}\\
\leq & \frac{1}{|t'-t|^{\frac 1 4}} \left|(\partial_{x_i}K_{t'-t+\frac{\tilde{t}}{4}}-\partial_{x_i}K_{\frac{\tilde{t}}{4}})*(\partial_{x_j}v(t-\frac{\tilde t}{4})) \right|\\
&+\frac{1}{|t'-t|^{\frac 1 4}} \left|\int_{t}^{t'} \partial_{x_i}K_{t'-s}*\partial_{x_j}F(s)ds \right|\\
&+\frac{1}{|t'-t|^{\frac 1 4}} \left|\int_{t-\frac{\tilde t}{4}}^t (\partial_{x_i}K_{t'-s}-\partial_{x_i}K_{t-s})*\partial_{x_j}F(s) ds\right|\\
\leq & \frac{1}{|t'-t|^{\frac 1 4}} \para \partial_{x_i}K_{t'-t+\frac{\tilde{t}}{4}}-\partial_{x_i}K_{\frac{\tilde{t}}{4}}\para_{L^1}\para \partial_{x_j}v(t-\frac{\tilde t}{4})\para_{L^{\infty}}\\
&+\frac{1}{|t'-t|^{\frac 1 4}} \int_{t}^{t'} \para \partial_{x_i}K_{t'-s}\para_{L^1}\para \partial_{x_j}F(s)\para_{L^{\infty}}ds \\
&+ \int_{t-\frac{\tilde t}{4}}^t \frac{1}{|t'-t|^{\frac 1 4}} \para \partial_{x_i}K_{t'-s}-\partial_{x_i}K_{t-s}\para_{L^1}\para \partial_{x_j}F(s)\para_{L^{\infty}} ds \\
\leq & \frac{C}{\tilde t^{\frac 3 4}} \delta e^{\mu t}+\frac{1}{|t'-t|^{\frac 1 4}} \int_{t}^{t'} \frac{C}{|t'-s|^{\frac{1}{2}}} \delta e^{\mu s} ds + \int_{t-\frac{\tilde t}{4}}^t \frac{1}{|t-s|^{\frac 3 4}} C\delta e^{\mu s}ds \\
\leq & \delta C(\tilde t,t_1) .
\ea
$$
The two above bounds imply that:
$$
\para \nabla^2 v \para_{C^{0,\frac 1 4}((t_0+\tilde t,t_1)\times \mathbb R^d)} \leq \delta C(\tilde t).
$$
The above bound and \fref{para:bd v W2infty 2}, as $v_t$ is related to $\nabla^2v$ via \fref{para:id vt}, implies the desired bound \fref{para:eq:bd v holder}.

\end{proof}

We end this section with some Strichartz type estimates for the heat equation. Such estimates are more often used in the context of dispersive equations, and for dissipative equations one can usually handle technical problems by means of other type of estimates. However in the present paper we study an energy critical equation in the energy space, and Strichartz type norms appear naturally as in \fref{eq:strichartz varepsilon}. We do not write here a proof of the following Lemma \ref{lem:strichartz}. A proof of these estimates can indeed be done using exactly the same techniques employed in the proof of Theorem 8.18 in \cite{Ba}. Let $d\geq 2$. We say that a couple of real numbers $(q,r)$ is admissible if they satisfy:
\be \label{admissibilite}
q,r\geq 2, \ (q,r,d)\neq (2,+\infty,2) \ \text{and} \  \frac{2}{q}+\frac d r =\frac d 2 .
\ee
For any exponent $p\geq 1$, we denote by $p'=\frac{p-1}{p}$ its Lebesgue conjugated exponent.

\begin{lemma}[Strichartz type estimates for solutions to the heat equation] \label{lem:strichartz}

Let $d\geq 2$ be an integer. The following inequalities hold.
\begin{itemize}
\item[(i)] \emph{The homogeneous case.} For any couple $(q,r)$ satisfying \fref{admissibilite} there exists constant $C=C(d,q)>0$ such that for any initial datum $u\in L^2(\mathbb R^d)$:
\be \label{strichartz homogene}
\parallel K_t*u \parallel_{L^q([0,+\infty),L^r(\mathbb R^d))}\leq C \parallel u \parallel_{L^2(\mathbb R^d)}.
\ee
\item[(ii)] \emph{The inhomogeneous case.} For any couples $(q_1,r_1)$, $(q_2,r_2)$ satisfying \fref{admissibilite} there exists a constant $C=C(d,q_1,q_2)$ such that for any source term $f\in L^{q_2'}([0,+\infty),L^{r_2'}(\mathbb R^d))$:
\be \label{strichartz inhomogene}
\left\Vert t\mapsto \int_0^t K_{t-t'}* f(t')dt' \right\Vert_{L^{q_1}([0,+\infty),L^{r_1}(\mathbb R^d))} \leq C \parallel f \parallel_{L^{q_2'}([0,+\infty),L^{r_2'}(\mathbb R^d))}.
\ee
\end{itemize}

\end{lemma}


\end{document}